\documentclass[12pt,oneside,a4paper,reqno]{amsart}

\usepackage{amssymb}
\usepackage{amsmath}
\usepackage{amsthm}
\usepackage{amscd}
\usepackage[all]{xy}
\usepackage{longtable}
\usepackage{mathrsfs}
\usepackage[dvipdfmx]{xcolor}
\usepackage[dvipdfmx]{pict2e}
\usepackage[dvipdfmx]{graphicx}

\usepackage{comment}
\usepackage{todonotes}

\usepackage[dvipdfmx]{hyperref}
\usepackage{hyperref}
\hypersetup{
	colorlinks=true,
	linkcolor=red,
	citecolor=blue}
	\usepackage[utf8]{inputenc}
\usepackage[T1]{fontenc}
\usepackage{color}

\usepackage{extarrows}
\usepackage{tikz}
\usetikzlibrary{cd}

\usepackage{geometry}
\theoremstyle{plain}
\newtheorem{theorem}{Theorem}[section]
\newtheorem{lemma}[theorem]{Lemma}
\newtheorem{corollary}[theorem]{Corollary}
\newtheorem{proposition}[theorem]{Proposition}
\newtheorem{remark}[theorem]{Remark}

\newtheorem{remark-theorem}[theorem]{Remark-Theorem}

\theoremstyle{definition}
\newtheorem{definition}[theorem]{Definition}

\newcommand{\kah}{K\"{a}hler }
\newcommand{\idd}{i\partial\overline{\partial}}
\newcommand{\dbar}{\overline{\partial}}

\newcommand{\cal}[1]{\mathcal{#1}}
\newcommand{\bb}[1]{\mathbb{#1}}
\newcommand{\scr}[1]{\mathscr{#1}}
\newcommand{\rom}[1]{\mathrm{#1}}

\newcommand{\xs}{X_{sing}}
\newcommand{\reg}{X_{reg}}
\newcommand{\ogr}{\omega_X^{GR}}
\newcommand{\tx}{\widetilde{X}}
\newcommand{\ts}{\widetilde{S}}
\newcommand{\tu}{\widetilde{U}}
\newcommand{\dvo}{dV_\omega}

\newcommand{\tl}[1]{\widetilde{#1}}

\newcommand{\exc}{\rom{Exc}}

\newcommand{\ureg}{U_{reg}}
\newcommand{\sreg}{S_{reg}}

\newcommand{\lara}[2]{\langle{#1},{#2}\rangle}

\newcommand{\iO}[1]{i\Theta_{#1}}

\newcommand{\oid}[1]{\otimes\mathrm{id}_{#1}}


\begin{document}
\title
[singular positivity and vanishing theorems on complex spaces] 
{Griffiths and Nakano positivity \\ and Nakano-Nadel vanishing theorems \\ for singular Hermitian metrics on complex spaces}
\author{Yuta Watanabe}
\address{Department of Mathematics, Faculty of Science and Engineering, Chuo University.
1-13-27 Kasuga, Bunkyo-ku, Tokyo 112-8551, Japan}
\email{{\tt wyuta.math@gmail.com}, {\tt wyuta@math.chuo-u.ac.jp}}

\begin{abstract}
    In this paper, in order to develop a more general $L^2$-theory for the $\overline{\partial}$-operator on complex spaces, we introduce appropriate notions of singular Griffiths/Nakano positivity on complex spaces and establish various properties of these notions. 
    By applying these results, we provide $L^2$-Dolbeault fine resolutions and cohomological isomorphisms, and $L^2$-existence theorems. As an application, we obtain Nakano-Nadel vanishing theorems on weakly pseudoconvex complex spaces.
\end{abstract}


\maketitle

\vspace{-5mm}

\setcounter{tocdepth}{1}
\tableofcontents

\vspace{-7mm}

\section{Introduction}

The $L^2$-theory for the $\dbar$-operator has become a fundamental and indispensable area in complex analysis through H\"{o}rmander's pioneering work on $L^2$-estimates and existence theorems (see \cite{Hor65,Hor90}) and the related work of Andreotti-Vesentini (see \cite{AV65}). 
Subsequently, Ohsawa and Takegoshi introduced the $L^2$-extension theorem (see \cite{OT87}), and, more recently, complex geometry has undergone further development through applications combining this theory with the study of singular Hermitian metrics on holomorphic vector bundles and their positivity properties (see \cite{Dem90,BP08,Dem12,Rau15,PT18,HPS18,Ina22,DNWZ23,LXYZ24,Wat26a,IMW26}).
While this theory is well developed on complex manifolds, various difficulties remain in constructing an appropriate $L^2$-theory on complex spaces. 
In the case of line bundles, several results are available (see \cite{Rup14,SZ23b,SZ24,Wat26b}), whereas many problems remain open for holomorphic vector bundles. 
Among the central issues are the formulation of suitable notions of positivity, the construction of $L^2$-Dolbeault fine resolutions and cohomological isomorphisms, and the establishment of global $L^2$-existence theorems.

In this paper, appropriate notions of Griffiths/Nakano positivity are introduced for singular Hermitian metrics on holomorphic vector bundles over complex spaces, and various properties of these notions are established. 
By combining these results with the strong openness property and related techniques, solutions to the fundamental problems mentioned above are obtained. As an application, Nakano--Nadel vanishing theorems are established for weakly pseudoconvex complex spaces.

Throughout this paper, all complex spaces are assumed to be reduced, and $X$ denotes a complex space of pure dimension, while $E\longrightarrow X$ is a holomorphic vector bundle and $h$ is a singular Hermitian metric on $E$ (see Definition \ref{Definition: singular Hermitian metrics on E}).
Furthermore, for a resolution of singularities $\pi:\tx\longrightarrow X$, let $\exc:=\pi^{-1}(\xs)$ denote the $\pi$-exceptional divisor, which has simple normal crossing.
When $X$ is not necessarily compact, the resolusion $\pi:\tx\longrightarrow X$ is always taken to be the canonical desingularization in Theorem \ref{Theorem: canonical desingularization}. 

On complex manifolds, singular Griffiths/Nakano positivity has been extensively studied recently (see \cite{BP08,Rau15,PT18,Ina22,DNWZ23,LXYZ24,Wat24a,Wat25a,Wat26a,IMW26}).
Following the case of complex manifolds, we first introduce appropriate notions of Griffiths/Nakano positivity on complex spaces (see Definitions \ref{Definition: singular Griffiths positivity} and \ref{Definition: singular Nakano positive}), 
and establish various properties of these notions, including their behavior under pull-backs by resolutions of singularities, in order to treat them effectively.

Since Griffiths semi-positivity is locally characterized by plurisubharmonic functions, the equivalence with its pull-back requires an extension condition for plurisubharmonic functions (see Theorem \ref{Theorem: idd>0 and a.e. psh}). 
In the case of Griffiths positivity, additional assumptions are needed since it only provides information in the sense of currents. In any case, these notions are simply equivalent to their pull-backs when $X$ is normal.

\begin{theorem}[{=\,Theorem \ref{Theorem: Grif semi-positivity of h and pi^*h}}]
    Let $\pi:\tx\longrightarrow X$ be a canonical desingularization in Theorem \ref{Theorem: canonical desingularization}. 
    If $h$ is Griffiths semi-positive on $X$, then the pulled-back $\pi^*h$ on $\pi^*E$ is also Griffiths semi-positive on $\tx$. 
    Conversely, if $X$ is locally irreducible and $\pi^* h$ is Griffiths semi-positive on $\tx$, then $h$ is also Griffiths semi-positive on $X$.
\end{theorem}

\begin{theorem}[{=\,Theorem \ref{Theorem: h Gri iff pi*h Gri}}]
    Let $\omega$ be a Hermitian metric on $X$ and $\varepsilon:X\longrightarrow\bb{R}_{\geq0}$ be a smooth semi-positive function.
    If $\pi^*h$ on $\pi^*E$ is $\pi^*(\varepsilon\omega)$-Griffiths positive on $\tx$, then $h$ on $E$ is $\varepsilon\omega$-Griffiths positive on $X$. 
    Conversely, if $X$ is normal, or if $X$ is locally irreducible and $h^*$ is locally bounded from above a neighborhood of $\xs$, and if $h$ is $\varepsilon\omega$-Griffiths positive on $X$, then $\pi^*h$ on $\pi^*E$ is $\pi^*(\varepsilon\omega)$-Griffiths positive on $\tx$.
\end{theorem}

Unlike Griffiths positivity, Nakano positivity is defined through $L^2$-estimates. 
Therefore, it is insensitive to codimension-one phenomena, including exceptional divisors, and hence admits the following equivalence with its pull-back.

\begin{theorem}[{=\,Theorem \ref{Theorem: h Nak iff pi*h Nak}}]\label{Theorem: h Nak iff pi*h Nak in Introduction}
    Let $\theta$ be a continuous $(1,1)$-form on $X$. 
    It follows that $h$ is $\theta$-Nakano positive on $X$ in the sense of $L^2$-estimates if and only if the pulled-back $\pi^*h$ on $\pi^*E$ is $\pi^*\theta$-Nakano positive on $\tx$ in the sense of $L^2$-estimates.
\end{theorem}

By applying these theorems and reducing the problem to the case of complex manifolds \cite{Ina22}, we obtain the following extension to complex spaces.

\begin{theorem}[{=\,Theorem\,\ref{Theorem: h>Grif then h det h>Nak}}]\label{Theorem: h>Grif then h det h>Nak in Introduction}
    Let $X$ be a complex space. 
    If $h$ is Griffiths positive and uniformly positive definite (resp. a.e.\! Griffiths semi-positive) on $X$, then $h\otimes\det h$ on $E\otimes\det E$ is Nakano positive (resp. Nakano semi-positive) on $X$.
\end{theorem}

Here, Griffiths semi-positivity implies a.e.\! Griffiths semi-positivity, which is equivalent to singular semi-positivity when the bundle is a line bundle. 
We introduce the following theorem, which plays a key role in removing the \kah assumption from Nakano-Nadel vanishing in the (relatively) compact case (= Theorem \ref{Theorem: Nakano-Nadel vanishing on relatively cpt w.p.c. cpx sp}) and in proving the vanishing of higher direct image sheaves (= Theorem \ref{Theorem: higher direct image vanishing}). 
By combining the Negativity Lemma (see \cite{Hir64}, \cite[Lemma 2.2]{Wat25b}) with the strong openness property (= Theorem \ref{Theorem: strong openness property}, see \cite{GZ15,LXYZ24}), the $L^2$-subsheaf can be preserved on the resolution space $\tx$ while compensating for the positivity that degenerates along $\exc$.

\begin{theorem}[{=\,Theorem \ref{Theorem: h>Nak then pi*he-psi>Nak & E(pi*h)=E(pi*he-psi)}}]\label{Theorem: h>Nak then pi*he-psi>Nak & E(pi*h)=E(pi*he-psi) in Introduction}
    Let $V$ be a relatively compact open subset of $X$, and set $\tl{V}:=\pi^{-1}(V)$.
    If $h$ is Nakano positive (resp. Griffiths positive and a.e.\! Griffiths semi-positive) on an open neighborhood of $\overline{V}$, then there exist a quasi-plurisubharmonic function $\psi:\tl{V}\longrightarrow[-\infty,+\infty)$, which is smooth on $\tl{V}\setminus\exc$,
    and a small number $\varepsilon_V>0$ such that the singular Hermitian metric $\pi^*he^{-\varepsilon\psi}$ on $\pi^*E$ is also Nakano positive (resp. Griffiths positive) on $\tl{V}$ for any $0<\varepsilon<\varepsilon_V$. 
    
    Furthermore, by additionally assuming that $h$ is a.e.\! Griffiths semi-positive, there exists a small number $\varepsilon_0$ with $0<\varepsilon_0<\varepsilon_V$ such that we have $\scr{E}(\pi^*h)=\scr{E}(\pi^*he^{-\varepsilon\psi})\, ($resp. $\scr{E}(\pi^*h\otimes\det \pi^*h)=\scr{E}(\pi^*h\otimes\det \pi^*he^{-(1+r)\varepsilon\psi}))$ on $\tl{V}$ for any $0<\varepsilon\leq\varepsilon_0$.
\end{theorem}

In order to construct $L^2$-Dolbeault resolutions on complex spaces, we introduce several notions and concepts.
We say that $h$ is \textit{locally Nakano} (resp. \textit{Griffiths}) \textit{bounded from below on} $X$ if for every point $x\in X$, 
there exists a smooth strictly plurisubharmonic function $\psi_U$ on an open Stein neighborhood $U$ of $x$ such that the metric $he^{-C_U\psi_U}$ is Nakano (resp. a.e.\! Griffiths) semi-positive on $U$, for a sufficiently large constant $C_U>0$.
In particular, when both conditions are satisfied, it is said to be \textit{locally Griffiths-Nakano bounded from below on} $X$.
Clearly, a smooth Hermitian metric satisfies these condition, and in the case of a line bundle, the condition is equivalent to any local weight function of $h$ begin quasi-plurisubharmonic.

The \textit{Grauert}-\textit{Riemenschneider canonical sheaf} $\omega_X^{GR}$ on $X$ is defined by $\Gamma(U,\omega_X^{GR}):=\{s\in\Gamma(U\cap\reg,\omega_{\reg})\mid i^{n^2}s\wedge\overline{s}\in L^1_{loc}(U)\}$ for any open subset $U\subset X$. 
For a singular Hermitian metric $h$ on $E$, the \textit{Grauert}-\textit{Riemenschneider canonical} $L^2$-\textit{subsheaf} $\ogr(E,h)$, which is an $\cal{O}_X$-subsheaf of $\ogr\otimes\cal{O}_X(E)$, is defined by 
\begin{align*}
    \Gamma(U,\omega_X^{GR}(E,h)):=\biggl\{s\in\Gamma(U\cap\reg,\omega_{\reg}\otimes\cal{O}_X(E))\biggm| \int_{U_{reg}}|s|^2_{h,\omega}\,dV_\omega<+\infty\biggr\}
\end{align*}
for any relatively compact open subset $U\Subset X$ and a Hermitian metric $\omega$ on $X$ (see Definition \ref{Definition: GR canonical L2-subsheaf}).  
This definition is independent of the choice of $\omega$.
Arguing as in \cite[Proposition 5.8]{Dem12}, and using the $\dbar$-extension Lemma \ref{Lemma: dbar-extension lemma} together with Lemma \ref{Lemma: inequality of (n,0)-forms}, the $\cal{O}_X$-sheaf $\ogr(E,h)$ satisfies the following functoriality property.

\begin{proposition}\label{Proposition: functoriality property}
    For any resolution of singularities $\pi:\tx\longrightarrow X$, we have 
    \begin{center}
        $\ogr(E,h)=\pi_*(K_{\tx}\otimes\scr{E}(\pi^*h))$.
    \end{center}

    In particular, if $X$ is smooth, then $\ogr(E,h)=K_X\otimes\scr{E}(h)$ by definition and $\pi:\tx\longrightarrow X$ can be taken to be any proper modification. 
\end{proposition}

The coherence of this canonical $L^2$-subsheaf $\ogr(E,h)$ is obtained by Proposition \ref{Proposition: functoriality property} and Theorems \ref{Theorem: h Nak iff pi*h Nak in Introduction} and \ref{Theorem: h>Grif then h det h>Nak in Introduction}, together with the coherence of $\scr{E}(\pi^*h)$ on $\tx$ \cite{Ina22} and Grauert's direct image theorem.

\begin{theorem}[{=\,Theorem \ref{Theorem: coherence of ogr(E,h) if Nakano} and Corollary \ref{Corollary: coherence of ogr(E,h) if a.e. Griffiths}}]
    If $h$ is locally Nakano bounded from below on $X$, then $\ogr(E,h)$ is coherent on $X$. 
    If $h$ is locally Griffiths bounded from below on $X$, then $\ogr(E\otimes\det E,h\otimes\det h)$ is coherent on $X$. 
\end{theorem}

We define the subsheaf $\scr{L}^{p,q}_{E,h}$ on a complex space $X$ to be the sheaf of germs of $(p,q)$-forms $u$ with values in $E$ and with measurable coefficients such that both $|u|^2_h$ and $|\dbar u|^2_h$ are locally integrable on the regular locus (see Section \ref{Section: dbar-operator and L2-Dolbeault complexes}).  
Here, $\dbar$-operator is considered in the sense of distribution on $\reg$. 
Using local $L^2$-estimates (= Theorem \ref{Theorem: global L2-estimates of Nakano}), we establish the following $L^2$-Dolbeault fine resolution of $\ogr(E,h)$.

\begin{theorem}\label{Theorem: fine Dolbeault resolution}
    Let $X$ be a complex space of pure dimension $n$ and $E\longrightarrow X$ be a holomorphic vector bundle with a singular Hermitian metric $h$.
    If $h$ is locally Nakano bounded from below on $X$, then the $L^2$-Dolbeault complex 
    \begin{align*}
        0\longrightarrow \ogr(E,h)\hookrightarrow\scr{L}^{n,0}_{E,h}\overset{\dbar}{\longrightarrow}\scr{L}^{n,1}_{E,h}\overset{\dbar}{\longrightarrow}\scr{L}^{n,2}_{E,h}\overset{\dbar}{\longrightarrow}\cdots
    \end{align*}
    is exact; that is, the complex $(\scr{L}^{n,\ast}_{E,h},\dbar)$ is $L^2$-Dolbeault fine resolution of $\ogr(E,h)$.
    Thus, we have the following $L^2$-Dolbeault isomorphism 
    \begin{align*}
        H^q(X,\ogr(E,h))\cong H^q(\Gamma(X,\scr{L}^{n,\ast}_{E,h}))
    \end{align*}
    for any $q>0$.
\end{theorem}

The $L^2$-Dolbeault resolution is already known when the metric is smooth on $\reg$ or has line-bundle type singularities (see \cite{SZ23b,SZ24,Wat26b}, as recalled in Remark \ref{Remark: Generalization of L2-Dol resolusion}). Theorem \ref{Theorem: fine Dolbeault resolution} provides a generalization to more general singular Hermitian metrics.
The following vanishing theorem for higher direct image sheaves twisted by the $L^2$-subsheaf follows by reducing, using the crucial Theorem \ref{Theorem: h>Nak then pi*he-psi>Nak & E(pi*h)=E(pi*he-psi) in Introduction}, to the Nakano-Nadel vanishing theorem on weakly pseudoconvex manifolds proved in \cite{Wat25a}. 
As a consequence, the $L^2$-Dolbeault isomorphism for resolutions of singularities is obtained.

\begin{theorem}[{=\,Theorem \ref{Theorem: higher direct image vanishing in section 6}}]\label{Theorem: higher direct image vanishing}
    If $h$ is locally Griffiths-Nakano bounded from below on $X$, then we have the higher direct image vanishing 
    \begin{align*}
        R^q\pi_*(K_{\tx}\otimes\scr{E}(\pi^*h))=0 
    \end{align*}
    for any $q>0$, where $\pi:\tx\longrightarrow X$ is a resolusion of singularities.
\end{theorem}

\begin{theorem}\label{Theorem: L2-Dolbeault isomorphism}
    If $h$ is locally Griffiths-Nakano bounded from below on $X$, then the $L^2$-Dolbeault complex $(\pi_*\scr{L}^{n,*}_{\pi^*E,\pi^*h},\pi_*\dbar)$ is a fine resolution of $\ogr(E,h)\cong\pi_*(\omega_{\tx}\otimes\scr{E}(\pi^*h))$. 
    Thus, we have the following $L^2$-Dolbeault isomorphism 
    \begin{align*}
        H^q(X,\ogr(E,h))\cong H^q(\tx,K_{\tx}\otimes\scr{E}(\pi^*h)) 
    \end{align*}
    for any $q\geq0$. 
\end{theorem}

Subsequently, to establish various vanishing theorems, we present the following global $L^2$-existence Theorem \ref{Theorem: global L2-estimates of Nakano} on the regular locus of a weakly pseudoconvex complex space.
A crucial ingredient in its proof is the following theorem, which provides a certain Stein exhaustion sequence obtained as an application of Takayama's embedding \cite{Tak98} and a refined Demailly approximation (cf. \cite[Theorems 3.2 and 3.5]{Wat24b}).
Here, weakly pseudoconvexity includes both compact and Stein cases and forms the broadest and highly significant class among the complex-geometric objects that can be treated.
In fact, every complex Lie group is always weakly pseudoconvex (see \cite{Kaz73}).

\begin{theorem}\label{Theorem: exhaustion a.e. Stein covering}
    Let $X$ be a weakly pseudoconvex complex space. Assume that there exist a holomorphic line bundle $L\longrightarrow X$ and a singular Hermitian metric $h$ on $L$ such that $\iO{\pi^*L,\pi^*h}\geq\pi^*(\varepsilon\omega)$ holds on $\tx$ in the sense of currents for some smooth positive function $\varepsilon:X\longrightarrow\bb{R}_{>0}$ and some Hermitian metric $\omega$ on $X$, where $\pi:\tx\longrightarrow X$ is canonical desingularization of $X$. 
    Then there exist an increasing sequence of real numbers $\{c_j\}_{j\in\bb{N}}$ with $c_1>\inf_X\varPsi$ and analytic subsets $A_j$ of each $X_{c_j}$ such that each open subset $(X_{c_j})_{reg}\setminus A_j=(X_{c_j}\setminus A_j)\cap\reg$ is Stein manifold. 

    In particular, the assumption holds whenever $X$ admits a singular positive line bundle.
\end{theorem}

\begin{theorem}\label{Theorem: global L2-estimates of Nakano}
    Let $X$ be a weakly pseudoconvex \kah complex space of pure dimension $n$ with a \kah metric $\omega$ and $E\longrightarrow X$ be a holomorphic vector bundle with a singular Hermitian metric $h$. 
    Assume that $X$ admits a singular positive line bundle. 
    If $h$ is $\theta$-Nakano positive on $X$ in the sense of $L^2$-estimates for a continuous $(1,1)$-form $\theta$ on $X$,
    then for any $q>0$, for any smooth quasi-plurisubharmonic function $\psi$ on $X$ such that $\theta+\idd\psi>0$ in $X$ and for any $f\in L^{2,loc}_{n,q}(X,E,he^{-\psi})$ satisfying $\dbar f=0$ on $\reg$ and $\int_X\lara{B^{-1}_{\theta,\psi,\omega}f}{f}_{h,\omega}\,e^{-\psi}\dvo<+\infty$, 
    there exists $u\in L^2_{n,q-1}(X,E;\omega,he^{-\psi})$ satisfying $\dbar u=f$ on $\reg$ and the following global $L^2$-estimate
    \begin{align*}
        \int_X|u|^2_{h,\omega}e^{-\psi}\,\dvo\leq\int_X\lara{B^{-1}_{\theta,\psi,\omega}f}{f}_{h,\omega}\,e^{-\psi}\dvo,
    \end{align*}
    where $B_{\theta,\psi,\omega}=[(\theta+\idd\psi)\oid{E},\Lambda_\omega]$.
    
    Furthermore, if $X$ is not necessarily \kah and $\omega$ is taken simply as a Hermitian metric, then the same global $L^2$-estimate as above holds only in the case $q=1$.
\end{theorem}

Finally, combining the $L^2$-existence Theorem \ref{Theorem: global L2-estimates of Nakano} with the $L^2$-Dolbeault resolution Theorem \ref{Theorem: fine Dolbeault resolution}, we obtain a global Nakano-Nadel vanishing theorem on weakly pseudoconvex complex spaces. 
In the (relatively) compact case, the Kähler assumption can be removed (see Theorem \ref{Theorem: Nakano-Nadel vanishing on relatively cpt w.p.c. cpx sp}).

\begin{theorem}\label{Theorem: Nakano-Nadel vanishing on w.p.c. cpx sp}
    Let $X$ be a weakly pseudoconvex \kah complex space of pure dimension. 
    Assume that $X$ admits a singular positive line bundle. 
    If $h$ is Nakano positive on $X$ in the sense of $L^2$-estimates, 
    then we have the following vanishing 
    \begin{align*}
        H^q(X,\ogr(E,h))=0 
    \end{align*}
    for any $q>0$. 
    Furthermore, if $X$ is not necessarily K\"{a}hler, then we have only the first cohomology vanishing $H^1(X,\ogr(E,h))=0$.
\end{theorem}

The assumption on the existence of a singular positive line bundle may be replaced by a condition concerning the Griffiths positivity of $h$; see Remark \ref{Remark: condition of L2-estimates for Grif>0 and singular positive}.
A similar result for Griffiths positivity follows from Theorems \ref{Theorem: h>Grif then h det h>Nak in Introduction}, \ref{Theorem: global L2-estimates of Griffiths}, and \ref{Theorem: fine Dolbeault resolution, isomorphism and higher direct image vanishing of Griffiths}.

\begin{theorem}\label{Theorem: Griffiths vanishing on on w.p.c. cpx sp}
    Let $X$ be a weakly pseudoconvex \kah complex space of pure dimension. 
    If $h$ is Griffiths positive and a.e.\! Griffiths semi-positive on $X$, then we have 
    \begin{align*}
        H^q(X,\ogr(E\otimes\det E,h\otimes\det h))=0 
    \end{align*}
    for any $q>0$. 
    Furthermore, if $X$ is not necessarily K\"{a}hler, then we have only the first cohomology vanishing $H^1(X,\ogr(E\otimes\det E,h\otimes\det h))=0$.
\end{theorem}

\section{Preliminaries}

\subsection{Canonical desingularization of complex spaces}

Even in the non-compact case, a global resolution of singularities can be obtained, which is locally given by a finite sequence of blow-ups with smooth centers. 
This is achieved by patching together the resolutions of singularities constructed on relatively compact subsets.

\begin{theorem}[{\cite[Theorem\,13.3 and 13.4]{BM97}, cf.\,\cite{Hir64}}]\label{Theorem: canonical desingularization}
    Let $X$ be a complex space which is not necessarily compact or reduced. 
    There exists a desingularization $\pi:\tx\longrightarrow X$, which is a composite of a locally finite sequence of blow-ups, such that 
    \begin{itemize}
        \item the map $\pi:\tx\longrightarrow X$ is proper holomorphic.
        \item the set $\tx$ is smooth and the $\pi$-exceptional set $E:=\pi^{-1}(\xs)$ is simple normal crossing, where $E$ denotes the collection of all exceptional divisors.
        \item for any relatively compact open subset $V$ of $X$, the restriction $\pi|_V:\tx|_{\pi^{-1}(V)}\longrightarrow X|_V$ is a composite of a finite sequence of blow-ups with smooth centres. 
        \item the restriction $\pi|_{\tx\setminus E}:\tx\setminus E\longrightarrow X\setminus\xs=\reg$ is biholomorphic.
        \item $\pi$ is canonical in the sense that for any isomorphic $\varphi:X|_U\overset{\cong}{\longrightarrow} X|_V$, where $U$ and $V$ are open subsets of $X$, lifts to an isomorphism $\widetilde{\varphi}:\tx|_{\pi^{-1}(U)}\longrightarrow\tx|_{\pi^{-1}(V)}$.
    \end{itemize}
\end{theorem}

The desingularization in this theorem is referred to as the \textit{canonical} \textit{desingularization}.
We introduce the following key lemma, which compensates for the positivity obtained from the Negativity Lemma (see \cite[Lemma 2.2]{Wat25b}).

\begin{lemma}[{\cite[Lemma 3.2]{Wat25b}}]\label{Lemma: key lemma}
    Let $X$ be a complex space, $\pi:\widetilde{X}\longrightarrow X$ be a canonical desingularization and $\exc\!:=\!\pi^{-1}(\xs)$ be the exceptional set which is simple normal crossing. 
    Let $V$ be a relatively compact open subsets of $X$, and set $\widetilde{V}:=\pi^{-1}(V)$.
    Then there exists a quasi-plurisubharmonic function $\psi:\widetilde{V}\longrightarrow[-\infty,+\infty)$ which is smooth on $\widetilde{V}\setminus\exc$, and whose decomposition $\psi=\psi_{\rom{sm}}+\psi_{\rom{psh}}$ in a neighborhood of $\exc$ satisfies that the smooth part $\psi_{\rom{sm}}$ has a Levi form $\idd\psi_{\rom{sm}}$ compensating for the loss of positivity of the pullback $\pi^*\omega$ on $\tl{V}$, where $\pi^*\omega$ degenerates along $\exc$ for any Hermitian metric $\omega$ on $X$.
    Hence, by the relative compactness of $V$, there exists $\varepsilon_V>0$ such that $\pi^*\omega+\varepsilon\idd\psi$ is strictly positive on $\tl{V}$ in the sense of currents for any $0<\varepsilon<\varepsilon_V$; in other words, there exists a Hermitian metric $\gamma_\varepsilon$ on $\tx$ such that 
    \begin{align*}
        \pi^*\omega+\varepsilon\idd\psi\geq\gamma_\varepsilon
    \end{align*}
    on $\tl{V}$ in the sense of currents.
    In particular, we have $\psi\in L^1_{loc}(\widetilde{V})$.
\end{lemma}

\subsection{Plurisubharmonicity and singular Hermitian metrics of line bundles on complex spaces}

Let $X$ be a complex space. 

\begin{definition}[{\cite[Chapter\,V, Definition\,1.4]{GPR94}, \cite{FN80}}]
    A function $\varphi:X\longrightarrow[-\infty,+\infty)$ is called (resp. \textit{strictly}) \textit{plurisubharmonic} if for any $x\in X$ 
    there exist an open neighborhood $U$ admitting a closed holomorphic embedding $\iota_U:U\hookrightarrow V\subset\bb{C}^N$, here $V$ is an open subset of $\bb{C}^N$, 
    and a (resp. strictly) plurisubharmonic function $\widetilde{\varphi}:V\longrightarrow[-\infty,+\infty)$ such that $\varphi|_U=\widetilde{\varphi}\circ\iota_U$.
\end{definition}

The same requirement is imposed on smooth functions, measurable functions, differential forms, test forms, Hermitian metrics, and \kah metrics on $X$: 
locally, they are required to be expressible, via a closed holomorphic embedding $\iota_U:U\hookrightarrow V\subset\bb{C}^N$, as the restrictions of the corresponding smooth functions, 
and related objects on defined on $V$.
Moreover, a function is said to be \textit{quasi}-\textit{plurisubharmonic} if it can be written locally as the sum of a smooth function and a plurisubharmonic function.

Let $L\longrightarrow X$ be a holomorphic line bundle. 
For any trivialization $\tau:L|_U\overset{\cong}{\longrightarrow}U\times\bb{C}$, a (smooth) Hermitian metric $h$ on $L$ can be expressed as 
\begin{align*}
    ||\xi||_h=|\tau(\xi)|e^{-\varphi(x)}, \quad x\in U, \,\xi\in L_x
\end{align*}
using a function $\varphi$ on $U$. 
The function $\varphi:U\longrightarrow\bb{R}$ is called the \textit{weight function of} $h$ \textit{with respect to the trivialization} $\tau$.

\begin{definition}[{\cite[Definition 2.6]{Wat26b}}]\label{Definition: singular Hermitian metrics on cpx sp}
    Let $X$ be a complex space and $L\longrightarrow X$ be a holomorphic line bundle. 
    We say that $h$ is a singular Hermitian metric on $L$ if for any smooth Hermitian metric $h_0$ on $L$, there exists a locally integrable function $\varphi$ on $X$, i.e., $\varphi\in L^1_{loc}(X)$, such that $h=h_0e^{-2\varphi}$ on $X$.
\end{definition}

For any open subset $U\subset X$, we denote 
\begin{align*}
    L^1_{loc}(U):=\biggl\{u:U\longrightarrow\bb{C} \text{ measurable}\,\bigg|\,\int_{K_{reg}}|u|\,dV<+\infty, \forall\,K\Subset U\biggr\}
\end{align*}
and note that, in general, $L^1_{loc}(U)\subsetneq L^1_{loc}(\ureg)$.

\begin{definition}[{\cite[Definition 2.7]{Wat26b}}]\label{Definition: singular positivity on cpx sp}
    Let $X$ be a complex space. 
    We say that a singular Hermitian metric $h$ on $L$ is \textit{singular} \textit{positive} (resp. \textit{singular} \textit{semi}-\textit{positive}) if the weight function of $h$ with respect to any trivialization coincides with a strictly plurisubharmonic (resp. plurisubharmonic) function almost everywhere. 
\end{definition}

When $X$ is smooth, Definitions \ref{Definition: singular Hermitian metrics on cpx sp} and \ref{Definition: singular positivity on cpx sp} coincide with the existing definitions (see \cite[Chapter 3]{Dem12}, \cite{Wat25a,Wat25b}).

\subsection{Definition of singular Hermitian metrics of vector bundle on complex spaces}

Let $X$ be a complex space and $E\longrightarrow X$ be a holomorphic vector bundle. Then there exist an open covering $\{U_\alpha\}_{\alpha\in\Lambda}$ of $X$ and isomorphisms $\iota_\alpha:E|_{U_\alpha}\overset{\cong}{\longrightarrow}U_\alpha\times\bb{C}^r$. 
Conversely, the set of transition functions $\{g_{\alpha\beta}\}_{\alpha,\beta\in\Lambda}$, where $g_{\alpha\beta}:=\iota_\alpha\circ\iota_\beta^{-1}|_{U_\alpha\cap U_\beta}\in\cal{O}_X(U_\alpha\cap U_\beta,GL(r,\bb{C}))$, defines $E$. 
Such a $(\{g_{\alpha\beta}\},\{U_\alpha\})_{\Lambda}$ is called a \textit{system of transition functions of} $E$.

\begin{definition}\label{Definition: smooth Hermitian metrics on E}
    We say that $h$ is \textit{smooth Hermitian metric} on $E$ if for a system of transition functions $(\{g_{\alpha\beta}\},\{U_\alpha\})_{\Lambda}$ of $E$, there is a collection of smooth positive definite Hermitian matrics $h=\{h_\alpha:U_\alpha\longrightarrow M_r(\bb{C})\}_{\alpha\in\Lambda}$ such that $h_\beta={}^t\overline{g_{\alpha\beta}}h_\alpha g_{\alpha\beta}$ on $U_\alpha\cap U_\beta$.
\end{definition}

Here, the metric $h_\alpha$ is said to be smooth on $U_\alpha$ if for some trivialization $E|_{U_\alpha}\simeq U_\alpha\times\bb{C}^r$, equivalently, for some holomorphic frame $e_\alpha=(e_1^\alpha,\dots,e_r^\alpha)$ on $U_\alpha$, 
the metric $h_\alpha$ is represented as $h_\alpha=(h_{jk}^\alpha)_{1\le j,k\le r}, h_{jk}^\alpha:=\lara{e_j^\alpha}{e_k^\alpha}_{h_\alpha}$, and every component $h_{jk}^\alpha$ is smooth on $U_\alpha$.

\begin{definition}\label{Definition: singular Hermitian metrics on E}
    We say that $h$ is \textit{singular Hermitian metric} on $E$ if for a system of transition functions $(\{g_{\alpha\beta}\},\{U_\alpha\})_{\Lambda}$ of $E$, there is a collection of measurable semi-positive Hermitian forms $h=\{h_\alpha:U_\alpha\longrightarrow M_r(\bb{C})\}_{\alpha\in\Lambda}$ such that the compatibility condition $h_\beta={}^t\overline{g_{\alpha\beta}}h_\alpha g_{\alpha\beta}$ holds on $U_\alpha\cap U_\beta$, and $\det h_\alpha$ satisfies $0<\det h_\alpha<+\infty$ almost everywhere on $U_\alpha$. 
    In other words, $h$ is a measurable map from the base space $X$ to the space of semi-positive Hermitian forms on the fibers of $E$ satisfying $0<\det h<+\infty$ almost everywhere. 
\end{definition}

As above, $h_\alpha$ is said to be measurable if every component $h^\alpha_{jk}$ is measurable on $U_\alpha$ for some trivialization. 

\subsection{Currents and its positivity on complex space}

Let $X$ be a complex space of pure dimension $n$. 
As in the smooth case, the sheaf $\scr{D}'^{\,p,q}_X$ of currents of bidegree $(p,q)$ on $X$ is by definition the dual of $\scr{D}^{\,p,q}_X$ which is the space of smooth differential forms of bidegree $(p,q)$ with compact support. 
Given a local embedding $\iota:U\hookrightarrow \bb{C}^N$, thus, the currents $T\in\scr{D}'^{\,p,q}_X$ precisely correspond, via $T\longmapsto\iota_*T$, to the currents of bidegree $(N-n+p,N-n+q)$ in the ambient space that vanish on all test forms $\phi$ such that $\iota^*\phi=0$ on $\reg$ (see \cite{Dem85,AS12}).
Here, the current $\iota_*T$ is defined by 
\begin{align*}
    \lara{\iota_*T}{\tl{\phi}}:=\lara{T}{\iota^*\tl{\phi}}=\int_{U_{reg}} T\wedge\iota^*\tl{\phi},
\end{align*}
for any test forms $\tl{\phi}$ on $\bb{C}^N$. 
Furthermore, the positivity of the current $T\in\scr{D}'^{\,p,p}_X$ is also defined as in the smooth case, and this is equivalent to the positivity of $\iota_*T$. 
In other words, a current $T\in\scr{D}'^{\,p,p}_X$ is said to be \textit{positive}, i.e., $T\geq0$ \textit{on} $X$ \textit{in the sense of currents}, if $\lara{T}{\phi}\geq0$ for any strongly positive test form $\phi\in\scr{D}^{n-p,n-p}_X(X)$ on $X$, which can be written as $\phi=\rho\cdot i^{(n-p)^2}\eta\wedge\overline{\eta}$ suitable forms $\rho\in\scr{D}(X,\bb{R}_{\geq0})$ and $\eta\in\cal{E}_X^{n-p,0}(X)$.
Here, $\cal{E}_X^{p,q}$ denotes the sheaf of smooth differential forms. 
Note that, if $\varphi\in L^1_{loc}(U)$ then $[\varphi]\in\scr{D}_X'(U)$ and $\iota_*[\varphi]\in\scr{D}'(\bb{C}^N)$; however, in general $\iota_*\varphi$ is not locally integrable. 
Moreover, the form $\idd\varphi$ is defined on $U$ as a current by 
\begin{align*}
    \lara{\idd\varphi}{\phi}:=\int_{U_{reg}}\varphi\,\idd\phi,
\end{align*}
for any test form $\phi\in\scr{D}^{n-1,n-1}(U)$, i.e., $\idd\varphi\in\scr{D}_X'^{\,1,1}(U)$.

When $X$ is smooth, the following relationship between the positivity of currents and plurisubharmonicity is well known.

\begin{proposition}\label{Proposition: idd>0 iff a.e. psh}
    Let $\Omega$ be an open subset of $\bb{C}^n$ and $\varphi$ be a local integrable function on $\Omega$, i.e., $\varphi\in L^1_{loc}(\Omega)$.
    The inequality $\idd\varphi\geq0$ holds on $\Omega$ in the sense of currents if and only if there exists a plurisubharmonic function defined on $\Omega$ which coincides with $\varphi$ almost everywhere on $\Omega$. 
\end{proposition}

However, it should be noted that the same statement does not necessarily hold on complex spaces.

\begin{theorem}[{cf.\,\cite{GR56,FN80}, \cite[Theorem 1.10]{Dem85}}]\label{Theorem: idd>0 and a.e. psh}
    Let $U$ be an open subset of a complex space $X$ and $\varphi$ be a local integrable function on $U$, i.e., $\varphi\in L^1_{loc}(U)$.  
    \begin{itemize}
        \item if there exists a plurisubharmonic function defined on $U$ which coincides with $\varphi$ almost everywhere on $U$, then we have $\idd\varphi\geq0$ on $U$ in the sense of currents. 
        \item if the inequality $\idd\varphi\geq0$ holds on $U$ in the sense of currents and one of the following conditions is satisfied: 
        \begin{itemize}
            \item [$(a)$] $X$ is normal; 
            \item [$(b)$] $X$ is locally irreducible and $\varphi$ is locally bounded from above on $U$, 
        \end{itemize} 
        then there exists a plurisubharmonic function on $U$ that coincides with $\varphi$ almost everywhere on $U$.
    \end{itemize}
    Furthermore, if $\varphi$ is simply plurisubharmonic on $\ureg$, then $\varphi$ extends uniquely to a plurisubharmonic function on $U$ whenever either $(a)$ or $(b)$ is satisfied.
\end{theorem}

Normal complex spaces are locally irreducible. 
A weakly plurisubharmonic function is defined as an upper semi-continuous function whose pull-back by every holomorphic map from the unit disc $\Delta\subset\bb{C}$ is subharmonic. 
In \cite{GR56}, this notion was simply called plurisubharmonic, and \cite{FN80} showed that the two notions coincide.

\begin{remark}\label{Remark: locally irreducible in [GR56, Satz 3]}
    In \cite[Satz 3]{GR56}, which corresponds to case $(b)$ of Theorem \ref{Theorem: idd>0 and a.e. psh} on the extension of plurisubharmonic functions, no assumption of local irreducibility of $X$ is imposed. 
    However, this assumption is in fact necessary, as in \cite{Dem85}.
\end{remark}

Indeed, consider the non-locally irreducible space $X=\{zw=0\}\subset\bb{C}^2$. 
Define a locally integrable function $\varphi$ on $X$ by $\varphi=0$ on $L^{\times}_1:=\{z=0, w\ne 0\}$, and $\varphi=1$ on $L^{\times}_2:=\{w=0, z\ne 0\}$, here $\idd\varphi=0$ on $X$ in the sense of currents.  
If a plurisubharmonic extension $\tl{\varphi}$ of $\varphi$ exists on $X$, then $\tl{\varphi}|_{\reg}=\varphi$, and for the normalization $\nu:\tx=X_1\sqcup X_2\to X$, the pull-back $\nu^*\tl{\varphi}$ is also plurisubharmonic, where $X_1=\{z=0\}$ and $X_2=\{w=0\}$. 
However, we obtain $\nu^*\tl{\varphi}|_{X_j}=j-1$ and $\nu^*\tl{\varphi}(0_{X_j})=\tl{\varphi}(0_X)\,(j=1,2)$ which contradicts plurisubharmonicity on at least one of the branches.
The failure of extension is due to the fact that the $\limsup$ of $\varphi$ along each branch takes different values, i.e., $\limsup_{z\to 0,\,z\in L^{\times}_1}\varphi(z)\ne\limsup_{w\to 0,\,w\in L^{\times}_2}\varphi(w)$ (cf. \cite[Theorem 1.10]{Dem85}).

It can also be checked directly that (weakly) plurisubharmonicity fails, without appealing to the normalization. 
Since upper semicontinuity is required, we extend $\varphi$ by setting $\varphi(0):=\limsup_{x\to0,\,x\in X}\varphi(x)=1$. 
Consider a holomorphic curve $\gamma:\Delta=\{t\in\bb{C}\mid|t|<1\}\to X$ given by $\gamma(t)=(\gamma_1(t),\gamma_2(t))$. 
Since $\gamma_1(t)\gamma_2(t)=0$, holomorphicity implies that either $\gamma_1\equiv 0$ or $\gamma_2\equiv 0$. 
Hence $\gamma$ is entirely contained in one of the branches, and no holomorphic curve crosses the two branches.
Thus, for curves $\kappa_j$ contained in $L_j\,(j=1,2)$, we have $\kappa_2^*\varphi\equiv1$, while $\kappa_1^*\varphi$ is equal to $0$ away from the origin and takes the value $1$ only at the origin. 
Consequently, $\kappa_1^*\varphi$ is not subharmonic on $\Delta$.

\begin{remark}\label{Remark: unclear of idd pi*phi>0 current}
    Let $U$ be an open subset of a complex space $X$ with $U\cap\xs\ne\emptyset$, and let $\pi:\tx\longrightarrow X$ be a resolution of singularities. 
    Even if $\varphi\in L^1_{loc}(U)$, it is not necessarily true that $\pi^*\varphi\in L^1_{loc}(\pi^{-1}(U))$. 
    Hence, even if $\idd\varphi\geq0$ holds on $U$ in the sense of currents, it does not seem to be known in general whether $\idd\pi^*\varphi\geq0$ holds on $\pi^{-1}(U)$, unless conditions $(a)$ or $(b)$ of Theorem \ref{Theorem: idd>0 and a.e. psh} are satisfied.
\end{remark}

\subsection{$L^2$-subsheaves and the strong openness property}

In this subsection, we consider the smooth case, and let $X$ be a complex manifold and $E\longrightarrow X$ be a holomorphic vector bundle.
Let $\varphi$ be a quasi-plurisubharmonic function on $X$ and let $\scr{I}(\varphi)$ be the sheaf of germs of holomorphic functions $f$ such that $|f|^2e^{-\varphi}$ is locally integrable which is called the \textit{multiplier ideal sheaf}. 
For a singular Hermitian metric $h$ on a line bundle $L$ with a global weight $\varphi\in L^1_{loc}(X)$, i.e., $h=h_0e^{-2\varphi}$, we define the multiplier ideal sheaf of $h$ by $\scr{I}(h):=\scr{I}(\varphi)$. 
Its extension to vector bundles is given below.

\begin{definition}[{\cite[\!Definition\,2.3.1]{deC98}}]
    Let $h$ be a singular Hermitian metric on $E$. 
    We define the $L^2$-\textit{subsheaf} $\!\scr{E}(h)$ of germs of local holomorphic sections on $E$ as follows:
    \begin{align*}
        \scr{E}(h)_x:=\{s\in\cal{O}(E)_x\mid |s|^2_h \text{ is locally integrable around } x\}.
    \end{align*}
\end{definition}

In fact, $\scr{E}(h)=\cal{O}(E)\otimes\scr{I}(h)$ if $E$ is a holomorphic line bundle. 
In \cite{Nad90}, the coherence of $\scr{I}(\varphi)$ was proved using H\"{o}rmander's $L^2$-estimates. 
After that, \cite{Ina22} showed that $\scr{E}(h)$ is coherent if $h$ is Nakano semi-positive, and that $\scr{E}(h\otimes\det h)$ is coherent if $h$ is Griffiths semi-positive.

We introduce the strongly openness property for holomorphic vector bundles. This property was first conjectured in \cite{Dem01} for multiplier ideal sheaves, resolved in dimension two by Jonsson-Musta\c{t}\u{a} \cite{JM12}, and in arbitrary dimension by Guan-Zhou \cite{GZ15}; more recently, it has also been extended to holomorphic vector bundles.

\begin{theorem}[{Strong openness property, cf.\,\cite[Definition 1.2 and Theorem 1.5]{LXYZ24}}]\label{Theorem: strong openness property}
    Let $h$ be a singular Hermitian metric on $E$. 
    Let $\{h_j\}_{j\in\bb{N}}$ be a sequence of Griffiths semi-positive singular Hermitian metrics on $E$.
    If $h$ is Nakano semi-positive and Griffiths semi-positive, and $h_j\geq h$ for any $j$, and $-\log\det h_j$ converges to $-\log\det h$ almost everywhere, then $\scr{E}(h)=\bigcup_{j\in\bb{N}}\scr{E}(h_j)$.
\end{theorem}

\begin{corollary}[{cf.\,\cite[Definition 1.2 and Corollary 3.16]{LXYZ24}}]\label{Corollary: strong openness property}
    Let $h$ be a singular Hermitian metric on $E$ and $\psi:X\longrightarrow[-\infty,+\infty)$ be a plurisubharmonic function on $X$. 
    If $h$ is Nakano semi-positive and Griffiths semi-positive, then $\scr{E}(h)=\bigcup_{\beta>0}\scr{E}(he^{-\beta\psi})$. 
\end{corollary}

\subsection{$L^2$-spaces and Grauert-Riemenschneider-type canonical sheaves}

Let $X$ be a complex space of pure dimension $n$ and $E\longrightarrow X$ be a holomorphic line bundle with a singular Hermitian metric $h$. 
For any open subset $U\subset X$, we denote by 
\begin{align*}
    L^{2,loc}_{p,q}(U,E,h):=\Bigl\{u\in L^{2,loc}_{p,q}(\ureg,E,h)\,\Big|\, u|_{K_{reg}}\in L^2_{p,q}(K_{reg},E,h),\forall\,K\Subset U\Bigr\}
\end{align*} 
the space of $E$-valued $(p,q)$-forms on $U$ which are locally square integrable with respect to $h$. 
This space does not depend on the choice of a Hermitian metric on $U$; more precisely, for any Hermitian metric $\omega_U$ on $U$, we have $L^2_{p,q}(K_{reg},E,h)=L^2_{p,q}(K_{reg},E;\omega_U,h)$.
We further denote by 
\begin{align*}
    L^2_{p,q}(X,E;\omega,h)=&\,L^2_{p,q}(\reg,E;\omega,h)\\
    :=&\,\biggl\{u\in L^{2,loc}_{p,q}(X,E,h)\,\bigg|\,||u||^2_{h,\omega}:=\int_{\reg}|u|^2_{h,\omega}dV_\omega<+\infty\biggr\}
\end{align*}
the space of $E$-valued $(p,q)$-forms $u$ with measurable coefficients that are square integrable with respect to $h$ and a fixed smooth Hermitian metric $\omega$.
Unless otherwise stated, integrals over $\reg$ will be denoted simply by integrals over $X$.
Note that in general $L^2_{p,q}(X,L;\omega,h)\subsetneq L^{2,loc}_{p,q}(X,L,h)$, and these spaces coincide when $X$ is compact.

In general, for a locally integrable function $f\in L^1_{loc}$, its pull-back $\pi^*f$ need not be locally integrable, where $\pi:\tx\longrightarrow X$ is a resolusion of singularities. 
However, for $(n,n)$-forms, local integrability is equivalent to that of their pull-backs. Using this fact, the \textit{Grauert}-\textit{Riemenschneider canonical sheaf} $\ogr$ on $X$ is defined as the sheaf obtained by pushing forward holomorphic $n$-forms on $\tx$ to $X$, as follows.
\begin{align*}
    \Gamma(U,\ogr):=\bigl\{s\in\Gamma(U\cap\reg,\omega_{\reg})\bigm|\,i^{n^2}s\wedge\overline{s}\in L^1_{loc}(U)\bigr\}
\end{align*}
for any open subset $U\subset X$. 
Furthermore, it follows that $\ogr=\pi_*\omega_{\tx}$ for any resolusion of singularities $\pi:\tx\longrightarrow X$. 
Indeed, for any $s\in \ogr(U)$, the pull-back $\pi^*s$ belongs to $\omega_{\tx}\cap L^1_{loc}(\pi^{-1}(U)\setminus\exc)$, and the $\dbar$-extension Lemma \ref{Lemma: dbar-extension lemma} below implies that $\pi^*s\in\omega_{\tx}(\pi^{-1}(U))$. The converse is clear.

\begin{lemma}[{\cite[Lemma 6.9]{Dem82}}]\label{Lemma: dbar-extension lemma}
    Let $\Omega$ be an open subset of $\bb{C}^n$ and $Z$ be a closed analytic subset of $X$ with $Z\subsetneq X$. 
    Assume that $u$ is a $(p,q-1)$-form with $L^2_{loc}$ coefficients and $g$ is a $(p,q)$-form with $L^1_{loc}$ coefficients such that $\dbar u=g$ on $\Omega\setminus Z$ in the sense of distributions. Then we obtain the $\dbar$-extension $\dbar u=g$ on $\Omega$ in the sense of distributions.
\end{lemma}

Here, for any Hermitian metric $\omega$ on $X$, we can write $i^{n^2}s\wedge\overline{s}=|s|^2_\omega dV_\omega$, and Lemma \ref{Lemma: inequality of (n,0)-forms} below shows that this is independent of the choice of $\omega$.
In particular, $\pi^*\omega$ is not a Hermitian metric on $\tx$, since it degenerates along $\exc$. Nevertheless, for any Hermitian metric $\gamma$ on $\tx$, the equality $i^{n^2}\pi^*s\wedge\overline{\pi^*s}=|\pi^*s|^2_{\pi^*\omega}dV_{\pi^*\omega}=|\pi^*s|^2_\gamma dV_{\gamma}$ holds.

\begin{lemma}\label{Lemma: inequality of (n,0)-forms}
    Let $\gamma_1$ and $\gamma_2$ be Hermitian metrics on $X$ with $\gamma_1\geq\gamma_2$. 
    Then we have 
    \begin{itemize}
        \item the equality $|u|^2_{\gamma_1} dV_{\gamma_1}=|u|^2_{\gamma_2} dV_{\gamma_2}$ holds for any $(n,0)$-form $u$, 
        \item the inequality $|u|^2_{\gamma_1} dV_{\gamma_1}\leq|u|^2_{\gamma_2} dV_{\gamma_2}$ holds for any $(n,q)$-form $u$ and any $q\geq1$.
    \end{itemize}
\end{lemma}

For a singular Hermitian metric on a vector bundle, we introduce an $\cal{O}_X$-subsheaf of $\ogr\otimes\cal{O}_X(E)$, analogous to $\ogr$ twisted by an $L^2$-subsheaf.

\begin{definition}\label{Definition: GR canonical L2-subsheaf}
    Let $h$ be a singular Hermitian metric on $E$. We define the \textit{Grauert}-\textit{Riemenschneider canonical} $L^2$-\textit{subsheaf} $\ogr(E,h)$ on $X$ by
    \begin{align*}
        \Gamma(U,\omega_X^{GR}(E,h))=\biggl\{s\in\Gamma(U\cap\reg,\omega_{\reg}\otimes\cal{O}_X(E))\biggm| \int_{U_{reg}}|s|^2_{h,\omega}\,dV_\omega<+\infty\biggr\},
    \end{align*}
    for any relatively compact open subset $U\Subset X$ and a Hermitian metric $\omega$ on $X$. 
\end{definition}

By Lemma \ref{Lemma: inequality of (n,0)-forms} and the above discussion, this definition is independent of the choice of $\omega$. 
Furthermore, it is clear that $\ogr(E,h)=K_X\otimes\scr{E}(h)$ when $X$ is smooth.

\section{Definitions of Griffiths/Nakano positivity \\ for singular Hermitian metrics on complex spaces}

In this section, Griffiths and Nakano positivity are defined for singular Hermitian metrics on holomorphic vector bundles over complex spaces.
First, following \cite{BP08}, \cite[Definition\,1.2]{Rau15}, \cite[Definition\,2.2.2]{PT18} and \cite[Definition\,5.4]{Wat24a} in the complex manifold setting, Griffiths positivity on complex spaces is defined as follows.

\begin{definition}\label{Definition: singular Griffiths positivity}
    Let $X$ be a complex space and $E\longrightarrow X$ be a holomorphic vector bundle. 
    We say that a singular Hermitian metric $h$ on $E$ is 
    \begin{itemize}
        \item \textit{Griffiths semi-negative at point} $x\in X$ if there exists an open neighborhood $U$ of $x$ such that $\log|u|_h$ is plurisubharmonic on $U$ for any local holomorphic section $u\in H^0(U,E)$.
        \item \textit{a.e.\! Griffiths semi-negative at point} $x\in X$ if there exists an open neighborhood $U$ of $x$ such that $\log|u|_h$ coincides almost everywhere on $U$ with a plurisubharmonic function defined on $U$ for any local holomorphic section $u\in H^0(U,E)$.
        \item \textit{Griffiths negative at point} $x\in X$ if there exist an open neighborhood $U$ of $x$ and $\delta>0$ such that the norm $|\bullet|^2_h$ is locally integrable on $U$, and for any local holomorphic section $u\in H^0(U,E)$, the inequality
        \begin{align*}
            \idd|u|^2_h\geq\delta|u|^2_h\idd\iota_U^*|z|^2
        \end{align*}
        holds on $U$ in the sense of currents, where $U$ admits a closed holomorphic embedding $\iota_U:U\hookrightarrow\bb{C}^N$, $(z_1,\ldots,z_N)$ are local coordinates of $\bb{C}^N$, and $\delta$ depends on the choice of the coordinates $(z_1,\ldots,z_N)$.
        \item (resp. \textit{a.e.}) \textit{Griffiths semi-negative} (resp. \textit{Griffiths negative}) on $X$ if $h$ is (resp. a.e.) Griffiths semi-negative (resp. Griffiths negative) at any point $x\in X$, 
    \end{itemize}
    Furthermore, the \textit{Griffiths positivity} of $h$ is defined by the Griffiths negativity of the dual metric $h^*$ on $E^*$.
\end{definition}

In the case of line bundles, it is clear from the definition that a.e.\! Griffiths semi-positivity is equivalent to singular semi-positivity. 
However, Griffiths positivity carries only information in the sense of currents, and it is not clear whether it is equivalent to singular positivity without additional assumptions such as the normality of $X$.

\begin{remark}\label{Remark: Grif>0 then a.e. Grif semi-posi & Remark in Wat24a}
    Consider the case where $X$ is smooth, that is, a complex manifold. 
    If $h$ on $E$ is Griffiths positive on $X$, then $h$ is clearly a.e.\! Griffiths semi-positive on $X$ by Proposition \ref{Proposition: idd>0 iff a.e. psh}. 
    If $h$ on $E$ is a.e.\! Griffiths semi-positive on $X$, then there exists a singular Hermitian metric $\hbar$ on $E$ such that $\hbar$ is Griffiths semi-positive on $X$ and coincides with $h$ a.e. on $X$; see \cite[Remark 5.12]{Wat24a}.
\end{remark}

This metric $\hbar$ is obtained in the proof of \cite[Theorem 5.10]{Wat24a}, locally as the limit $\hbar^*=\lim_{\nu\to+\infty} h_\nu^*$, where $h_\nu^*:=h^*\ast\rho_\nu$ is constructed by convolution with an approximate identity $\{\rho_\nu\}_{\nu\in\bb{N}}$. 
However, it is unclear whether this approximation procedure remains valid when $X$ has singularities. More precisely, when $X$ is a complex space, it is unknown whether a Griffiths semi-positive singular Hermitian metric can be constructed from an a.e.\! Griffiths semi-positive singular Hermitian metric $h$.
Of course, after pulling back by a resolution of singularities $\pi:\tx\longrightarrow X$, such a construction is possible. However, it is unclear whether the Griffiths semi-positive singular Hermitian metric $\mathcal{H}:=\bigl(\lim_{\nu\to+\infty}\pi^*h_\nu^*\bigr)^*$, which coincides almost everywhere with $\pi^*h$, can be pushed forward to a Griffiths semi-positive singular Hermitian metric on $E$.
Indeed, $\mathcal{H}$ need not take the same value along the fibers of the $\pi$-exceptional divisor.

Second, following \cite{Ina22,DNWZ23} and \cite[Definition\,2.1]{IMW26} in the complex manifold setting, Nakano positivity on complex spaces will be defined.
Before proceeding, as a difference from these definitions, the following notion is introduced, which plays an important role in the $L^2$-Dolbeault complex and related contexts.

\begin{definition}\label{Definition: uniformly positive definite}
    Let $X$ be a complex space and $E\longrightarrow X$ be a holomorphic vector bundle. 
    A singular Hermitian metric $h$ on $E$ is said to be \textit{uniformly positive definite on} $X$ if, for every point of $X$, there exist a neighborhood $U$ of the point, a positive constant $C_U>0$, and a smooth Hermitian metric $h_0$ on $E|_U$ such that 
    \begin{align*}
        |\bullet|^2_h\geq C_U|\bullet|^2_{h_0}>0 \quad\text{almost everywhere on } U.
    \end{align*}
    In other words, the norm $|\bullet|^2_{h^*}$ associated with $h^*$ is uniformly bounded from above almost everywhere.
\end{definition}

\begin{definition}\label{Definition: singular Nakano positive}
    Let $X$ be a complex space and $\theta$ be a continuous $(1,1)$-form on $X$. 
    Let $E\longrightarrow X$ be a holomorphic vector bundle. 
    We define a singular Hermitian metric $h$ on $E$ to be \textit{$\theta$-Nakano positive on $X$ in the sense of $L^2$-estimates} 
    if $h$ is uniformly positive definite on $X$ and, 
    for any data consisting of
    \begin{itemize}
    \item a Stein space coordinate $S \subset X$ admitting a trivialization $E|_S\cong S\times \bb{C}^r$, 
    \item a \kah metric $\omega_S$ on $S$,
    \item a smooth function $\psi$ on $S$ such that $\theta + \idd \psi >0$ on $S$, 
    \item a positive integer $q$ with $1\leq q\leq n$,
    \item a $\dbar$-closed $g\in L^{2,loc}_{n,q}(S,E,he^{-\psi})$ 
    satisfying 
    \begin{align*}
        \int_S \langle B^{-1}_{\theta,\psi,\omega_S}g, g\rangle_{h,\omega_S}\,e^{-\psi}dV_{\omega_S}<+\infty,
    \end{align*} 
    \end{itemize}
    there exists $u\in L^2_{n,q-1}(S,E;\omega_S,he^{-\psi})$ such that 
    \begin{align*}
        \dbar u=g \,\,\, \text{on}\,\,\, S_{reg} \quad \text{and} \quad 
        \int_S |u|^2_{h,\omega_S}e^{-\psi} \,dV_{\omega_S}\leq \int_S \langle B^{-1}_{\theta,\psi,\omega_S}g, g\rangle_{h,\omega_S}\,e^{-\psi}dV_{\omega_S},
    \end{align*}
    where $B_{\theta,\psi,\omega_S}=[(\theta+\idd\psi)\oid{E},\Lambda_{\omega_S}]$. 

    In addition, when $\theta \geq 0$ on $X$, it is simply said to be \textit{Nakano semi-positive on $X$ (in the sense of $L^2$-estimates)}, and when $\theta > 0$ on $X$, it is simply said to be \textit{Nakano positive on $X$ (in the sense of $L^2$-estimates)}.
\end{definition}

In some works, upper semi-continuity of $h^*$ is also assumed; however, this assumption alone does not imply the boundedness of $h^*$. 
Since $h^*$ may tend to $+\infty$, the metric $h$ may degenerate to $0$. In this situation, it is unclear whether the $\dbar$-extension Lemma \ref{Lemma: dbar-extension lemma} can be applied even under local integrability with respect to the norm induced by $h$, and hence an $L^2$-Dolbeault resolution is not guaranteed (see Remark-Theorem \ref{Remark-Theorem: uniformly positive definite and L2-Dolbeault resolusions}), which is a crucial issue.
On complex manifolds, Griffiths semi-positivity of $h$ implies uniform positive definiteness, and the same statement holds on complex spaces (see Corollary \ref{Corollary: a.e. Grif semi-posi then uniformly positive definite}).
In the case of line bundles, if the weight of a singular Hermitian metric is quasi-plurisubharmonic, then uniform positive definiteness is automatically ensured.

Singular positivity in the sense of Definitions \ref{Definition: singular Griffiths positivity} and \ref{Definition: singular Nakano positive} is abundant by Proposition \ref{Proposition: smooth Grif/Nak>0 on X_reg then also >0 on X}, and provides a natural extension of the case of smooth Hermitian metrics. 

\begin{proposition}\label{Proposition: smooth Grif/Nak>0 on X_reg then also >0 on X}
    Let $h$ be a smooth Hermitian metric on $E$. If $h$ is Griffiths (resp. Nakano) semi-positive on $\reg$ in the usual sense, then $h$ is also Griffiths (resp. Nakano) semi-positive on $X$. 
    If $h$ is Griffiths (resp. Nakano) positive on $\reg$ in the usual sense, and in particular uniformly Griffiths (resp. Nakano) positive in the usual sense near $\xs$, then $h$ is also Griffiths (resp. Nakano) positive on $X$. 
\end{proposition}

A proof of Proposition \ref{Proposition: smooth Grif/Nak>0 on X_reg then also >0 on X} is given at the end of the next section.
Here, \textit{uniformly Griffiths/Nakano positive near} $\xs$ means that the positivity does not tend to $0$ as one approaches $\xs$. 
Equivalently, for every $x \in \xs$, there exist a neighborhood $U$ of $x$ admitting a closed holomorphic embedding $\iota_U : U \hookrightarrow \mathbb{C}^N$ and a positive constant $\delta>0$ such that $he^{\delta\iota_U^*|z|^2}$ is Griffiths/Nakano semi-positive in the usual sense on $\ureg$, where $(z_1,\ldots,z_N)$ denotes local coordinates on $\mathbb{C}^N$.
Furthermore, Proposition \ref{Proposition: smooth Grif/Nak>0 on X_reg then also >0 on X} also holds under the weaker assumption that $h$ is continuous by \cite[Corollary 1.12]{Dem85}.

Since the converse of Proposition \ref{Proposition: smooth Grif/Nak>0 on X_reg then also >0 on X} is clear from definitions, a smooth Hermitian metric $h$ is simply said to be \textit{Griffiths} (resp. \textit{Nakano}) \textit{positive} (resp. \textit{semi}-\textit{positive}) \textit{on} $X$ if it satisfies the assumptions of Proposition \ref{Proposition: smooth Grif/Nak>0 on X_reg then also >0 on X}, respectively.

\section{Properties of Griffiths/Nakano positivity on complex spaces}

In this section, we provide properties of singular Griffiths/Nakano positivity.
Throughout this section, $X$ is a complex space of pure dimension $n$, $\pi:\tx\longrightarrow X$ is a resolution of singularities, $E\longrightarrow X$ is a holomorphic vector bundle with a singular Hermitian metric $h$, and $L\longrightarrow X$ be a holomorphic line bundle with a singular Hermitian metric $h_L$. 
First, various properties of Griffiths positivity are presented. 

\begin{theorem}\label{Theorem: Grif semi-positivity of h and pi^*h}
    If $h$ is Griffiths semi-positive on $X$, then the pulled-back $\pi^*h$ on $\pi^*E$ is also Griffiths semi-positive on $\tx$. 
    Conversely, if $X$ is locally irreducible and $\pi^* h$ is Griffiths semi-positive on $\tx$, then $h$ is also Griffiths semi-positive on $X$.
\end{theorem}

\begin{proof}
    Griffiths semi-positivity is a local notion, and the equivalence on $\tx\setminus\exc$ and on $\reg$ is from the biholomorphicity of the restriction $\pi|_{\tx\setminus\exc}:\tx\setminus\exc\overset{\simeq}{\longrightarrow}\reg$. 
    
    First, we prove the Griffiths semi-positivity of $\pi^*h$ on $\exc$ by the Griffiths semi-positivity of $h$. 
    That is, it is sufficient to show that for any $y\in\exc$, there exists a Stein neighborhood $\ts$ of $y$ such that $|u|^2_{\pi^*h^*}$ is a plurisubharmonic function for any local section $u\in H^0(\ts,\pi^*E^*)$. 
    Here, by the Griffiths semi-positivity on $\tx\setminus\exc$, $|u|^2_{\pi^*h^*}$ is already plurisubharmonic on $\ts\setminus\exc$, and it remains to show that it is in fact plurisubharmonic on $\ts$.
    We take $\ts$ sufficiently small such that there exists a Stein subset $\tl{V}$ with $\ts\Subset\tl{V}$ and $\pi^*E$ trivial on $\tl{V}$. 
    Let $\{\rho_\nu\}_{\nu\in\bb{N}}$ be an approximate identity on $\tl{V}$.

    We show that each smooth Hermitian metric $\pi^*h^*_\nu:=\pi^*h^*\ast\rho_\nu$, which is a regularization of $\pi^*h^*$, is Griffiths semi-positive on $\ts$. 
    For any $x\in\xs$ and any local section $s\in\cal{O}(E^*)_{X,x}$, the plurisubharmonicity of $\pi^*\log|s|^2_{h^*}=\log |\pi^*s|^2_{\pi^*h^*}$ implies that $|\pi^*s|^2_{\pi^*h^*}$ is also plurisubharmonic on a neighborhood of $\pi^{-1}(x)$. 
    Hence, viewing $|\bullet|^2_{\pi^*h^*}$ as an inner product, $|\bullet|^2_{\pi^*h^*}$ is locally bounded from above on $\ts$.
    Therefore, there exists a unique plurisubharmonic function $\varPsi_{h,u}$ on $\ts$ such that $\varPsi_{h,u} = |u|^2_{\pi^*h^*}$ on $\ts\setminus\exc$. Consider the integral 
    \begin{align*}
        |u|^2_{\pi^*h^*_\nu}(z)=\int_{\rom{supp}\,\rho_\nu}|u|^2_{\pi^*h^*_{(w)}}(z)\rho_\nu(w)dV_w,
    \end{align*} 
    where $\pi^*h^*_{(w)}(z)=\pi^*h^*(z-w)$ denotes the translation of $\pi^*h^*$.  
    For each fixed $w$, as above, there exists a unique plurisubharmonic extension $\varPsi_{h,u,w}$ such that $\varPsi_{h,u,w}(z)=|u|^2_{\pi^*h^*_{(w)}}(z)$ on $\ts\setminus\exc$.
    Hence, the integral above coincides with the integral of $|u|^2_{\pi^*h^*_{(w)}}(z)$, viewed as a plurisubharmonic function on $\ts$. Therefore, $\pi^*h^*_\nu$ is Griffiths semi-negative (see the proof of \cite[Proposition 6.2]{Rau15}).

    For completeness, we also give a proof using another construction. 
    The unique extension $\varPsi_{h,u}$ of the plurisubharmonic function $|u|^2_{\pi^*h^*}$ on $\ts\setminus\exc$ is defined for $y\in\exc$ by $\varPsi_{h,u}(y):=\limsup_{z\to y,\,z\in S\setminus\exc}|u|^2_{\pi^*h^*}(z)=\limsup_{z\to y,\,z\in S\setminus\exc}|u(y)|^2_{\pi^*h^*(z)}$. 
    Each component of $\pi^*h^*=(H_{jk})_{1\leq j,k\leq r}$ can be reconstructed from constant vectors (see the proof of \cite[Theorem 5.10]{Wat24a}). Indeed, the constant vectors $v_1={}^t(1,0,\ldots,0),\ldots,v_r={}^t(0,\ldots,0,1)$ give the diagonal entries by $H_{jj}=|v_j|^2_{\pi^*h^*}$.
    Furthermore, off-diagonal entries such as $H_{12}$ can be reconstructed from the constant vectors $v={}^t(1,1,\ldots,0), v'={}^t(1,i,\ldots,0)$, through the identities $|v|^2_{\pi^*h^*}=H_{11}+H_{22}+2\rom{Re}\,H_{12}$ and $|v'|^2_{\pi^*h^*}=H_{11}+H_{22}+2\rom{Im}\,H_{12}$.
    Hence, a new singular Hermitian metric $\widetilde{\pi^*h^*}:=(\widetilde H_{jk})$ is defined by setting $\tl{\pi^*h^*}=\pi^*h^*$ on $\ts\setminus\exc$, and for each $y\in\exc$, by setting the diagonal entries as $\widetilde H_{jj}(y):=\limsup_{z\to y, z\in \ts\setminus\exc}|v_j|^2_{\pi^*h^*}$ and defining the remaining entries similarly; 
    for example, $\tl{H}_{12}$ is defined by $2\rom{Re}\,\tl{H}_{12}(y):=\tl{H}_{11}(y)+\tl{H}_{22}(y)-\limsup_{z\to y, z\in \ts\setminus\exc}|v|^2_{\pi^*h^*}$ and $2\rom{Im}\,\tl{H}_{12}(y):=\tl{H}_{11}(y)+\tl{H}_{22}(y)-\limsup_{z\to y, z\in \ts\setminus\exc}|v'|^2_{\pi^*h^*}$.
    Then, for any constant section $c$, we obtain $|c|^2_{\tl{\pi^*h^*}}(y)=\limsup_{z\to y, z\in \ts\setminus\exc}|c|^2_{\pi^*h^*(z)}=\varPsi_{h,c}(y)$, which implies that $|c|^2_{\tl{\pi^*h^*}}$ is plurisubharmonic on $\ts$.
    Thus, for any holomorphic section $u\in H^0(S,\pi^*E^*)$, the function $|u|^2_{\widetilde{\pi^*h^*}}$ is plurisubharmonic on $\ts$, and $|u|^2_{\widetilde{\pi^*h^*}}=\varPsi_{h,u}=|u|^2_{\pi^*h^*}$ on $\ts\setminus\exc$.
    In particular, $\tl{\pi^*h^*}$ is Griffiths semi-negative.
    Under this construction, we obtain
    \begin{align*}
        |\sigma|^2_{\pi^*h^*_\nu}(z)=\int_{\rom{supp}\,\rho_\nu}\!\!\!\!|\sigma|^2_{\pi^*h^*_{(w)}}(z)\rho_\nu(w)dV_w=\int_{\rom{supp}\,\rho_\nu}\!\!\!\!|\sigma|^2_{\tl{\pi^*h^*}_{(w)}}(z)\rho_\nu(w)dV_w=|\sigma|^2_{(\tl{\pi^*h^*})\ast\rho_\nu}(z),
    \end{align*}
    which implies that $(\tl{\pi^*h^*})\ast\rho_\nu=\pi^*h^*_\nu$ is Griffiths semi-negative (see the proof of \cite[Proposition 6.2]{Rau15}).

    It remains to show that $\pi^*h^*_\nu$ is decreasing with respect to $\nu$. 
    To verify the decreasing property, it is necessary to show that $|\tau|^2_{\pi^*h^*}$ is plurisubharmonic for any constant section $\tau\in H^{0}(\cal{U},\pi^{*}E^*)$, where $\cal{U}\subset\ts$. 
    Let a local holomorphic frame of $\pi^*E^*|_{\cal{U}}$ be given by the pullback $\pi^*e=(\pi^{*}e_{1},\ldots,\pi^{*}e_{r})$ of a local holomorphic frame $e=(e_{1},\ldots,e_{r})$ of $E^*|_{\pi(\cal{U})}$. 
    Then $\tau=\sum c_{j}\pi^{*}e_{j}=\pi^{*}\bigl(\sum c_{j}e_{j}\bigr)$, $c_j\in\bb{C}$, and $\pi_{*}\tau=\sum c_{j}e_{j}\in H^{0}(\pi(\cal{U}),E^*)$. 
    Hence, $|\tau|^2_{\pi^*h^*}=|\pi^*\pi_*\tau|^2_{\pi^*h^*}=\pi^*(|\pi_*\tau|^2_{h^*})$ and by assumption, $|\pi_*\tau|^2_{h^*}$ is also plurisubharmonic. 
    Therefore, $|\tau|^2_{\pi^*h^*}$ is also plurisubharmonic, which implies that $\pi^*h^*_\nu$ is decreasing and converges pointwise to $\pi^*h^*$.
    Indeed, by the plurisubharmonicity, $|\tau|^2_{\pi^*h^*}=\lim_{\nu\to+\infty}|\tau|^2_{\pi^*h^*}\ast\rho_\nu=\lim_{\nu\to+\infty}|\tau|^2_{\pi^*h^*_\nu}$ (see \cite[Chapter I, Theorem 5.5]{Dem-book}), and therefore $\pi^*h^*=\lim_{\nu\to+\infty}\pi^*h^*_\nu$. 
    From the plurisubharmonicity of $|\sigma|^2_{\pi^*h^*_\nu}$ for any $\sigma\in H^0(\ts,\pi^*E^*)$ and any $\nu$, its decreasing limit $|\sigma|^2_{\pi^*h^*}=\lim_{\nu\to+\infty}|\sigma|^2_{\pi^*h^*_\nu}$ is also plurisubharmonic, which shows that $\pi^*h^*$ is Griffiths semi-negative on $\exc$ as well.

    Finally, we show the Griffiths semi-positivity of $h$ on $X$ from the Griffiths semi-positivity of $\pi^{*}h$ on $\tx$. It suffices to prove that, for any $x\in\xs$ and any $s\in\cal{O}(E^*)_{X,x}$, the function $\log |s|^2_{h^*}$ is plurisubharmonic on an open neighborhood $U$ of $x$, where $\log |s|^2_{h^*}$ is already plurisubharmonic on $\ureg$. 
    By the Griffiths semi-positivity of $\pi^*h$, the pulled-back $\pi^*\log |s|^2_{h^*}=\log |\pi^*s|^2_{\pi^*h^*}$ is plurisubharmonic on $\pi^{-1}(U)$, and the function $\log|\pi^*c|^2_{\pi^*h^*}=\pi^*\log|c|^2_{h^*}$ is also plurisubharmonic on $\pi^{-1}(U)$ for any constant section $c\in H^0(U,E^*)$. Thus, $\log|\bullet|^2_{h^*}$ is locally bounded from above on $U$.
    Together with Theorem \ref{Theorem: idd>0 and a.e. psh}, this implies that the plurisubharmonic function $\log |s|^2_{h^*}$ on $\ureg$ extends uniquely to a plurisubharmonic function $\varPsi_{log,s}$ on $U$. 
    Hence, the two plurisubharmonic functions $\pi^*\log |s|^2_{h^*}$ and $\pi^*\varPsi_{log,s}$ on $\pi^{-1}(U)$ satisfy $\pi^*\log |s|^2_{h^*}=\pi^*\varPsi_{log,s}$ on $\pi^{-1}(U)\setminus\exc$, and by the Lemma \ref{Lemma: two psh} below, it follows that $\pi^*\log |s|^2_{h^*}=\pi^*\varPsi_{log,s}$ on $\pi^{-1}(U)$.
    This implies that $\log |s|^2_{h^*}=\varPsi_{log,s}$ on $U$, 
    which shows that $\pi^*h$ is Griffiths semi-positive on $X$.
\end{proof}

\begin{lemma}\label{Lemma: two psh}
    Let $\Omega$ be an open subset $\bb{C}^n$ and $A$ be a proper analytic subset of $\Omega$, and let $\varphi$ and $\psi$ be plurisubharmonic functions on $\Omega$. 
    If $\varphi=\psi$ on $\Omega\setminus A$, then $\varphi=\psi$ on $\Omega$.
\end{lemma}

\begin{proof}
    We obtain $\varphi\ast\rho_\nu=\psi\ast\rho_\nu$ on $\Omega$ for every $\nu\in\bb{N}$, and hence, it follows that $\displaystyle\varphi=\lim_{\nu\to+\infty}\varphi\ast\rho_\nu=\lim_{\nu\to+\infty}\psi\ast\rho_\nu=\psi$ on $\Omega$ by \cite[Chapter I, Theorem 5.5]{Dem-book}.
\end{proof}

\begin{corollary}\label{Corollary: a.e. Grif semi-posi of h and pi*h}
    If $h$ is a.e.\! Griffiths semi-positive on $X$, then $\pi^*h$ on $\pi^*E$ is also a.e.\! Griffiths semi-positive on $\tx$. 
    Conversely, if $X$ is locally irreducible and $\pi^* h$ is a.e.\! Griffiths semi-positive on $\tx$, then $h$ is also a.e.\! Griffiths semi-positive on $X$.
\end{corollary}

\begin{proof}
    As in the proof of Theorem \ref{Theorem: Grif semi-positivity of h and pi^*h}, by setting $\cal{H}:=(\lim_{\nu\to+\infty}\pi^*h^*_\nu)^*$, we obtain a Griffiths semi-positive singular metric on $\pi^{*}E$, which coincides with $\pi^{*}h$ a.e.
\end{proof}

\begin{corollary}\label{Corollary: a.e. Grif semi-posi then uniformly positive definite}
    If $h$ is a.e.\! Griffiths semi-positive, then $h$ is uniformly positive definite. 
\end{corollary}

\begin{proof}
    From the above discussion, a singular Hermitian metric $\cal{H}:=(\lim_{\nu\to+\infty}\pi^*h^*_\nu)^*$ on $\pi^*E$, which coincides with $\pi^*h$ a.e., is Griffiths semi-positive. 
    Each $\pi^*h^*_\nu$ is a smooth Griffiths semi-negative metric. The sequence of smooth metrics $\{\pi^*h^*_\nu\}_{\nu\in\bb{N}}$ is decreasing and converges pointwise a.e. to $\cal{H}^*$. 
    Hence, $|\bullet|^2_{\pi^*h}=|\bullet|^2_{\cal{H}}\geq|\bullet|^2_{(\pi^*h^*_\nu)^*}$ a.e. Since $\pi^*h$ is uniformly positive definite, $h$ is also uniformly positive definite.
\end{proof}

In the case of line bundles, the equivalence between a.e.\! Griffiths semi-positivity and singular semi-positivity follows immediately from the definitions. 
Hence, Corollary \ref{Corollary: a.e. Grif semi-posi of h and pi*h} yields the following.

\begin{corollary}\label{Corollary: singular semi-positivity of h_L and pi*h_L}
    If $h_L$ on $L$ is singular semi-positive on $X$, then $\pi^*h_L$ on $\pi^*L$ is also singular semi-positive on $\tx$. 
    Conversely, if $X$ is locally irreducible and $\pi^*h_L$ on $\pi^*L$ is singular semi-positive on $\tx$, then $h_L$ on $L$ is also singular semi-positive on $X$.
\end{corollary}

Since $\tx$ is smooth, it follows from \cite[Proposition 5.3]{Wat24a} and Proposition \ref{Proposition: idd>0 iff a.e. psh} that $\pi^*h$ is a.e.\! Griffiths semi-positive on $\tx$ if and only if 
for any point $x\in\tx$, there exists an open neighborhood $\tu$ of $x$ such that $\idd|\tl{s}|^2_{\pi^*h^*}\geq0$ on $\tu$ in the sense of currents, for any local holomorphic section $\tl{s}\in H^0(\tu,\pi^*E^*)$.
Furthermore, in the case of a line bundle $L=E$, this is equivalent to $\iO{\pi^*L,\pi^*h}\geq0$ on $\tl{X}$ in the sense of currents.
However, when $X$ is a complex space, these notions do not necessarily coincide by Theorem \ref{Theorem: idd>0 and a.e. psh}, whereas they do coincide when $X$ is normal.

As in Remark \ref{Remark: Grif>0 then a.e. Grif semi-posi & Remark in Wat24a}, it is expected that Griffiths positivity implies a.e.\! Griffiths semi-positivity near $\xs$. 
However, this is unclear because of difficulties with the approximation argument. The following result nevertheless holds.

\begin{proposition}\label{Proposition: normal and Gri>0 then a.e. Grif>0}
    If $X$ is normal, or if $X$ is locally irreducible and $h^*$ is locally bounded from above a neighborhood of $\xs$, and if $h$ is Griffiths positive on $X$, then $h$ is a.e.\! Griffiths semi-positive on $X$.
\end{proposition}

\begin{proof}
    By the assumption, for any open subset $U\subset X$ and any section $u\in H^0(U,E^*)$, $\idd|u|^2_{h^*}\geq0$ holds on $U$ in the sense of currents. 
    Then, for any strongly positive test form $\tl{\phi}\in\scr{D}_{\tl{X}}^{n-1,n-1}(\pi^{-1}(U)\setminus\exc)$, the pullback $(\pi^{-1})^*\tl{\phi}\in\scr{D}_X^{n-1,n-1}(\ureg)$ is also strongly positive test form on $\ureg$, and hence the following inequality holds 
    \begin{align*}
        0\leq\lara{\idd|u|^2_{h^*}}{(\pi^{-1})^*\tl{\phi}}=\int_{\pi^{-1}(U)\setminus\exc}\idd|\pi^*u|^2_{\pi^*h^*}\wedge\tl{\phi}=\lara{\idd|\pi^*u|^2_{\pi^*h^*}}{\tl{\phi}}.
    \end{align*}
    Therefore, we have $\idd|\pi^*u|^2_{\pi^*h^*}\geq0$ on $\pi^{-1}(U)\setminus\exc$ in the sense of currents. 
    Since the restriction $\pi|_{\reg}$ is biholomorphic, this shows that $\pi^{*}h^{*}$ is a.e.\! Griffiths semi-positive on $\tx\setminus\exc$. 
    Hence, $\pi^*\log|u|^2_{h^*}=\log|\pi^*u|^2_{\pi^*h^*}$ coincides a.e. on $\pi^{-1}(U)\setminus\exc$ with a plurisubharmonic function defined on $\pi^{-1}(U)\setminus\exc$. 
    That is, $\log|u|^2_{h^*}$ coincides a.e. on $\ureg$ with a plurisubharmonic function on $\varPsi_{log,u}$ defined on $\ureg$. 
    By the assumption together with 
    Theorem \ref{Theorem: idd>0 and a.e. psh}, the function $\varPsi_{log,u}$ extends uniquely to a plurisubharmonic function $\tl{\varPsi}_{log,u}$ on $U$, and $\log |u|^2_{h^*}$ coincides a.e. on $U$ with $\tl{\varPsi}_{log,u}$, which shows the a.e.\! Griffiths semi-positivity of $h$.
\end{proof}

The following is immediate from the same argument as in \cite[Proposition 3.1]{Wat26b}.

\begin{proposition}\label{Proposition: Grif>0 then εω-Grif>0}
    If $h$ is Griffiths positive, then for any Hermitian metric $\omega$ on $X$, there exists a smooth positive function $\varepsilon:X\longrightarrow\bb{R}_{>0}$ satisfying the following: for every point $x\in X$, 
    there exists a neighborhood $U$ of $x$ such that $\idd|u|^2_{h^*}\geq\varepsilon|u|^2_{h^*}\omega$ on $U$ in the sense of currents, for any local holomorphic section $u\in H^0(U,E^*)$.
\end{proposition}

In this case, $h$ is said to be $\varepsilon\omega$-\textit{Griffiths positive}, and one writes $i\Theta_{E,h} \geq_{Grif} \varepsilon\omega \otimes \rom{id}_E$.

\begin{theorem}\label{Theorem: h Gri iff pi*h Gri}
    Let $\omega$ be a Hermitian metric on $X$ and $\varepsilon:X\longrightarrow\bb{R}_{\geq0}$ be a smooth semi-positive function.
    If the pulled-back $\pi^*h$ on $\pi^*E$ is $\pi^*(\varepsilon\omega)$-Griffiths positive on $\tx$, then $h$ on $E$ is $\varepsilon\omega$-Griffiths positive on $X$. 
    Conversely, if $X$ is normal, or if $X$ is locally irreducible and $h^*$ is locally bounded from above a neighborhood of $\xs$, and if $h$ is $\varepsilon\omega$-Griffiths positive on $X$, then $\pi^*h$ on $\pi^*E$ is $\pi^*(\varepsilon\omega)$-Griffiths positive on $\tx$.
\end{theorem}

\begin{proof}
    The biholomorphicity of $\pi|_{\tx\setminus\exc}:\tx\setminus\exc\overset{\simeq}{\longrightarrow}\reg$ 
    immediately yields the equivalence between the positivity on $\tx\setminus\exc$ and the positivity on $\reg$.

    First, assume that $\pi^*h$ is $\pi^*(\varepsilon\omega)$-Griffiths positive on $\tx$. 
    It suffices to show that for any point $x\in \xs$, there exists an open neighborhood $U$ of $x$ such that $\idd|s|^2_{h^*}\geq|s|^2_{h^*}\varepsilon\omega$ on $U$ in the sense of currents, for any section $s\in H^0(U,E^*)$. 
    This reduces to showing $\lara{\idd|s|^2_{h^*}}{\phi}\geq\lara{|s|^2_{h^*}\varepsilon\omega}{\phi}$ for any strongly positive test form $\phi\in\scr{D}_X^{n-1,n-1}(U)$, which can be written as $\phi=\rho\cdot i^{(n-1)^2}\eta\wedge\overline{\eta}$ for suitable forms $\rho\in\scr{D}(X,\bb{R}_{\geq0})$ and $\eta\in\cal{E}_X^{n-1,0}(U)$. 
    By $\pi^*(\varepsilon\omega)$-Griffiths positivity of $\pi^*h$ on $\tx$, we obtain 
    \begin{align*}
        \lara{\idd|s|^2_{h^*}}{\phi}&=\int_{\ureg}|s|^2_{h^*}\,\idd\phi=\int_{\pi^{-1}(U)\setminus\exc}\pi^*|s|^2_{h^*}\,\idd\pi^*\phi\\
        &=\int_{\pi^{-1}(U)}|\pi^*s|^2_{\pi^*h^*}\,\idd\pi^*\phi=\lara{\idd|\pi^*s|^2_{\pi^*h^*}}{\pi^*\phi}\\ 
        &\geq\lara{\pi^*(\varepsilon\omega)|\pi^*s|^2_{\pi^*h^*}}{\pi^*\phi}=\lara{\varepsilon|s|^2_{h^*}\omega}{\phi},
    \end{align*}
    where $\pi^*s\in H^0(\pi^{-1}(U),\pi^*E^*)$.

    Second, the converse is shown. Assume that $h$ is $\varepsilon\omega$-Griffiths positive on $X$.
    By the $\pi^*(\varepsilon\omega)$-Griffiths positivity of $\pi^*h$ on $\tx\setminus\exc$, for any point $x\in\exc$, there exists an open neighborhood $\tu$ of $x$ such that $\idd|\tl{s}|^2_{\pi^*h^*}\geq|\tl{s}|^2_{\pi^*h^*}\pi^*(\varepsilon\omega)$ on $\tu\setminus\exc$ in the sense of currents, for any local section $\tl{s}\in H^0(\tu,\pi^*E^*)$.
    By Proposition \ref{Proposition: normal and Gri>0 then a.e. Grif>0} and Corollary \ref{Corollary: a.e. Grif semi-posi of h and pi*h}, $\pi^*h$ is a.e.\! Griffiths semi-positive on $\tx$, and $|\tl{s}|^2_{\pi^*h^*}$ coincides a.e. on $\tu$ with a plurisubharmonic function defined on $\tu$. Therefore, by Skoda-El Mir extension theorem (see \cite[Chapter III, Theorem 2.3]{Dem-book}), $\idd|\tl{s}|^2_{\pi^*h^*}$ and $|\tl{s}|^2_{\pi^*h^*}\pi^*(\varepsilon\omega)$ extend uniquely as currents on $\tu$.

    We prove that $\idd|\tl{s}|^2_{\pi^*h}\geq|\tl{s}|^2_{\pi^*h}\pi^*(\varepsilon\omega)$ on $\tu$ in the sense of currents.
    Assume that there exists a strong positive test form $\phi\in\scr{D}^{n-1,n-1}_{\tx}(\tu)$ such that $\lara{\idd|\tl{s}|^2_{\pi^*h^*}}{\phi}<\lara{|\tl{s}|^2_{\pi^*h^*}\pi^*(\varepsilon\omega)}{\phi}$, i.e., $C:=\lara{|\tl{s}|^2_{\pi^*h^*}\pi^*(\varepsilon\omega)}{\phi}-\lara{\idd|\tl{s}|^2_{\pi^*h^*}}{\phi}>0$, and derive a contradiction.
    For any $\tau > 0$, let $\exc_{\tau}$ denote the $\tau$-neighborhood of $\exc$. Then there exists $\tl{\phi}_{\tau}\in\scr{D}^{n-1,n-1}_{\tx}(\exc_{\tau})$ such that $\tl{\phi}_{\tau}=\phi$ on $\exc_{\tau/2}$. 
    Setting $\psi := \phi-\tl{\phi}_{\tau}$, we have $\psi\in\scr{D}^{n-1,n-1}_{\tx}(\tu\setminus\exc)$, and the $\pi^*(\varepsilon\omega)$-Griffiths positivity of $\pi^*h$ on $\tx\setminus\exc$ implies 
    \begin{align*}
        \lara{|\tl{s}|^2_{\pi^*h^*}\pi^*(\varepsilon\omega)}{\psi+\tl{\phi}_{\tau}}&=\lara{|\tl{s}|^2_{\pi^*h^*}\pi^*(\varepsilon\omega)}{\phi}=C+\lara{\idd|\tl{s}|^2_{\pi^*h^*}}{\phi}\\
        &=C+\lara{\idd|\tl{s}|^2_{\pi^*h^*}}{\psi}+\lara{\idd|\tl{s}|^2_{\pi^*h^*}}{\tl{\phi}_{\tau}}\geq C+\lara{|\tl{s}|^2_{\pi^*h^*}\pi^*(\varepsilon\omega)}{\psi},
    \end{align*}
    where the a.e.\! Griffiths semi-positivity of $\pi^*h$ on $\tx$ obtained from Proposition \ref{Proposition: normal and Gri>0 then a.e. Grif>0} and Corollary \ref{Corollary: a.e. Grif semi-posi of h and pi*h} implies that $\lara{\idd|\tl{s}|^2_{\pi^*h^*}}{\tl{\phi}_{\tau}}\geq0$.
    Thus, we obtain $\lara{|\tl{s}|^2_{\pi^*h^*}\pi^*(\varepsilon\omega)}{\tl{\phi}_{\tau}}\geq C$.
    Here, 
    $|\tl{s}|^2_{\pi^*h^*}$ is locally bounded from above, and since $\pi^*\omega$ degenerates along the tangent directions of $\exc$ on $\exc$, and the support of $\tl{\phi}_{\tau}$ can be taken arbitrarily small under the condition $\exc\cap\rom{supp}\,\phi\subset\rom{supp}\,\tl{\phi}_{\tau}$, there exists $\tau>0$ such that $\lara{|\tl{s}|^2_{\pi^*h^*}\pi^*(\varepsilon\omega)}{\tl{\phi}_{\tau}}<C$, which yields a contradiction.
    Indeed, $\pi^*(\varepsilon\omega)$ is smooth and does not admit any singular exploding measure such as a delta distribution.
    Hence, the $\pi^*(\varepsilon\omega)$-Griffiths positivity of $\pi^*h$ on $\tx$ has been proved.
\end{proof}

\begin{remark}\label{Remark: condition of a.e. Grif in Thm of Grif posi}
    In Theorem \ref{Theorem: h Gri iff pi*h Gri}, the assumptions that $X$ is normal, or that $X$ is locally irreducible and $h^*$ is local boundedness from above, are only used to ensure the a.e.\! Griffiths semi-positivity of $\pi^*h$. 
    Therefore, the conclusion remains valid if these assumptions are replaced by the assumption that $\pi^*h$ is a.e.\! Griffiths semi-positive.
\end{remark}

From the proof of Theorem \ref{Theorem: h Gri iff pi*h Gri} and Remark \ref{Remark: condition of a.e. Grif in Thm of Grif posi}, the following corollary holds.

\begin{corollary}\label{Corollary: iO_h>0 iff iO_pi*h>0}
    Let $\omega$ be a Hermitian metric on $X$ and $\varepsilon:X\longrightarrow\bb{R}_{\geq0}$ be a smooth semi-positive function.
    If $\iO{\pi^*L,\pi^*h_L}\geq\pi^*(\varepsilon\omega)$ holds on $\tx$ in the sense of currents, then $\iO{L,h_L}\geq\varepsilon\omega$ holds on $X$ in the sense of currents. 
    If $\pi^*h_L$ on $\pi^*L$ is singular semi-positive on $\tx$ and $\iO{L,h_L}\geq\varepsilon\omega$ holds on $X$ in the sense of currents, then $\iO{\pi^*L,\pi^*h_L}\geq\pi^*(\varepsilon\omega)$ holds on $\tx$ in the sense of currents.
\end{corollary}

\begin{theorem}\label{Theorem: h>Grif then iO_det h>0 as currents}
    If $h$ is a.e.\! Griffiths semi-positive on $X$, then $\iO{\det E,\det h}\geq0$ holds on $X$ in the sense of currents. 
    If $h$ is Griffiths positive and a.e.\! Griffiths semi-positive on $X$, then for any Hermitian metric $\omega$ on $X$, there exists a smooth positive function $\varepsilon:X\longrightarrow\bb{R}_{>0}$ such that $\iO{\det E,\det h}\geq\varepsilon\omega$ on $X$ in the sense of currents. 
    
    Furthermore, if $X$ is normal and $h$ is Griffiths positive (resp. a.e.\! Griffiths semi-positive) on $X$, then the singular Hermitian metric $\det h$ on $\det E$ is singular positive (resp. singular semi-positive) on $X$. 
\end{theorem}

\begin{proof}
    First, we prove that $\iO{\det E,\det h}\geq0$ on $X$ in the sense of currents. By Corollary \ref{Corollary: a.e. Grif semi-posi of h and pi*h}, $\pi^*h$ is also a.e.\! Griffiths semi-positive on $\tx$. 
    Remark \ref{Remark: Grif>0 then a.e. Grif semi-posi & Remark in Wat24a} and \cite[Proposition 1.3]{Rau15} imply that $\det\pi^*h$ on $\det\pi^*E$ is singular semi-positive on $\tx$. 
    Hence, the result follows from Corollary \ref{Corollary: iO_h>0 iff iO_pi*h>0}. 

    We show that $\iO{\det E,\det h}\geq\varepsilon\omega$ on $X$ in the sense of currents. 
    By Proposition \ref{Proposition: Grif>0 then εω-Grif>0}, Theorem \ref{Theorem: h Gri iff pi*h Gri} and Remark \ref{Remark: condition of a.e. Grif in Thm of Grif posi}, there exists a smooth positive function $\tau:X\longrightarrow\bb{R}_{>0}$ such that $\pi^*h$ is $\pi^*(2\tau\omega)$-Griffiths positive on $\tx$. 
    For every $x\in\tx\setminus\exc$, there exist a Stein open neighborhood $\tu \subset \tx\setminus\exc$ of $x$ and a constant $\delta > 0$ such that $\pi^*(2\tau\omega)\geq\delta\idd|z|^2\geq\pi^*(\tau\omega)$ on $\tu$, and $\idd|u|^2_{\pi^*h^*}\geq\delta|u|^2_{\pi^*h^*}\idd|z|^2$ for any section $u\in H^0(\tu,\pi^*E^*)$, where $(z_1,\ldots,z_n)$ is a local coordinate of $\tu$.
    It follows from \cite[Proposition 5.5 (d)]{Wat24a} that $\log|u|^2_{\pi^*h^*}-\delta|z|^2$ coincides a.e. on $\tu$ with a plurisubharmonic function defined on $\tu$, and hence $\pi^*he^{\delta|z|^2}$ is a.e.\! Griffiths semi-positive on $\tu$. 
    Remark \ref{Remark: Grif>0 then a.e. Grif semi-posi & Remark in Wat24a} and \cite[Proposition 1.3]{Rau15} imply that $(\det \pi^*h)e^{r\delta|z|^2}$ is singular semi-positive on $\tu$, and hence we obtain $\iO{\det \pi^*E,\det \pi^*h}\geq r\delta\idd|z|^2\geq r\pi^*(\tau\omega)$ on $\tu$ in the sense of currents; in particular, this also holds on $\tx\setminus\exc$. 
    As above, by Corollary \ref{Corollary: a.e. Grif semi-posi of h and pi*h}, Remark \ref{Remark: Grif>0 then a.e. Grif semi-posi & Remark in Wat24a} and \cite[Proposition 1.3]{Rau15}, $\det \pi^*h$ is singular semi-positive, i.e., a.e.\! Griffiths semi-positive, on $\tx$.  
    Hence, arguing as in the proof of Theorem \ref{Theorem: h Gri iff pi*h Gri} and Remark \ref{Remark: condition of a.e. Grif in Thm of Grif posi}, we obtain $\iO{\det \pi^*E,\det \pi^*h}\geq r\pi^*(\tau\omega)$ on $\tx$ in the sense of currents. Finally, by choosing $\varepsilon = r\tau$, Corollary \ref{Corollary: iO_h>0 iff iO_pi*h>0} yields $\iO{\det E,\det h}\geq\varepsilon\omega$ on $X$ in the sense of currents. 

    The final assertion follows from Theorem \ref{Theorem: idd>0 and a.e. psh}.
\end{proof}

\begin{corollary}\label{Corollary: h>Grif and a.e Grif then iO_det pi*h>0 as currents} 
    If $h$ is Griffiths positive and a.e.\! Griffiths semi-positive on $X$, then for any Hermitian metric $\omega$ on $X$, there exists a smooth positive function $\varepsilon:X\longrightarrow\bb{R}_{>0}$ such that $\iO{\det \pi^*E,\det \pi^*h}\geq\pi^*(\varepsilon\omega)$ on $\tx$ in the sense of currents. 
\end{corollary}

Several properties of Nakano positivity and its relationship with Griffiths positivity are given below.

\begin{theorem}\label{Theorem: h Nak iff pi*h Nak}
    Let $\theta$ be a continuous $(1,1)$-form on $X$. 
    It follows that $h$ is $\theta$-Nakano positive on $X$ in the sense of $L^2$-estimates if and only if the pulled-back $\pi^*h$ on $\pi^*E$ is $\pi^*\theta$-Nakano positive on $\tx$ in the sense of $L^2$-estimates.
\end{theorem}

\begin{proof}
    First, assume that $h$ is $\theta$-Nakano positive on $X$. 
    Consider any data consisting of 
    \begin{itemize}
        \item a Stein coordinate $\tl{S}\subset\tx$ admitting a trivialization $\pi^*E|_{\tl{S}}\cong\tl{S}\times\bb{C}^r$,
        \item a \kah metric $\omega_{\ts}$ on $\ts$, 
        \item a smooth function $\tl{\psi}$ on $\ts$ such that $\pi^*\theta+\idd\tl{\psi}>0$ on $\ts$,
        \item a $\dbar$-closed $\tl{g}\in L^{2,loc}_{n,q}(\ts,\pi^*E,\pi^*he^{-\tl{\psi}})$ satisfying 
        \begin{align*}
            \int_{\ts}\lara{B^{-1}_{\pi^*\theta,\tl{\psi},\omega_{\ts}}\tl{g}}{\tl{g}}_{\pi^*h,\omega_{\ts}}e^{-\tl{\psi}}dV_{\omega_{\ts}}<+\infty.
        \end{align*}
    \end{itemize}
    Set $\sreg := \pi(\ts\setminus\exc)$, and denote the restriction map by $\mu := \pi^{-1}|_{\reg}:\reg\overset{\simeq}{\longrightarrow}\tx\setminus\exc$. 
    Here, there exists a hyperplane $H$ such that $\ts \setminus H$ is also Stein, and since $\exc$ is simple normal crossing, $\ts\setminus(H\cup\exc)$ is also Stein (see \cite[Theorem 3.12]{Die96}).
    Therefore, since $\mu$ is biholomorphic, the set $\sreg\setminus\pi(H)$ is also Stein, and we obtain 
    \begin{align*}
        \int_{\ts}\lara{B^{-1}_{\pi^*\theta,\tl{\psi},\omega_{\ts}}\tl{g}}{\tl{g}}_{\pi^*h,\omega_{\ts}}e^{-\tl{\psi}}dV_{\omega_{\ts}}&=\int_{\ts\setminus(H\cup\exc)}\lara{B^{-1}_{\pi^*\theta,\tl{\psi},\omega_{\ts}}\tl{g}}{\tl{g}}_{\pi^*h,\omega_{\ts}}e^{-\tl{\psi}}dV_{\omega_{\ts}}\\
        &=\int_{\sreg\setminus\pi(H)}\lara{B^{-1}_{\theta,\mu^*\tl{\psi},\mu^*\omega_{\ts}}\mu^*\tl{g}}{\mu^*\tl{g}}_{h,\mu^*\omega_{\ts}}e^{-\mu^*\tl{\psi}}dV_{\mu^*\omega_{\ts}}.
    \end{align*}
    By the biholomorphicity of $\mu$, $E|_{\sreg} \cong \sreg\times\bb{C}^r$, and $\mu^*\omega_{\ts}$ is also \kah on $\sreg$. 
    Applying the $\theta$-Nakano positivity of $h$ on the Stein space $\sreg\setminus\pi(H)$, there exists $u \in L^2_{n,q-1}(\sreg\setminus\pi(H),E;\mu^*\omega_{\ts},he^{-\mu^*\tl{\psi}})$ such that $\dbar u=\mu^*g$ on $\sreg\setminus\pi(H)$ and 
    \begin{align*}
        \int_{\ts\setminus(H\cup\exc)}|\pi^*u|^2_{\pi^*h,\omega_{\ts}}e^{-\tl{\psi}}dV_{\omega_{\ts}}&=\int_{\sreg\setminus\pi(H)}|u|^2_{h,\mu^*\omega_{\ts}}e^{-\mu^*\tl{\psi}}dV_{\mu^*\omega_{\ts}}\\
        &\leq\int_{\sreg\setminus\pi(H)}\lara{B^{-1}_{\theta,\mu^*\tl{\psi},\mu^*\omega_{\ts}}\mu^*\tl{g}}{\mu^*\tl{g}}_{h,\mu^*\omega_{\ts}}e^{-\mu^*\tl{\psi}}dV_{\mu^*\omega_{\ts}}\\
        &=\int_{\ts}\lara{B^{-1}_{\pi^*\theta,\tl{\psi},\omega_{\ts}}\tl{g}}{\tl{g}}_{\pi^*h,\omega_{\ts}}e^{-\tl{\psi}}dV_{\omega_{\ts}}.
    \end{align*}
    Since $\dbar\pi^*u=\tl{g}$ on $\ts\setminus(H\cup\exc)$, by the $\dbar$-extension Lemma \ref{Lemma: dbar-extension lemma}, there is $\tl{u}\in L^2_{n,q-1}(\ts,\pi^*E;\omega_{\ts},\pi^*he^{-\tl{\psi}})$ such that $\tl{u}|_{\ts\setminus(H\cup\exc)}=\pi^*u$, $\dbar\tl{u}=\tl{g}$ on $\ts$ and 
    \begin{align*}
        \int_{\ts}|\tl{u}|^2_{\pi^*h,\omega_{\ts}}e^{-\tl{\psi}}dV_{\omega_{\ts}}\leq\int_{\ts}\lara{B^{-1}_{\pi^*\theta,\tl{\psi},\omega_{\ts}}\tl{g}}{\tl{g}}_{\pi^*h,\omega_{\ts}}e^{-\tl{\psi}}dV_{\omega_{\ts}},
    \end{align*}
    which shows the $\pi^*\theta$-Nakano positivity of $\pi^*h$.

    Second, the converse is shown, which can be proved similarly. 
    Assume that $\pi^*h$ is $\pi^*\theta$-Nakano positive on $\tx$. 
    Consider any data consisting of 
    \begin{itemize}
        \item a Stein space coordinate $S\subset X$ admitting a trivialization $E|_{S}\cong S\times\bb{C}^r$,
        \item a \kah metric $\omega_S$ on $S$, 
        \item a smooth function $\psi$ on $S$ such that $\theta+\idd\psi>0$ on $S$,
        \item a $\dbar$-closed $g\in L^{2,loc}_{n,q}(S,E,he^{-\psi})$ satisfying 
        \begin{align*}
            \int_S\lara{B^{-1}_{\theta,\psi,\omega_S}g}{g}_{h,\omega_S}e^{-\psi}dV_{\omega_S}<+\infty.
        \end{align*}
    \end{itemize}
    Since $\sreg$ is Stein, $\ts\setminus\exc$ is also Stein, and we obtain 
    \begin{align*}
        \int_S\lara{B^{-1}_{\theta,\psi,\omega_S}g}{g}_{h,\omega_S}e^{-\psi}dV_{\omega_S}=\int_{\ts\setminus\exc}\lara{B^{-1}_{\pi^*\theta,\pi^*\psi,\pi^*\omega_S}\pi^*g}{\pi^*g}_{\pi^*h,\pi^*\omega_S}e^{-\pi^*\psi}dV_{\pi^*\omega_S}
    \end{align*}
    where $\ts:=\pi^{-1}(S)$ and $\pi^*\omega_S$ is also \kah on $\ts\setminus\exc$. 
    Applying the $\pi^*\theta$-Nakano positivity of $\pi^*h$ on the Stein space $\ts\setminus\exc$, there exists $\tl{u} \in L^2_{n,q-1}(\ts\setminus\exc,\pi^*E;\pi^*\omega_S,\pi^*he^{-\pi^*\psi})$ such that $\dbar\tl{u}=\pi^*g$ on $\ts\setminus\exc$ and 
    \begin{align*}
        \int_S|\mu^*\tl{u}|^2_{h,\omega_S}e^{-\psi}dV_{\omega_S}&=\int_{\ts\setminus\exc}|\tl{u}|^2_{\pi^*h,\pi^*\omega_S}e^{-\pi^*\psi}dV_{\pi^*\omega_S}\\
        &\leq\int_{\ts\setminus\exc}\lara{B^{-1}_{\pi^*\theta,\pi^*\psi,\pi^*\omega_S}\pi^*g}{\pi^*g}_{\pi^*h,\pi^*\omega_S}e^{-\pi^*\psi}dV_{\pi^*\omega_S}\\
        &=\int_S\lara{B^{-1}_{\theta,\psi,\omega_S}g}{g}_{h,\omega_S}e^{-\psi}dV_{\omega_S}.
    \end{align*}
    Therefore, setting $u:=\mu^*\tl{u}$, we have $u\in L^2_{n,q-1}(S,E;\omega_S,he^{-\psi})$, $\dbar u=g$ on $\sreg$, and 
    \begin{align*}
        \int_S|u|^2_{h,\omega_S}e^{-\psi}dV_{\omega_S}\leq\int_S\lara{B^{-1}_{\theta,\psi,\omega_S}g}{g}_{h,\omega_S}e^{-\psi}dV_{\omega_S},
    \end{align*}
    which shows the $\theta$-Nakano positivity of $h$.
\end{proof}

\begin{theorem}\label{Theorem: h>Grif then h det h>Nak}
    If $h$ is Griffiths positive and uniformly positive definite (resp. a.e.\! Griffiths semi-positive) on $X$, then the singular Hermitian metric $h\otimes\det h$ on $E\otimes\det E$ is Nakano positive (resp. Nakano semi-positive) on $X$.  
\end{theorem}

\begin{proof}
    Assuming that $h$ is a.e.\! Griffiths semi-positive on $X$, Corollary \ref{Corollary: a.e. Grif semi-posi of h and pi*h} implies that $\pi^*h$ is a.e.\! Griffiths semi-positive on $\tx$. 
    Hence there exists a Griffiths semi-positive metric $\cal{H}$ on $\pi^*E$ that coincides with $\pi^*h$ a.e. on $\tx$ (see Remark \ref{Remark: Grif>0 then a.e. Grif semi-posi & Remark in Wat24a}, \cite[Remark 5.12]{Wat24a}). 
    By \cite[Theorem 1.3]{Ina22}, $\cal{H}\otimes\det\cal{H}$ is Nakano semi-positive on $\tx$. Since $\cal{H}$ coincides a.e. with $\pi^*h$, the metric $\pi^*h\otimes\det\pi^*h$ is also Nakano semi-positive on $\tx$. 
    Therefore, Theorem \ref{Theorem: h Nak iff pi*h Nak} implies that $h\otimes\det h$ is Nakano semi-positive on $X$.

    Assume that $h$ is Griffiths positive on $X$. By Proposition \ref{Proposition: Grif>0 then εω-Grif>0}, $h$ is $2\varepsilon\omega$-Griffiths positive on $X$ for some smooth positive function $\varepsilon:X\longrightarrow\bb{R}_{>0}$ and Hermitian metric $\omega$ on $X$. 
    Since $\pi|_{\tx\setminus\exc}:\tx\setminus\exc\overset{\simeq}{\longrightarrow}\reg$ is biholomorphic, $\pi^*h$ is $\pi^*(2\varepsilon\omega)$-Griffiths positive on $\tx\setminus\exc$. 
    Arguing as in the proof of Theorem \ref{Theorem: h>Grif then iO_det h>0 as currents} and using Remark \ref{Remark: Grif>0 then a.e. Grif semi-posi & Remark in Wat24a}, there exists a singular Hermitian metric $\cal{H}$ on $\pi^*E$ that is Griffiths semi-positive and $\pi^*(\varepsilon\omega)$-Griffiths positive on $\tx\setminus\exc$, and coincides with $\pi^*h$ a.e. on $\tx\setminus\exc$.
    Then \cite[Theorem 3.6]{Ina22} implies that $\cal{H}\otimes\det\cal{H}$ is $(r+1)\pi^*(\varepsilon\omega)$-Nakano positive on $\tx\setminus\exc$, and therefore $\pi^*h\otimes\det\pi^*h$ is also $(r+1)\pi^*(\varepsilon\omega)$-Nakano positive on $\tx\setminus\exc$.
    Arguing as in the proof of Theorem \ref{Theorem: h Nak iff pi*h Nak}, the $(r+1)\pi^*(\varepsilon\omega)$-Nakano positivity of $\pi^*h\otimes\det\pi^*h$ extends from $\tx\setminus\exc$ to $\tx$.  
    Hence, Theorem \ref{Theorem: h Nak iff pi*h Nak} shows that $h\otimes\det h$ is $(r+1)\varepsilon\omega$-Nakano positive on $X$.
\end{proof}

From Corollary \ref{Corollary: a.e. Grif semi-posi then uniformly positive definite}, the conclusion of Theorem \ref{Theorem: h>Grif then h det h>Nak} remains valid if the uniform positive definiteness of $h$ is replaced by the a.e.\! Griffiths semi-positivity.

In general, near the singular locus of $X$, Griffiths positivity of $h$ carries only information in the sense of currents, and it is unclear whether a negative integral current supported in $\exc$ may appear in the curvature of $\pi^*h$. 
In contrast, the uniform positive definiteness of $h$ prevents the appearance of such negative integration currents and makes it possible to apply the $\dbar$-extension Lemma \ref{Lemma: dbar-extension lemma}. 
As a consequence, Nakano positivity is not affected by codimension-one phenomena and can be extended across $\exc$.
Of course, as 
in the proof of Theorem \ref{Theorem: h Gri iff pi*h Gri}, the presence of negative integral currents can be excluded under the assumption of a.e.\! Griffiths semi-positivity.

This is more clear in the case of line bundles. Let $U\subset X$ be an open subset with $U\cap\xs\ne\emptyset$ and let $\varphi\in L^1_{loc}(U)$ satisfy $\idd\varphi\geq0$ on $U$ in the sense of currents. 
It is unclear whether $\idd\pi^*\varphi$ is well-defined in the sense of currents and whether it is positive on $\pi^{-1}(U)$, i.e. whether $\idd\pi^*\varphi\geq0$ holds on $\pi^{-1}(U)$ (see Remark \ref{Remark: unclear of idd pi*phi>0 current}), or equivalently whether $\pi^*\varphi$ coincides a.e. on $\pi^{-1}(U)$ with a plurisubharmonic function (see Proposition \ref{Proposition: idd>0 iff a.e. psh}), since $\pi^*\varphi$ need not be locally integrable in general.
The pull-back is stable in the plurisubharmonic case, and this is where Theorem \ref{Theorem: idd>0 and a.e. psh} plays a key role.

\begin{theorem}\label{Theorem: h>Gri0 & h_L>0 then h*h_L>Gri0}
    Let $\omega$ and $\gamma$ be Hermitian metrics on $X$ and $\theta$ be a continuous $(1,1)$-form on $X$. 
    We assume that $h_L$ is singular positive on $X$, then there exists a smooth positive function $\varepsilon:X\longrightarrow\bb{R}_{>0}$ such that $\iO{L,h}\geq2\varepsilon\omega$ on $X$ in the sense of currents. 
    For a singular Hermitian metric $h\otimes h_L$ on $E\otimes L$, we obtain the following.
    \begin{itemize}
        \item [$(a)$] for a smooth positive function $\tau$ on $X$, if $h$ is $\tau\gamma$-Griffiths positive and a.e.\! Griffiths semi-positive on $X$, then $h\otimes h_L$ is $(\tau\gamma+\varepsilon\omega)$-Griffiths positive on $X$.
        \item [$(b)$] if $h$ is $\theta$-Nakano positive on $X$, then $h\otimes h_L$ is $(\theta+\varepsilon\omega)$-Nakano positive on $X$. 
    \end{itemize}
\end{theorem}

\begin{proof}
    $(a)$ For simplicity, write $\tl{\gamma}:=\pi^*(\tau\gamma)$ and $\tl{\omega}:=\pi^*(\varepsilon\omega)$.
    By the proof of Theorem \ref{Theorem: h Gri iff pi*h Gri} and Remark \ref{Remark: condition of a.e. Grif in Thm of Grif posi}, it suffices to show that $\pi^*h\otimes\pi^*h_L$ is $(\tl{\gamma}+\tl{\omega})$-Griffiths positive on $\tx\setminus\exc$, where $\pi^*h$ is $\tl{\gamma}$-Griffiths positive on $\tx$. 
    Here, for any $x\in\tx\setminus\exc$, there exists an open neighborhood $\tl{V}\subset\tx\setminus\exc$ of $x$ such that $\pi^*L$ is trivializable on $\tl{V}$, we can write $\pi^*h_L=e^{-\pi^*\varphi}$ on $\tl{V}$ for some weight $\varphi$ of $h_L$ satisfying $\idd\varphi\geq\varepsilon\omega$ on $\pi(\tl{V})$, and $\idd|u|^2_{\pi^*h^*}\geq|u|^2_{\pi^*h^*}\tl{\gamma}$ holds on $\tl{V}$ in the sense of currents for any section $u\in H^0(\tl{V},\pi^*E^*)$.
    Furthermore, there exist a small Stein open neighborhood $\tu \Subset \tl{V}$ of $x$ and a small number $0 < \vartheta \ll 1$ such that we have $\tl{\gamma}\geq\delta_\tau\idd|z|^2\geq(1-\vartheta)\tl{\gamma}$, $\tl{\omega}/2\geq\vartheta\tl{\gamma}$ and $2\tl{\omega}\geq\delta_\varepsilon\idd|z|^2\geq3\tl{\omega}/2$ on $\tu$, for some coordinate system $(z_1,\ldots,z_N)$ on $\tu$ and constants $\delta_\tau,\delta_\varepsilon>0$ depending on it.
    Indeed, we first take a neighborhood satisfying the first two conditions, and then consider its intersection with a neighborhood satisfying the last condition.
    Thus, $\idd|u|^2_{\pi^*h^*}\geq\delta_\tau|u|^2_{\pi^*h^*}\idd|z|^2$ for any section $u\in H^0(\tu,\pi^*E^*)=H^0(\tu,\pi^*E^*\otimes\pi^*L^*)$. By the proof of \cite[Proposition 5.5 (d)]{Wat24a}, it follows that $\idd\log|u|^2_{\pi^*h^*}\geq\delta_\tau\idd|z|^2$ on $\tu$ in the sense of currents.
    Hence, 
    \begin{align*}
        \idd\log|u|^2_{\pi^*h^*\otimes\pi^*h_L}&=\idd\log|u|^2_{\pi^*h^*}+\idd\pi^*\varphi\geq\delta_\tau\idd|z|^2+2\tl{\omega}\\
        &\geq(\delta_\tau+\delta_\varepsilon)\idd|z|^2\geq(1-\vartheta)\tl{\gamma}+\frac{3}{2}\tl{\omega}\geq\tl{\gamma}+\tl{\omega}
    \end{align*}
    holds on $\tu$ in the sense of currents. From \cite[Proposition 5.5]{Wat24a}, it follows that 
    \begin{align*}
        \idd|u|^2_{\pi^*h^*\otimes\pi^*h_L}\geq(\delta_\tau+\delta_\varepsilon)|u|^2_{\pi^*h^*\otimes\pi^*h_L}\idd|z|^2\geq|u|^2_{\pi^*h^*\otimes\pi^*h_L}(\tl{\gamma}+\tl{\omega})
    \end{align*}
    holds on $\tu$ in the sense of currents, which represents the $(\tl{\gamma}+\tl{\omega})$-Griffiths positivity of $\pi^*h\otimes\pi^*h_L$ on $\tl{X}\setminus\exc$.


    $(b)$ To show the $(\theta+\varepsilon\omega)$-Nakano positivity of $h\otimes h_L$, consider any data consisting of 
    \begin{itemize}
        \item a Stein space coordinate $S\subset X$ admitting a trivialization $E\otimes L|_S\cong S\times\bb{C}^r$,
        \item a \kah metric $\omega_S$ on $S$,
        \item a smooth function $\psi$ on $S$ such that $\theta+\varepsilon\omega+\idd\psi>0$ on $S$,
        \item a $\dbar$-closed $g\in L^{2,loc}_{n,q}(S,E\otimes L,h\otimes h_Le^{-\psi})$ satisfying 
        \begin{align*}
            \int_S\lara{B^{-1}_{\theta+\varepsilon\omega,\psi,\omega_S}g}{g}_{h\otimes h_L,\omega_S}e^{-\psi}dV_{\omega_S}<+\infty.
        \end{align*}
    \end{itemize}
    Here, $\sreg$ is also Stein, and there exists a hypersurface $H$ such that $L$ is trivialized on $\sreg\setminus H$, which is also Stein, and we can write $h_L = e^{-\varphi}$ on $\sreg\setminus H$.
    For the Stein submanifold $\sreg\setminus H$, by using an approximate identity $\{\rho_\nu\}_{\nu\in\bb{N}}$, 
    there exist an increasing sequence of Stein open subsets $\{\cal{S}_\nu\}_{\nu\in\bb{N}}$ exhausting $\sreg\setminus H$ and smooth strictly plurisubharmonic functions $\varphi_\nu := \varphi \ast \rho_\nu$ on $\cal{S}_\nu$ such that $\varphi_\nu$ converges decreasingly to $\varphi$ a.e. and satisfies $\idd\varphi_\nu\geq\varepsilon\omega$ on $\cal{S}_\nu$. 
    Since the smooth function $\varphi_\nu+\psi$ on $\cal{S}_\nu$ satisfies $\theta+\idd(\varphi_\nu+\psi)\geq\theta+\varepsilon\omega+\idd\psi>0$ and 
    $B_{\theta,\varphi_\nu+\psi,\omega_S}\geq B_{\theta+\varepsilon\omega,\psi,\omega_S}$, by applying the $\theta$-Nakano positivity of $h$ on the Stein subset $\cal{S}_\nu$, 
    there exists $u_\nu\in L^2_{n,q-1}(\cal{S}_\nu,E;\omega_S,he^{-\varphi_\nu-\psi})$ such that $\dbar u_\nu=g$ on $\cal{S}_\nu$ and 
    \begin{align*}
        \int_{\cal{S}_\nu}|u_\nu|^2_{h,\omega_S}e^{-\varphi_\nu-\psi}dV_{\omega_S}&\leq\int_{\cal{S}_\nu}\lara{B_{\theta,\varphi_\nu+\psi,\omega_S}^{-1}g}{g}_{h,\omega_S}e^{-\varphi_\nu-\psi}dV_{\omega_S}\\
        &\leq\int_{\cal{S}_\nu}\lara{B_{\theta+\varepsilon\omega,\psi,\omega_S}^{-1}g}{g}_{h,\omega_S}e^{-\varphi_\nu-\psi}dV_{\omega_S}\\
        &\leq\int_{\cal{S}_\nu}\lara{B_{\theta+\varepsilon\omega,\psi,\omega_S}^{-1}g}{g}_{h,\omega_S}e^{-\varphi-\psi}dV_{\omega_S}\\
        &\leq\int_S\lara{B^{-1}_{\theta+\varepsilon\omega,\psi,\omega_S}g}{g}_{h\otimes h_L,\omega_S}e^{-\psi}dV_{\omega_S}.
    \end{align*}
    Since the sequence of smooth functions $\{e^{-\varphi_\nu}\}_{\nu\in\bb{N}}$ is increasing and converges to $h = e^{-\varphi}$ a.e., \cite[Lemma 3.18]{Wat25a} implies the existence of a solution $u\in L^2_{n,q-1}(\sreg\setminus H,E;\omega_S,he^{-\varphi-\psi})=L^2_{n,q}(S,E\otimes L;\omega_S,h\otimes h_Le^{-\psi})$ as the weak limit of a convergent subsequence, satisfying the $\dbar$-equation $\dbar u=g$ on $\sreg\setminus H$ and the $L^2$-estimates 
    \begin{align*}
        \int_S|u|^2_{h\otimes h_L,\omega_S}e^{-\psi}dV_{\omega_S}=\int_{\sreg\setminus H}|u|^2_{h,\omega_S}e^{-\varphi-\psi}dV_{\omega_S}\leq\int_S\lara{B^{-1}_{\theta+\varepsilon\omega,\psi,\omega_S}g}{g}_{h\otimes h_L,\omega_S}e^{-\psi}dV_{\omega_S}.
    \end{align*}
    By the $\dbar$-extension Lemma \ref{Lemma: dbar-extension lemma}, we have $\dbar u=g$ on $\sreg$, which shows the $(\theta+\varepsilon\omega)$-Nakano positivity of $h\otimes h_L$.
\end{proof}

The following result is proved in the same way as above.

\begin{corollary}\label{Corollary: h>Gri0 & h_L>0 then h*h_L>Gri0}
    Assume that $h_L$ is singular positive on $X$, and that $\omega$ and $\varepsilon:X\longrightarrow\bb{R}_{>0}$ as in Theorem \ref{Theorem: h>Gri0 & h_L>0 then h*h_L>Gri0}, then the following result holds.
    \begin{itemize}
        \item [$(a)$] if $h$ is a.e.\! Griffiths semi-positive, then $h\otimes h_L$ is $\varepsilon\omega$-Griffiths positive on $X$.
        \item [$(b)$] if $h$ is Nakano semi-positive on $X$, then $h\otimes h_L$ is $\varepsilon\omega$-Nakano positive on $X$. 
    \end{itemize}
    Furthermore, the same conclusion remains valid with $\varepsilon\equiv0$ when $h$ is singular semi-positive instead of singular positive. 
    In particular, part $(a)$ implies that $h\otimes h_L$ is a.e.\! Griffiths semi-positive on $X$. 
\end{corollary}

Finally, the following theorem is obtained using the Negativity lemma and the strong openness property; it is used to remove the \kah assumption in vanishing theorems on (relatively) compact spaces and to establish vanishing for higher direct image sheaves.

\begin{theorem}\label{Theorem: h>Nak then pi*he-psi>Nak & E(pi*h)=E(pi*he-psi)}
    Let $V$ be a relatively compact open subset of $X$, and set $\tl{V}:=\pi^{-1}(V)$.
    If $h$ is Nakano positive (resp. Griffiths positive and a.e.\! Griffiths semi-positive) on an open neighborhood of $\overline{V}$, then there exist a quasi-plurisubharmonic function $\psi:\tl{V}\longrightarrow[-\infty,+\infty)$, which is smooth on $\tl{V}\setminus\exc$,
    and a small number $\varepsilon_V>0$ such that the singular Hermitian metric $\pi^*he^{-\varepsilon\psi}$ on $\pi^*E$ is also Nakano positive (resp. Griffiths positive) on $\tl{V}$ for any $0<\varepsilon<\varepsilon_V$. 
    
    Furthermore, by additionally assuming that $h$ is a.e.\! Griffiths semi-positive, there exists a small number $\varepsilon_0$ with $0<\varepsilon_0<\varepsilon_V$ such that we have $\scr{E}(\pi^*h)=\scr{E}(\pi^*he^{-\varepsilon\psi})\, ($resp. $\scr{E}(\pi^*h\otimes\det \pi^*h)=\scr{E}(\pi^*h\otimes\det \pi^*he^{-(1+r)\varepsilon\psi}))$ on $\tl{V}$ for any $0<\varepsilon\leq\varepsilon_0$.
\end{theorem}

\begin{proof}
    Take a relatively compact open neighborhood $W$ of $V$ with $V\Subset W$ such that $h$ is Nakano positive (resp. Griffiths positive and a.e.\! Griffiths semi-positive) on $W$. 
      
    First, assume that $h$ is Nakano positive on $W$.
    It is shown that there exists a sufficiently small $\varepsilon_V > 0$ such that $\pi^*he^{-\varepsilon\psi}$ is Nakano positive on $\tl{W}:=\pi^{-1}(W)$ for every $0 < \varepsilon < \varepsilon_V$. 
    By Theorem \ref{Theorem: h Nak iff pi*h Nak}, there exists a positive continuous $(1,1)$-form $\theta > 0$ on $W$ such that $\pi^*h$ is $\pi^*\theta$-Nakano positive on $\tl{W}$.  
    By applying Lemma \ref{Lemma: key lemma} to the relatively compact subset $W$, there exist a quasi-plurisubharmonic function $\psi : \tl{W}\longrightarrow[-\infty,+\infty)$ which is smooth on $\tl{W}\setminus\exc$ and a number $\varepsilon_V>0$ such that $\pi^*\theta+\varepsilon\idd\psi$ is strictly positive on $\tl{W}$ in the sense of currents for any $0<\varepsilon<\varepsilon_V$; 
    that is, there exists a Hermitian metric $\tl{\gamma}_\varepsilon > 0$ on $\tl{W}$ such that $\pi^*\theta+\varepsilon\idd\psi\geq\tl{\gamma}_\varepsilon$ on $\tl{W}$ in the sense of currents. 
    To show the $\tl{\gamma}_\varepsilon$-Nakano positivity of $\pi^*he^{-\varepsilon\psi}$ on $\tl{W}$, consider any data consisting of 
    \begin{itemize}
        \item a Stein coordinate $\tl{S}\subset\tl{W}$ admitting a trivialization $\pi^*E|_{\tl{S}}\cong\tl{S}\times\bb{C}^r$,
        \item a \kah metric $\omega_{\tl{S}}$ on $\tl{S}$,
        \item a smooth function $\varPsi$ on $\tl{S}$ such that $\tl{\gamma}_\varepsilon+\idd\varPsi>0$ on $\tl{S}$,
        \item a $\dbar$-closed $g\in L^{2,loc}_{n,q}(\tl{S},\pi^*E,\pi^*he^{-\varPsi})$ satisfying 
        \begin{align*}
            \int_{\tl{S}}\lara{B^{-1}_{\tl{\gamma}_\varepsilon,\varPsi,\omega_{\tl{S}}}g}{g}_{\pi^*h,\omega_{\tl{S}}}e^{-\varPsi}dV_{\omega_{\tl{S}}}<+\infty.
        \end{align*}
    \end{itemize}
    Since $\pi^*\theta+\idd(\varepsilon\psi+\varPsi)\geq\idd\varPsi+\tl{\gamma}_\varepsilon>0$ on $\tl{S}\setminus\exc$, we have $B_{\pi^*\theta,\varepsilon\psi+\varPsi,\omega_{\tl{S}}}\geq B_{\tl{\gamma}_\varepsilon,\varPsi,\omega_{\tl{S}}}$ on $\tl{S}\setminus\exc$, and 
    by applying the $\pi^*\theta$-Nakano positivity of $\pi^*h$ on the Stein space $\tl{S}\setminus\exc$, there exists $u \in L^2_{n,q-1}(\tl{S}\setminus\exc,\pi^*E;\omega_{\tl{S}},\pi^*he^{-\varepsilon\psi-\varPsi})$ such that $\dbar u=g$ on $\tl{S}\setminus \exc$ and 
    \begin{align*}
        \int_{\tl{S}}|u|^2_{\pi^*he^{-\varepsilon\psi},\omega_{\tl{S}}}e^{-\varPsi}dV_{\omega_{\tl{S}}}&=\int_{\tl{S}\setminus\exc}|u|^2_{\pi^*h,\omega_{\tl{S}}}e^{-\varepsilon\psi-\varPsi}dV_{\omega_{\tl{S}}}\\
        &\leq\int_{\tl{S}\setminus\exc}\lara{B_{\pi^*\theta,\varepsilon\psi+\varPsi,\omega_{\tl{S}}}^{-1}g}{g}_{\pi^*h,\omega_{\tl{S}}}e^{-\varepsilon\psi-\varPsi}dV_{\omega_{\tl{S}}}\\
        &\leq\int_{\tl{S}\setminus\exc}\lara{B_{\tl{\gamma}_\varepsilon,\varPsi,\omega_{\tl{S}}}^{-1}g}{g}_{\pi^*h,\omega_{\tl{S}}}e^{-\varepsilon\psi-\varPsi}dV_{\omega_{\tl{S}}}\\
        &=\int_{\tl{S}}\lara{B^{-1}_{\tl{\gamma}_\varepsilon,\varPsi,\omega_{\tl{S}}}g}{g}_{\pi^*he^{-\varepsilon\psi},\omega_{\tl{S}}}e^{-\varPsi}dV_{\omega_{\tl{S}}}.
    \end{align*}
    Hence, we have $u \in L^2_{n,q-1}(\tl{S},\pi^*E;\omega_{\tl{S}},\pi^*he^{-\varepsilon\psi-\varPsi})$ and $\dbar u=g$ on $\tl{S}$ by the $\dbar$-extension Lemma \ref{Lemma: dbar-extension lemma}, which shows $\tl{\gamma}_\varepsilon$-Nakano positivity of $\pi^*he^{-\varepsilon\psi}$ on $\tl{W}$.

    Here, for any $0\leq\varepsilon<\varepsilon_V$, the Nakano (semi)-positivity of $\pi^*he^{-\varepsilon\psi}$ implies that the $L^2$-subsheaf $\scr{E}(\pi^*he^{-\varepsilon\psi})$ is coherent.
    By the strong openness property (= Theorem \ref{Theorem: strong openness property} and Corollary \ref{Corollary: strong openness property}), we obtain $\scr{E}(\pi^*h)=\bigcup_{0<\varepsilon<\varepsilon_V}\scr{E}(\pi^*he^{-\varepsilon\psi})$ on $\tl{W}$.
    Therefore, by the relative compact-ness of $\tl{V}$ and the strong Noetherian property of coherent sheaves (see \cite[Chapter\,II, (3.22)]{Dem-book}), there exists $0<\varepsilon_0<\varepsilon_V$ such that $\bigcup_{0<\varepsilon<\varepsilon_V}\scr{E}(\pi^*he^{-\varepsilon\psi})=\scr{E}(\pi^*he^{-\varepsilon_0\psi})$ on $\tl{V}$.

    Second, assume that $h$ is Griffiths positive and a.e.\! Griffiths semi-positive on $W$.
    It follows from Proposition \ref{Proposition: Grif>0 then εω-Grif>0}, Corollary \ref{Corollary: a.e. Grif semi-posi of h and pi*h}, Theorem \ref{Theorem: h Gri iff pi*h Gri} and Remark \ref{Remark: condition of a.e. Grif in Thm of Grif posi} that there exist a smooth positive function $\tau:W\longrightarrow\bb{R}_{>0}$ and a Hermitian metric $\omega$ on $X$ such that $\pi^*h$ is $\pi^*(2\tau\omega)$-Griffiths positive and a.e.\! Griffiths semi-positive on $\tl{W}$. 
    Applying Lemma \ref{Lemma: key lemma} to the relatively compact subset $W$, there exist a quasi-plurisubharmonic function $\psi : \tl{W}\longrightarrow[-\infty,+\infty)$ which is smooth on $\tl{W}\setminus\exc$, and a number $\varepsilon_V>0$ such that the following holds: for any $0<\varepsilon<\varepsilon_V$, there exists a Hermitian metric $\tl{\gamma}_\varepsilon>0$ on $\tl{W}$ such that $\pi^*(\tau\omega)+\varepsilon\idd\psi\geq2\tl{\gamma}_\varepsilon$ on $\tl{W}$ in the sense of currents.
    By the proof of Theorem \ref{Theorem: h Gri iff pi*h Gri} and Remark \ref{Remark: condition of a.e. Grif in Thm of Grif posi}, it suffices to show that $\pi^*he^{-\varepsilon\psi}$ is $\tl{\gamma}_\varepsilon$-Griffiths positive on $\tl{W}\setminus\exc$, and this positivity extends to $\tl{W}$. 
    For any point $x\in \tl{W}\setminus\exc$, there exist a number $\delta>0$ and an open neighborhood $\cal{U}\subset\tl{W}\setminus\exc$ of $x$ such that $2\pi^*(\tau\omega)\geq\delta\idd|z|^2\geq\pi^*(\tau\omega)$ on $\cal{U}$. Thus, we have $\idd|u|^2_{\pi^*h^*}\geq\delta|u|^2_{\pi^*h^*}\idd|z|^2$ on $\cal{U}$ in the sense of currents, for any $u\in H^0(\cal{U},\pi^*E^*)$.
    By the proof of \cite[Proposition 5.5]{Wat24a}, the function $\log|u|^2_{\pi^*h^*}-\delta|z|^2$ coincides a.e. on $\cal{U}$ with a plurisubharmonic function defined on $\cal{U}$; that is, $\idd\log|u|^2_{\pi^*h^*}\geq\delta\idd|z|^2$ on $\cal{U}$ in the sense of currents, then 
    \begin{align*}
        \idd\log|u|^2_{\pi^*h^*e^{\varepsilon\psi}}\geq\delta\idd\log|u|^2_{\pi^*h^*}+\varepsilon\idd\psi\geq\idd(\varepsilon\psi+\delta|z|^2)\geq2\tl{\gamma}_\varepsilon
    \end{align*}
    holds on $\cal{U}$ in the sense of currents. 
    Furthermore, there exist $\tl{\delta}>0$ and an open neighborhood $\cal{V}\subset\cal{U}$ of $x$ such that $2\tl{\gamma}_\varepsilon\geq\tl{\delta}\idd|z|^2\geq\tl{\gamma}_\varepsilon$ on $\cal{V}$. 
    Hence, \cite[Proposition 5.5]{Wat24a} implies that $\idd|u|^2_{\pi^*h^*e^{\varepsilon\psi}}\geq\tl{\delta}|u|^2_{\pi^*h^*e^{\varepsilon\psi}}\idd|z|^2\geq|u|^2_{\pi^*h^*e^{\varepsilon\psi}}\tl{\gamma}_\varepsilon$ on $\cal{V}$ in the sense of currents, which yields the $\tl{\gamma}_\varepsilon$-Griffiths positivity of $\pi^*he^{-\varepsilon\psi}$ on $\tl{W}\setminus\exc$. 

    Here, by Theorem \ref{Theorem: h>Grif then h det h>Nak}, $\pi^*h\otimes\det\pi^*h e^{-(r+1)\varepsilon\psi}$ is also Nakano (semi)-positive for any $0\leq\varepsilon<\varepsilon_V$, and hence the same argument applies to the $L^2$-sheaf $\scr{E}(\pi^*h\otimes\det\pi^*h)$.
\end{proof}

The following immediately follows from Proposition \ref{Proposition: smooth Grif/Nak>0 on X_reg then also >0 on X}, and finally the proof of Proposition \ref{Proposition: smooth Grif/Nak>0 on X_reg then also >0 on X} is given.

\begin{remark}\label{Remark: smooth case relative positivity and O(pi*E)=E(pi*h)}
    Theorem \ref{Theorem: h>Nak then pi*he-psi>Nak & E(pi*h)=E(pi*he-psi)} also holds if $h$ is a smooth Hermitian metric on $E$ and is Nakano $($resp. Griffiths$)$ positive on $V$; 
    in the case, $\cal{O}_{\tx}(\pi^*\!E)=\scr{E}(\pi^*\!h)=\scr{E}(\pi^*\!h\,e^{-\varepsilon_0\psi})$ $($resp. $\cal{O}_{\tx}(\pi^*E\otimes\det \pi^*E)=\scr{E}(\pi^*h\otimes\det \pi^*h)=\scr{E}(\pi^*h\otimes\det \pi^*\!h\,e^{-(1+r)\varepsilon_0\psi}))$ on $\tl{V}$.
\end{remark}

\begin{proof}[Proof of Proposition \ref{Proposition: smooth Grif/Nak>0 on X_reg then also >0 on X}]
    $(i)$ Assume that $h$ is Griffiths semi-positive on $\reg$, then $\pi^{*}h$ is also Griffiths semi-positive on $\tx\setminus\exc$. 
    Using Lemma \ref{Lemma: two psh} together with the boundedness arising from the smoothness of $h$ and the unique extension property of plurisubharmonic functions, it follows, as in the proof of Theorem \ref{Theorem: Grif semi-positivity of h and pi^*h}, that $\pi^{*}h$ is Griffiths semi-positive on $\tx$. 
    In the proof of Theorem \ref{Theorem: Grif semi-positivity of h and pi^*h}, replacing Theorem \ref{Theorem: idd>0 and a.e. psh} with \cite[Corollary 1.12]{Dem85} yields the Griffiths semi-positivity of $h$ without assuming that $X$ is locally irreducible.

    $(ii)$ Assume that $h$ is Griffiths positive on $\reg$ and uniformly Griffiths positive in a neighborhood of $\xs$. 
    By Proposition \ref{Proposition: Grif>0 then εω-Grif>0}, there exist a Hermitian metric $\omega$ on $X$ and a smooth positive function $\varepsilon:X\longrightarrow\bb{R}_{>0}$ such that $\iO{E,h}\geq_{Grif}\varepsilon\omega\oid{E}$ on $\reg$ and $\iO{\pi^*E,\pi^*h}\geq_{Grif}\pi^*(\varepsilon\omega)\oid{\pi^*E}\geq_{Grif}0$ on $\tx\setminus\exc$.
    In particular, $\pi^{*}h$ is Griffiths semi-positive on $\tx$ by $(i)$. Hence, $\pi^{*}h$ is $\varepsilon\omega$-Griffiths positive on $\tx$ by the proof of Theorem \ref{Theorem: h Gri iff pi*h Gri} and Remark \ref{Remark: condition of a.e. Grif in Thm of Grif posi}, and $h$ is the Griffiths positive on $X$ by Theorem \ref{Theorem: h Gri iff pi*h Gri}.

    $(iii)$ If $h$ is Nakano semi-positive on $\reg$, then $\pi^{*}h$ is also Nakano semi-positive on $\tx\setminus\exc$, and arguing as in the proof of Theorem \ref{Theorem: h Nak iff pi*h Nak}, the semi-positivity of $\pi^{*}h$ extends to $\tx$, which yields the Nakano semi-positivity of $h$ on $X$.

    $(iv)$ Assume that $h$ is Nakano positive on $\reg$ and uniformly Nakano positive in a neighborhood of $\xs$. 
    Arguing as in Proposition \ref{Proposition: Grif>0 then εω-Grif>0}, there exists a positive smooth $(1,1)$-form $\theta$ on $X$ such that $h$ is $\theta$-Nakano positive on $\reg$. 
    Arguing as in the proof of Theorem \ref{Theorem: h Nak iff pi*h Nak}, the $\pi^{*}\theta$-Nakano positivity of $\pi^{*}h$ on $\tx\setminus\exc$ extends to $\tx$, which yields the $\theta$-Nakano positivity of $h$ on $X$.
\end{proof}

\section{Global $L^2$-estimates on the regular locus $\reg$}

In this section, we provide a proof of Theorem \ref{Theorem: global L2-estimates of Nakano}.
Here, a function $\varPsi:X\longrightarrow[-\infty,+\infty)$ on a complex space $X$ is \textit{exhaustion} if all sublevel sets $X_c:=\{x\in X\mid\varPsi(x)<c\}$, $\forall\,c\in\bb{R}$, are relatively compact. 
A complex space is said to be \textit{weakly} \textit{pseudoconvex} if there exists a smooth exhaustion plurisubharmonic function.
To obtain global $L^2$-estimates on weakly pseudoconvex complex spaces, the following theorem providing a certain Stein exhaustion sequence plays a crucial role.

\begin{theorem}[{= Theorem \ref{Theorem: exhaustion a.e. Stein covering}}]\label{Theorem: exhaustion a.e. Stein covering in subsection}
    Let $X$ be a weakly pseudoconvex complex space. Assume that there exist a holomorphic line bundle $L\longrightarrow X$ and a singular Hermitian metric $h$ on $L$ such that $\iO{\pi^*L,\pi^*h}\geq\pi^*(\varepsilon\omega)$ holds on $\tx$ in the sense of currents for some smooth positive function $\varepsilon:X\longrightarrow\bb{R}_{>0}$ and some Hermitian metric $\omega$ on $X$, where $\pi:\tx\longrightarrow X$ is canonical desingularization of $X$. 
    Then there exist an increasing sequence of real numbers $\{c_j\}_{j\in\bb{N}}$ with $c_1>\inf_X\varPsi$ and analytic subsets $A_j$ of each $X_{c_j}$ such that each open subset $(X_{c_j})_{reg}\setminus A_j=(X_{c_j}\setminus A_j)\cap\reg$ is Stein manifold. 

    In particular, the assumption holds whenever $X$ admits a singular positive line bundle.
\end{theorem}

\begin{proof}
    For any constant $c>\inf_X\varPsi$, it suffices to show that there exists an analytic subset $A\subset X_c$ such that $(X_c)_{reg}\setminus A$ is a Stein manifold. 
    Here, the sublevel set $\tx_c:=\pi^{-1}(X_c)=\{x\in \tx\mid\pi^*\varPsi(x)<c\}$ is also weakly pseudoconvex. 
    Let $\tau>0$ be a sufficiently small constant. Applying Lemma \ref{Lemma: key lemma} to $\tx_{c+\tau}$, there exists a quasi-plurisubharmonic function $\psi:\tx_{c+\tau}\longrightarrow[-\infty,+\infty)$ which is smooth on $\tx_{c+\tau}\setminus\exc$ and whose smooth part $\psi_{\mathrm{sm}}$ has a Levi form compensating for the loss of positivity of $\pi^*(\varepsilon\omega)$ along $\exc$.
    For $\delta>0$, define a new singular Hermitian metric on $\pi^*L|_{\tx_{c+\tau}}$ by $\tl{h}:=\pi^*h\cdot e^{-2\delta\psi}$.
    Since $\tx_{c+\tau}$ is relatively compact, we may choose $\delta>0$ sufficiently small such that $\tl{h}$ is singular positive on $\tx_{c+\tau}$. Equivalently,
    \begin{align*}
        \iO{\pi^*L,\tl{h}}=\iO{\pi^*L,\pi^*h}+2\delta\idd\psi\geq\pi^*(\varepsilon\omega)+2\delta\idd\psi\geq\gamma_{\delta}
    \end{align*}
    \vspace{-5mm}

    \noindent
    holds on $\tx_{c+\tau}$ in the sense of currents for some Hermitian metric $\gamma_\delta>0$ on $\tx$.

    For the singular positive Hermitian metric $\tl{h}$ defined on $\pi^*L|_{\tx_{c+\tau}}$, we consider one approximation $\cal{H}:=\tl{h}_{\nu_0}$ obtained from a refined Demailly approximation $\{\tl{h}_\nu\}_{\nu\in\mathbb N}$ preserving the multiplier ideal sheaves and producing algebraic singularities (see \cite[Theorem 3.2]{Wat24b}). 
    By blowing up the analytic singular locus $Z$ of $\cal{H}$, there exists a proper modification $\mu:\widehat{X}_c\longrightarrow\tx_c$, given by a composition of finitely many blow-ups with smooth center, such that the $\mu$-exceptional set $E_Z:=\mu^{-1}(Z)$ is simple normal crossing and $\mu|_{\widehat{X}_c\setminus E_Z}:\widehat{X}_c\setminus E_Z\overset{\simeq}{\longrightarrow}\tx_c\setminus Z$ is biholomorphic. 
    Furthermore, since the approximation has algebraic singularities, it becomes possible to offset the singularities of the suitably pulled-back metric $\pi^*\cal{H}$. Subsequently, applying the Negativity Lemma (see \cite[Lemma 2.2]{Wat25b}) on the relatively compact subset $\widehat{X}_c$, one obtains a positive line bundle $\cal{L}\longrightarrow \widehat{X}_c$ constructed from $(\pi^*L,\cal{H})$ (see \cite[Theorem 3.5]{Wat24b}).

    By Takayama's embedding theorem (see \cite[Theorem 1.2]{Tak98}), there exists a holomorphically embedding $\varPhi_c:\widehat{X}_c\longrightarrow\bb{P}^{2n+1}$. 
    We take a general hyperplane $H$ of $\bb{P}^{2n+1}$. Then $\bb{P}^{2n+1}\setminus H$ is affine thus Stein, and since $H$ is general, it intersects $\varPhi_c(\widehat{X}_c)$.
    Let $\widehat{H}:=\varPhi_c^{-1}(H)$, then the open subset $\widehat{X}_c\setminus\widehat{H}$ is Stein manifold. 
    Since $E_Z$ is the support of a simple normal crossing divisor, an open subset $(\widehat{X}_c\setminus\widehat{H})\setminus E_Z=\widehat{X}_c\setminus(\widehat{H}\cup E_Z)$ is also Stein.
    In fact, the complement of a hypersurface in a Stein manifold is also Stein (see \cite[Theorem 3.12]{Die96}). 
    By the biholomorphicity of $\mu|_{\widehat{X}_c\setminus E_Z}$, the inverse image $\widehat{E}:=\mu^{-1}(\exc)=(\pi\circ\mu)^{-1}(\xs)$ is also simple normal crossing on $\widehat{X}_c\setminus E_Z$, and hence the open subset $(\widehat{X}_c\setminus(\widehat{H}\cup E_Z))\setminus\widehat{E}=\widehat{X}_c\setminus(\widehat{H}\cup E_Z\cup\widehat{E})$ is Stein as well.

    Since $\pi$ and $\mu$ are proper holomorphic maps, the images $Z_X:=\pi(Z)=(\pi\circ\mu)(E_Z)$ and $H_X:=(\pi\circ\mu)(\widehat{H})$ are analytic subsets of $X$ by Remmert's proper mapping theorem (see \cite[Chapter II, (8.8)]{Dem-book}). 
    Finally, due to the biholomorphicity of $\pi\circ\mu|_{\widehat{X}_c\setminus(E_Z\cup\widehat{E})}:\widehat{X}_c\setminus(E_Z\cup\widehat{E})\overset{\simeq}{\longrightarrow} (X_c)_{reg}\setminus Z_X$, the open subset $(X_c)_{reg}\setminus(H_X\cup Z_X)=\pi\circ\mu|_{\widehat{X}_c\setminus(E_Z\cup\widehat{E})}(\widehat{X}_c\setminus(\widehat{H}\cup E_Z\cup\widehat{E}))$ is also Stein manifold, and by setting $A_j:=H_X\cup Z_X$, the proof is complete. 
    Note that the assumption of this theorem is satisfied if weakly pseudoconvex complex space $X$ admits a singular positive line bundle, by Corollary \ref{Corollary: iO_h>0 iff iO_pi*h>0}.
\end{proof}

\begin{theorem}[{= Theorem\,\ref{Theorem: global L2-estimates of Nakano}}]\label{Theorem: global L2-estimates of Nakano in section 5}
    Let $X$ be a weakly pseudoconvex \kah complex space of pure dimension $n$ with a \kah metric $\omega$ and $E\longrightarrow X$ be a holomorphic vector bundle with a singular Hermitian metric $h$. 
    Assume that $X$ admits a singular positive line bundle. 
    If $h$ is $\theta$-Nakano positive on $X$ in the sense of $L^2$-estimates for a continuous $(1,1)$-form $\theta$ on $X$,
    then for any $q>0$, for any smooth quasi-plurisubharmonic function $\psi$ on $X$ such that $\theta+\idd\psi>0$ in $X$ and for any $f\in L^{2,loc}_{n,q}(X,E,he^{-\psi})$ satisfying $\dbar f=0$ on $\reg$ and $\int_X\lara{B^{-1}_{\theta,\psi,\omega}f}{f}_{h,\omega}\,e^{-\psi}\dvo<+\infty$, 
    there exists $u\in L^2_{n,q-1}(X,E;\omega,he^{-\psi})$ satisfying $\dbar u=f$ on $\reg$ and the following global $L^2$-estimate
    \begin{align*}
        \int_X|u|^2_{h,\omega}e^{-\psi}\,\dvo\leq\int_X\lara{B^{-1}_{\theta,\psi,\omega}f}{f}_{h,\omega}\,e^{-\psi}\dvo,
    \end{align*}
    where $B_{\theta,\psi,\omega}=[(\theta+\idd\psi)\oid{E},\Lambda_\omega]$.
    
    Furthermore, if $X$ is not necessarily \kah and $\omega$ is taken simply as a Hermitian metric, then the same global $L^2$-estimate as above holds only in the case $q=1$.
\end{theorem}

\begin{proof}
    By Theorem \ref{Theorem: exhaustion a.e. Stein covering}, there exist an increasing sequence of real numbers $\{c_j\}_{j\in\bb{N}}$ with $c_1>\inf_X\varPsi$ and analytic subsets $A_j$ of each $X_{c_j}$ such that each open subset $(X_{c_j})_{reg}\setminus A_j$ is Stein manifold. 
    Due to Stein-ness, $E$ becomes trivial on $S_j := ((X_{c_j})_{reg} \setminus A_j) \setminus H_j$, using a certain hypersurface $H_j$ of $(X_{c_j})_{reg} \setminus A_j$ (see \cite[Lemma 3.14]{Wat25a}). 
    Since the complement of a hypersurface in a Stein manifold is also Stein, each $S_j$ is also Stein (see \cite[Theorem 3.12]{Die96}). 
    Applying the $\theta$-Nakano positivity of $h$, there exists $u_j\in L^2_{n,q-1}(S_j,E;\omega,he^{-\psi})$ such that $\dbar u_j=f$ on $S_j$ and 
    \begin{align*}
        \int_{S_j}|u_j|^2_{h,\omega}e^{-\psi}dV_\omega\leq\int_{S_j}\lara{B^{-1}_{\theta,\psi,\omega}f}{f}_{h,\omega}\,e^{-\psi}\dvo\leq\int_X\lara{B^{-1}_{\theta,\psi,\omega}f}{f}_{h,\omega}\,e^{-\psi}\dvo.
    \end{align*}
    It follows from the $\dbar$-extension Lemma \ref{Lemma: dbar-extension lemma} that $\dbar u_j=f$ on $(X_{c_j})_{reg}$ and $u_j\in L^2_{n,q-1}(X_{c_j},E;\omega,he^{-\psi})$. 
    Hence, by applying \cite[Lemma 3.18]{Wat25a} for the index $j$, we obtain the desired solution $u\in L^2_{n,q-1}(X,E;\omega,he^{-\psi})$ as the weak limite of a convergent subsequence, satisfying the $\dbar$-equation $\dbar u=f$ on $\reg$ and the global $L^2$-estimate 
    \begin{align*}
        \int_X|u|^2_{h,\omega}e^{-\psi}\,\dvo\leq\int_X\lara{B^{-1}_{\theta,\psi,\omega}f}{f}_{h,\omega}\,e^{-\psi}\dvo.
    \end{align*}

    Finally, we prove that the same statement holds for $q = 1$ when $X$ is not necessarily \kah and $\omega$ is simply a Hermitian metric. 
    By the relative compactness and Stein property of $S_j$, there exists a complete \kah metric $\widehat{\omega}_j$ on $S_j$ such that $\omega<\widehat{\omega}_j$ on $S_j$. 
    Therefore, using \cite[Chapter VIII, Lemma 6.3]{Dem-book}, we have 
    \begin{align*}
        \int_{S_j}\lara{B^{-1}_{\theta,\psi,\widehat{\omega}_j}f}{f}_{h,\widehat{\omega}_j}\,e^{-\psi}dV_{\widehat{\omega}_j}\leq\int_{S_j}\lara{B^{-1}_{\theta,\psi,\omega}f}{f}_{h,\omega}\,e^{-\psi}\dvo.
    \end{align*}
    By applying the $\theta$-Nakano positivity of $h$ and Lemma \ref{Lemma: inequality of (n,0)-forms}, we obtain a solution 
    \begin{align*}
        u_j\in L^2_{n,q-1}(S_j,E;\widehat{\omega}_j,he^{-\psi})=L^2_{n,q-1}(S_j,E;\omega,he^{-\psi})
    \end{align*}
    of the $\dbar$-equation $\dbar u_j=f$ on $S_j$, satisfying the $L^2$-estimate 
    \begin{align*}
        \int_{S_j}|u_j|^2_{h,\omega}e^{-\psi}dV_\omega=\int_{S_j}|u_j|^2_{h,\widehat{\omega}_j}e^{-\psi}dV_{\widehat{\omega}_j}&\leq\int_{S_j}\lara{B^{-1}_{\theta,\psi,\widehat{\omega}_j}f}{f}_{h,\widehat{\omega}_j}\,e^{-\psi}dV_{\widehat{\omega}_j}\\
        &\leq\int_X\lara{B^{-1}_{\theta,\psi,\omega}f}{f}_{h,\omega}\,e^{-\psi}\dvo.
    \end{align*}
    Therefore, the proof is completed by repeating \cite[Lemma 3.18]{Wat25a} in the same way as in the above argument.
\end{proof}

\begin{remark}\label{Remark: condition of L2-estimates for Grif>0 and singular positive}
    By Corollary \ref{Corollary: h>Grif and a.e Grif then iO_det pi*h>0 as currents}, the assumption in global $L^2$-existence Theorem \ref{Theorem: global L2-estimates of Nakano in section 5} that $X$ carries a singular positive line bundle can be replaced by the assumption that $h$ is Griffiths positive and a.e.\! Griffiths semi-positive on $X$.
\end{remark}

Applying Theorem \ref{Theorem: h>Grif then h det h>Nak} and Remark \ref{Remark: condition of L2-estimates for Grif>0 and singular positive} to Theorem \ref{Theorem: global L2-estimates of Nakano in section 5}, we obtain the following global $L^2$-estimate for Griffiths positivity.

\begin{theorem}\label{Theorem: global L2-estimates of Griffiths}
    Let $X$ be a weakly pseudoconvex \kah complex space of pure dimension $n$ with a \kah metric $\omega$ and $E\longrightarrow X$ be a holomorphic vector bundle with a singular Hermitian metric $h$. 
    If $h$ is Griffiths positive and a.e.\! Griffiths semi-positive on $X$, then there exists a smooth positive $(1,1)$-form $\theta$ on $X$, and the following holds:
    for any $q>0$, for any smooth plurisubharmonic function $\psi$ on $X$ and for any $f\in L^{2,loc}_{n,q}(X,E\otimes\det E,h\otimes\det he^{-\psi})$ satisfying $\dbar f=0$ on $\reg$ and $\int_X\lara{B^{-1}_{\theta,\psi,\omega}f}{f}_{h\otimes\det h,\omega}\,e^{-\psi}\dvo<+\infty$, 
    there exists $u\in L^2_{n,q-1}(X,E\otimes\det E;\omega,h\otimes\det he^{-\psi})$ satisfying $\dbar u=f$ on $\reg$ and the following global $L^2$-estimate
    \begin{align*}
        \int_X|u|^2_{h\otimes\det h,\omega}e^{-\psi}\,\dvo\leq\int_X\lara{B^{-1}_{\theta,\psi,\omega}f}{f}_{h,\omega}\,e^{-\psi}\dvo,
    \end{align*}
    where $B_{\theta,\psi,\omega}=[(\theta+\idd\psi)\oid{E\otimes\det E},\Lambda_\omega]$.
    
    Furthermore, if $X$ is not necessarily \kah and $\omega$ is taken simply as a Hermitian metric, then the same global $L^2$-estimate as above holds only in the case $q=1$.
\end{theorem}

\section{The weak $\overline{\partial}$-operator and its $L^2$-Dolbeault complex}\label{Section: dbar-operator and L2-Dolbeault complexes}

In this section, we provide proofs of Theorems \ref{Theorem: fine Dolbeault resolution} and \ref{Theorem: L2-Dolbeault isomorphism}.
Let $X$ be a complex space of pure dimension $n$ with a Hermitian metric $\omega$, $E\longrightarrow X$ be a holomorphic vector bundle with a singular Hermitian metric $h$ and $U\subset X$ be an open subset. 
Recall that 
\begin{align*}
    L^{2,loc}_{p,q}(U,E,h):=\Bigl\{u\in L^{2,loc}_{p,q}(\ureg,E,h)\,\Big|\, u|_{K_{reg}}\in L^2_{p,q}(K_{reg},E,h),\forall\,K\Subset U\Bigr\}.
\end{align*}  
It is easy to check that the presheaves given as $\scr{L}^{2,loc}_{p,q}(E,h)(U):=L^{2,loc}_{p,q}(U,E,h)$ are already sheaves $\scr{L}^{2,loc}_{p,q}(E,h)\longrightarrow X$. 
On $L^{2,loc}_{p,q}(U,E,h)$, we denote by 
\begin{align*}
    \dbar(U):L^{2,loc}_{p,q}(U,E,h)\longrightarrow L^{2,loc}_{p,q+1}(U,E,h)
\end{align*}
the $\dbar$-operator in the sense of distributions on $\ureg=U\setminus\xs$ which is closed and densely defined. 
This $\dbar$-operator is often denoted by $\dbar_w$ or $\dbar_{max}$.
Since $\dbar$ is a local operator, 
we can define the presheaves of germs of forms in the domain of $\dbar$, 
\begin{align*}
    \scr{L}^{p,q}_{E,h}:=\scr{L}^{2,loc}_{p,q}(E,h)\cap \dbar^{-1}\scr{L}^{2,loc}_{p,q+1}(E,h),
\end{align*}
given by $\scr{L}^{p,q}_{E,h}(U)=\scr{L}^{2,loc}_{p,q}(E,h)(U)\cap\rom{Dom}\,\dbar(U)$. 
This sheaf $\scr{L}^{p,q}_{E,h}$ is well known when $X$ is smooth (see \cite{Dem12,Wat25a}), and admits various types of $L^2$-Dolbeault resolutions (see \cite{Rup14,SZ23a,SZ23b,SZ24,Wat25a,Wat26a,Wat26b}). 
It is easy to see that the sheaves $\scr{L}^{p,q}_{E,h}$ admit partitions of unity, then we obtain sequences of fine sheaves 
\begin{align*}
    \scr{L}^{p,0}_{E,h}\overset{\dbar}{\longrightarrow}\scr{L}^{p,1}_{E,h}\overset{\dbar}{\longrightarrow}\scr{L}^{p,2}_{E,h}\overset{\dbar}{\longrightarrow}\cdots.
\end{align*}

Let $\pi:\tx\longrightarrow X$ be a canonical desingularization. 
Here, the smooth metric $\pi^*\omega$ is positive on $\tx\setminus \exc$ and semi-positive on $\tx$.
Let $\gamma$ be a Hermitian metric on $\tx$, then $\pi^*\omega\lesssim\gamma$ and $\pi^*\omega\sim\gamma$ on compact subsets of $\tx\setminus \exc$.
Since $\gamma$ is positive and $\pi^*\omega$ is semi-positive, there exists a smooth function $g\in \cal{C}^\infty(\tx,\bb{R}_{\geq0})$ such that $dV_{\pi^*\omega}=g\,dV_\gamma$, where $g>0$ on $\tx\setminus \exc$. 
It follows immediately that $g|u|^2_{\pi^*\omega}=|u|^2_\gamma$ for any $(n,0)$-form $u$, and that $|u|^2_\gamma\lesssim_{\tl{U}} g|u|^2_{\pi^*\omega}$ on any open subset $\tl{U}\Subset\tx$, for any $(n,q)$-form $u$, hence  
\begin{align*}
    \int_{\tl{U}}|u|^2_\gamma dV_\gamma \lesssim_{\tl{U}} \int_{\tl{U}}g|u|^2_{\pi^*\omega}g^{-1}dV_{\pi^*\omega}=\int_{\tl{U}}|u|^2_{\pi^*\omega}dV_{\pi^*\omega},
\end{align*}
here $1\leq q\leq n$.
Therefore, we obtain 
\begin{align*}
    L^2_{n,q}(\tl{U},\pi^*\omega)\subset L^2_{n,q}(\tl{U},\gamma) \quad \text{and} \quad L^2_{n,q}(\tl{U},\pi^*E;\pi^*\omega,\pi^*h)\subset L^2_{n,q}(\tl{U},\pi^*E;\gamma,\pi^*h),
\end{align*}
for any open subset $\tl{U}\Subset \tx$. In particular, these $L^2$-spaces coincide when $q=0$ by Lemma \ref{Lemma: inequality of (n,0)-forms}.
The pullback of forms under $\pi$ gives the isometry 
\begin{align*}
    \pi^*\!:\!L^2_{p,q}(\ureg,E;\omega,h)\longrightarrow L^2_{p,q}(\pi^{-1}(U)\!\setminus\!\exc,\pi^*E;\pi^*\omega,\pi^*h)\!=\!L^2_{p,q}(\pi^{-1}(U),\pi^*E;\pi^*\omega,\pi^*h).
\end{align*}
Hence, together with the above inclusions of $L^2$-spaces, if $U\Subset X$ is a relatively compact open subet, then the following map induced by $\pi^*$ is continuous: 
\begin{align*}
    \pi^*:L^2_{n,q}(\ureg,E;\omega,h)\longrightarrow L^2_{n,q}(\pi^{-1}(U),\pi^*E;\gamma,\pi^*h).  \tag{$\ast$}
\end{align*}

\begin{proposition}\label{Proposition: Ker dbar}
    If a singular Hermitian metric $h$ on $E$ is uniformly positive definite on $X$, then we have the isomorphism $\ogr(E,h)\cong\rom{Ker}\,\dbar_X\subset\scr{L}^{n,0}_{E,h}$.
\end{proposition}

\begin{proof}  
    It follows from $(\ast)$ that $\pi^*$ defines a morphism of $L^2$-Dolbeault complexes 
    \begin{align*}
        \pi^*:(\scr{L}^{n,\ast}_{E,h},\dbar)\longrightarrow (\pi_*(\scr{L}^{n,\ast}_{\pi^*E,\pi^*h}),\pi_*\dbar), 
    \end{align*} 
    as follows.
    Let $U\subset X$ be an open subset, and let $u\in\scr{L}^{n,q}_{E,h}(U)$ and $f\in\scr{L}^{n,q+1}_{E,h}(U)$ be chosen such that $\dbar u=f$.
    By the morphism $(\ast)$, we have $\pi^*u\in\scr{L}^{2,loc}_{n,q}(\pi^{-1}(U),\pi^*E,\pi^*h)$ and $\pi^*f\in\scr{L}^{2,loc}_{n,q+1}(\pi^{-1}(U),\pi^*E,\pi^*h)$ with $\dbar\pi^*u=\pi^*f$ on $\pi^{-1}(U)\setminus\exc$. 
    From the uniform positive definiteness of $h$, the $\dbar$-extension Lemma \ref{Lemma: dbar-extension lemma} can be applied, and $\dbar\pi^*u=\pi^*f$ on $\pi^{-1}(U)$ follows. 
    Hence, we obtain $\pi^*u\in\scr{L}^{n,q}_{\pi^*E,\pi^*h}(\pi^{-1}(U))$ and $\pi^*f\in$ $\scr{L}^{n,q+1}_{\pi^*E,\pi^*h}(\pi^{-1}(U))$, and the above map $\pi^*$ is in fact a morphism of $L^2$-Dolbeault complexes. 
    Including $\rom{Ker}\,\dbar_X:=\rom{Ker}\,(\dbar_X:\scr{L}^{n,0}_{E,h}\longrightarrow\scr{L}^{n,1}_{E,h})\subset\scr{L}^{n,0}_{E,h}$ and 
    $\rom{Ker}\,\dbar_{\tx}:=\rom{Ker}\,(\dbar_{\tx}:\scr{L}^{n,0}_{\pi^*E,\pi^*h}\longrightarrow\scr{L}^{n,1}_{\pi^*E,\pi^*h})\subset\scr{L}^{n,0}_{\pi^*E,\pi^*h}$, we obtain the follows commutative diagram:
\[
\small
\begin{CD}
0 @>>> \rom{Ker}\,\dbar_X
  @>>> \scr{L}^{n,0}_{E,h}
  @>{\dbar_X}>> \scr{L}^{n,1}_{E,h}
  @>{\dbar_X}>> \scr{L}^{n,2}_{E,h}
  @>{\dbar_X}>>  \\
@. @VV{\pi^*}V
   @VV{\pi^*}V
   @VV{\pi^*}V
   @VV{\pi^*}V \\
0 @>>> \pi_*(\rom{Ker}\,\dbar_{\tx})
  @>>> \pi_*(\scr{L}^{n,0}_{\pi^*E,\pi^*h})
  @>{\pi_*\dbar_{\tx}}>>
     \pi_*(\scr{L}^{n,1}_{\pi^*E,\pi^*h})
  @>{\pi_*\dbar_{\tx}}>>
     \pi_*(\scr{L}^{n,2}_{\pi^*E,\pi^*h})
  @>{\pi_*\dbar_{\tx}}>> 
\end{CD}
\]

    Here, it is already known from \cite[Theorem 5.3]{Wat25a} that $\rom{Ker}\,\dbar_{\tx}\cong K_{\tx}\otimes\scr{E}(\pi^*h)$.
    Together with the functoriality property given in Proposition \ref{Proposition: functoriality property}, this implies $\pi_*(\rom{Ker}\,\dbar_{\tx})$ $=\pi_*(K_{\tx}\otimes\scr{E}(\pi^*h))=\ogr(E,h)$; in particular, the desired isomorphism
    \begin{align*}
        \rom{Ker}\,\dbar_X\cong\pi_*(\rom{Ker}\,\dbar_{\tx})=\ogr(E,h)
    \end{align*}
    is obtained.
    In fact, since the isomorphism $\scr{L}^{2,loc}_{n,0}(E,h)\cong\pi_*(\scr{L}^{2,loc}_{n,0}(\pi^*E,\pi^*h))$ holds and the $\dbar$-equation extends across $\exc$ by the $\dbar$-extension Lemma \ref{Lemma: dbar-extension lemma}, the left vertical arrow in this commutative diagram is an isomorphism.
\end{proof}

\begin{remark-theorem}\label{Remark-Theorem: uniformly positive definite and L2-Dolbeault resolusions}
    The uniformly positive definite condition required for the application of the $\dbar$-extension Lemma \ref{Lemma: dbar-extension lemma} is essential in Proposition \ref{Proposition: Ker dbar}. 
    Indeed, if the $\dbar$-extension Lemma \ref{Lemma: dbar-extension lemma} were applicable without assuming the uniformly positive definite condition of $h$, then a contradiction would arise by using the $L^2$-Dolbeault resolution for logarithmic sheaves; see \cite[Theorem 3.2 or Proposition 3.3]{Wat26a}.
\end{remark-theorem}

\begin{proof}
    Let $X$ be a compact \kah manifold, $D=\sum^J_{j=1}D_j$ be a simple normal crossing divisor in $X$ and $\omega_P$ be a smooth \kah metric on $X\setminus D$ which is of Poincar\'{e} type along $D$. 
    Let $\sigma_j$ be the defining section of $D_j$ and $|\bullet|_{D_j}$ be a fixed smooth Hermitian metric on $\cal{O}_X(D_j)$. Let $E\longrightarrow X$ be a holomorphic vector bundle with a singular Hermitian metric $h$.
    We introduce a new singular Hermitian metric $h_{\log D}$ on $E$ by 
    \begin{align*}
        h_{\log D}:=h\prod_{j=1}^{J}|\sigma_j|^2_{D_j}(\log|\sigma_j|^2_{D_j})^2,
    \end{align*}
    which is not uniformly positive definite.
    If the $\dbar$-extension Lemma \ref{Lemma: dbar-extension lemma} is applicable even to this non-uniformly positive definite metric $h_{\log D}$, then, by the same argument as in the proof of Proposition \ref{Proposition: Ker dbar}, 
    we obtain 
    \begin{align*}
        \rom{Ker}\,(\dbar:\scr{L}^{n,0}_{E,h_{\log D}}\longrightarrow \scr{L}^{n,1}_{E,h_{\log D}})\cong K_X\otimes\scr{E}(h_{\log D})\subseteq K_X\otimes E.
    \end{align*}

    Define the sheaf $\scr{L}^{p,q}_{E,h_{\log D},\omega_P}$ over $X$ as follows. For any open subset $U\subset X$, the space of sections $\scr{L}^{p,q}_{E,h_{\log D},\omega_P}(U)$ consists of $E$-valued $(p,q)$-forms $u$ with measurable coefficients such that $|u|^2_{h_{\log D},\omega_P}dV_{\omega_P}$ and $|\dbar u|^2_{h_{\log D},\omega_P}dV_{\omega_P}$ are integrable on $U\setminus D$ (see \cite[Subsection 2.3]{Wat26a}).
    Here, it follows from \cite[Proposition 3.3]{Wat26a} that 
    \begin{align*}
        \rom{Ker}\,(\dbar:\scr{L}^{p,0}_{E,h_{\log D},\omega_P}\longrightarrow \scr{L}^{p,1}_{E,h_{\log D},\omega_P})\cong \Omega^p_X(\log D)\otimes\scr{E}(h). 
    \end{align*}
    In particular, when $p=n$, we have $\rom{Ker}\,\dbar\cong \Omega^n_X(\log D)\otimes\scr{E}(h)=K_X\otimes\cal{O}_X(D)\otimes\scr{E}(h)$.

    For any Hermitian metric $\gamma$ on $X$, it follows from Lemma \ref{Lemma: inequality of (n,0)-forms} that $|\bullet|^2_{h_{\log D},\omega_P}dV_{\omega_P}=|\bullet|^2_{h_{\log D},\gamma}dV_{\gamma}$ holds for $(n,0)$-forms. 
    Consequently, $\scr{L}^{n,0}_{E,h_{\log D},\omega_P}=\scr{L}^{n,0}_{E,h_{\log D}}$, and we have 
    \begin{align*}
        K_X\otimes\cal{O}_X(D)\otimes\scr{E}(h)\cong K_X\otimes\scr{E}(h_{\log D})\subseteq K_X\otimes E.
    \end{align*}
    Therefore, if $\nu(-\log\det h,x)<2$ for all $x\in X$, then $\scr{E}(h)=\cal{O}_X(E)$ (see \cite[Lemma 6.4]{Wat25a}), and hence $\cal{O}_X(D)\subseteq\cal{O}_X$, which is a contradiction. 
    This phenomenon arises from the fact that, if the $\dbar$-extension Lemma \ref{Lemma: dbar-extension lemma} were applicable to $h_{\log D}$, then sections that should have logarithmic poles would become holomorphic.
\end{proof}

By Proposition \ref{Proposition: Ker dbar}, if $h$ is uniformly positive definite, then $\ogr(E,h)$ coincides with the well-known multiplier $S$-sheaf $S_X(E,h)$ (see \cite[Definition 4.1]{SZ24}). 
The following results are obtained from Theorems \ref{Theorem: h Nak iff pi*h Nak} and \ref{Theorem: h>Grif then h det h>Nak} together with the coherence of $\scr{E}(\pi^*h)$ on $\tx$ (see \cite[Proposition 4.4]{Ina22}) and Grauert's direct image theorem.

\begin{theorem}\label{Theorem: coherence of ogr(E,h) if Nakano}
    If a singular Hermitian metric $h$ on $E$ is locally Nakano bounded from below on $X$ in the sense of $L^2$-estimates, then $\ogr(E,h)$ is coherent on $X$. 
\end{theorem}

\begin{corollary}\label{Corollary: coherence of ogr(E,h) if a.e. Griffiths}
    If a singular Hermitian metric $h$ on $E$ is locally Griffiths bounded from below on $X$, then $\ogr(E\otimes\det E,h\otimes\det h)$ is coherent on $X$. 
\end{corollary}

\begin{remark}
    In \cite[Proposition 2.9 and Remark 2.12]{SZ23b}, it is shown that $S_X(E,h)$ is coherent when the smooth Hermitian metric $h$ on $E|_{\reg}$ is Nakano semi-positive on $\reg$ and $(E,h)$ on $\reg$ is tame $($see \cite[Definition 2.8]{SZ23b}$)$; 
    furthermore, it is stated that the coherence remains unclear if the tameness assumption is removed. 
    More precisely, however, coherence and the isomorphism $S_X(E,h)\cong\ogr(E,h)$ follows from the weaker condition that the metric is uniformly positive definite on $X$.
\end{remark}

\begin{proof}
    Assume that the singular Hermitian metric $h$ on $E$ is smooth and Nakano semi-positive on $\reg$, and uniformly positive definite on $X$. Then $\pi^*h$ is Nakano semi-positive on $\tx\setminus\exc$. 
    Arguing as in the proof of Theorem \ref{Theorem: h Nak iff pi*h Nak} and applying the $\dbar$-extension Lemma \ref{Lemma: dbar-extension lemma} using the uniform positive definiteness of $\pi^*h$, Nakano semi-positivity of $\pi^*h$ extends from $\tx\setminus\exc$ to $\tx$. 
    It then follows from \cite[Proposition 4.4]{Ina22} that $\scr{E}(\pi^*h)$ is coherent on $\tx$. Note that \cite{Ina22} uses only optimal $L^2$-estimates and does not rely on Griffiths semi-positivity. 
    Therefore, the assertion follows from Proposition \ref{Proposition: functoriality property} and Grauert's direct image theorem.
\end{proof}

In this sense, Theorem \ref{Theorem: coherence of ogr(E,h) if Nakano} and Corollary \ref{Corollary: coherence of ogr(E,h) if a.e. Griffiths} generalize the considerations in \cite{SZ23b}.

\begin{proof}[Proof of Theorem \ref{Theorem: L2-Dolbeault isomorphism}]
    By Theorem \ref{Theorem: h Nak iff pi*h Nak}, $\pi^*h$ is also locally Nakano bounded from below on $\tx$. Therefore, for every point $x$ of $\tx$, we can choose an open neighborhood $\tu\subset\tx$ of $x$ and a smooth weight $\varphi$ on $\tu$ such that $\pi^*he^{-\varphi}$ is Nakano semi-positive on $\tu$. 
    In particular, $\scr{L}^{p,q}_{\pi^*E,\pi^*h}(\tu)=\scr{L}^{p,q}_{\pi^*E,\pi^*he^{-\varphi}}(\tu)$. It follows from \cite[Theorem 5.3 (b)]{Wat25a} that the $L^2$-Dolbeault complex $(\scr{L}^{n,\ast}_{\pi^*E,\pi^*h},\dbar_{\tx})$ on $\tx$ is a fine resolusion of $K_{\tx}\otimes\scr{E}(\pi^*h)$; that is, the sequence of sheaves
    \begin{align*}
        0\longrightarrow K_{\tx}\otimes\scr{E}(\pi^*h)\longrightarrow\scr{L}^{n,\ast}_{\pi^*E,\pi^*h}
    \end{align*}
    is exact. Hence, by using the Leray spectral sequence, the following higher direct image vanishing Theorem \ref{Theorem: higher direct image vanishing in section 6} implies that the lower line of the commutative diagram is a fine resolution of $\pi_*(K_{\tx}\otimes\scr{E}(\pi^*h))=\ogr(E,h)$. 
\end{proof}

We prove the following Takegoshi-type higher direct image vanishing theorem for singular Hermitian metrics, corresponding to \cite[Theorem I and Remark 2]{Tak85} in the case of a smooth Hermitian metric.
Here, following Takegoshi, the proof is carried out by applying the Nakano-Nadel vanishing theorem locally on weakly pseudoconvex manifolds. 
To preserve positivity while keeping the $L^2$-subsheaf unchanged, Theorem \ref{Theorem: h>Nak then pi*he-psi>Nak & E(pi*h)=E(pi*he-psi)} plays a key role, which is obtained from the Negativity Lemma (see \cite[Lemma 2.2]{Wat25b}) and the strong openness property (see \cite{GZ15,LXYZ24}).

\begin{theorem}[{= Theorem \ref{Theorem: higher direct image vanishing}}]\label{Theorem: higher direct image vanishing in section 6}
    Let $X$ be a complex space and $\pi:\tx\longrightarrow X$ be a canonical desingularization. 
    If $h$ is locally Griffiths-Nakano bounded from below on $X$, then we have the higher direct image vanishing 
    \begin{align*}
        R^q\pi_*(K_{\tx}\otimes\scr{E}(\pi^*h))=0 
    \end{align*}
    for any $q>0$.
\end{theorem}

\begin{proof}
    It suffices to show that $H^q(\pi^{-1}(U),K_{\tx}\otimes\scr{E}(\pi^*h))=0$ for any $q>0$ and any sufficiently small Stein open neighborhood $U\subset X$. 
    We may take the Stein open neighborhood $U$ to admit a trivialization $\tau:E|_U\overset{\simeq}{\longrightarrow}U\times\bb{C}^r$ and a holomorphic embedding $\iota:U\hookrightarrow\bb{C}^N$. 
    Let $(z_1,\ldots,z_N)$ be the standard coordinates on $\bb{C}^N$. 
    By the assumption of $h$ and Corollary \ref{Corollary: h>Gri0 & h_L>0 then h*h_L>Gri0}, the metric $he^{-C\iota^*|z|^2}$ is Nakano and Griffiths positive and a.e.\! Griffiths semi-positive on a neighborhood of $\overline{U}$ for some large constant $C>0$. 
    Therefore, by applying Theorem \ref{Theorem: h>Nak then pi*he-psi>Nak & E(pi*h)=E(pi*he-psi)}, there exist a quasi-plurisubharmonic function $\psi:\pi^{-1}(U)\longrightarrow[-\infty,+\infty)$ which is smooth on $\pi^{-1}(U) \setminus \exc$,
    and a small number $\varepsilon_0 > 0$ such that $\pi^*he^{-C\pi^*\iota^*|z|^2-\varepsilon_0\psi}$ is Nakano positive on $\pi^{-1}(U)$ and satisfies 
    \begin{align*}
        \scr{E}(\pi^*h)=\scr{E}(\pi^*he^{-C\pi^*\iota^*|z|^2})=\scr{E}(\pi^*he^{-C\pi^*\iota^*|z|^2-\varepsilon_0\psi})
    \end{align*}
    on $\pi^{-1}(U)$, where $\pi^*\iota^*|z|^2$ is smooth on $\pi^{-1}(U)$.
    
    By the Negativity lemma (see \cite[Lemma 2.2]{Wat25b}), there exists a positive line bundle on $\pi^{-1}(U)$, induced by the exceptional divisor together with the smooth metric $e^{-\pi^*\iota^*|z|^2}$ on a trivial line bundle $\pi^{-1}(U)\times\bb{C}$, whose curvature provides a \kah metric on $\pi^{-1}(U)$. 
    Here, $\pi^{-1}(U)$ is weakly pseudoconvex. Hence, by the Nakano-Nadel vanishing theorem (see \cite[Theorem 1.1]{Wat25c}, \cite[Corollary 6.2]{Wat25a}) on weakly pseudoconvex manifolds, the proof is complete. 
    
    Indeed, the trivial line bundle $\pi^{-1}(U)\times\bb{C}$ admits the singular positive Hermitian metric $e^{-C\pi^*\iota^*|z|^2-\varepsilon_0\psi}$, which is smooth and simply positive on the Stein subset $\pi^{-1}(U)\setminus\exc$. 
    Arguing as in the proof of \cite[Corollary 6.2 and Theorem 1.2]{Wat25a}, by extending the $\dbar$-equation obtained from the $L^2$-estimate to $\pi^{-1}(U)\setminus\exc$, we obtain the cohomology vanishing $H^q(\pi^{-1}(U),K_{\tx}\otimes\scr{E}(\pi^*he^{-C\pi^*\iota^*|z|^2-\varepsilon_0\psi}))=0$ for any $q>0$.
\end{proof}

\begin{remark}\label{Remark: Thms condition of pi add blow ups}
    In Theorems \ref{Theorem: L2-Dolbeault isomorphism} and \ref{Theorem: higher direct image vanishing}, it suffices for $\pi:\tx\longrightarrow X$ to be a composition of a resolution of singularities and finitely many local blow-ups, and the statements remain valid even if additional blow-ups of certain analytic subsets are composed.
    In particular, in the case where $X$ is smooth, Theorems \ref{Theorem: L2-Dolbeault isomorphism} and \ref{Theorem: higher direct image vanishing} also hold when $\pi:\tx\longrightarrow X$ is taken as the blow-up of some analytic subset.
\end{remark}

Finally, we show that the upper line of the commutative diagram above is also exact; this is precisely the proof of Theorem \ref{Theorem: fine Dolbeault resolution}.
Let $(M,\omega)$ be a Hermitian complex manifold, $F\longrightarrow M$ be a holomorphic vector bundle equipped with a singular Hermitian metric $h_F$. 
Let $H^{p,q}_{L^2}(M,F;\omega,h_F)_{max}$ (resp. $H^{p,q}_{L^2}(M,F;\omega,h_F)_{min}$) be the $L^2$-Dolbeault cohomology on $M$ with respect to the maximal closed extension $\dbar_{max}$ (resp. the minimal closed extension $\dbar_{min}$).  
Here, $\dbar_{max}$ is the $\dbar$-operator in the sense of distributions on $M$, i.e., $\dbar$ or $\dbar_w$, and $\dbar_{min}$ is the the minimal closed Hilbert space extension of the $\dbar$-operator on smooth forms with compact support, i.e., $\dbar_s$.

\begin{proof}[Proof of Theorem \ref{Theorem: fine Dolbeault resolution}]
    By Proposition \ref{Proposition: Ker dbar}, the $L^2$-Dolbeault complex $(\scr{L}^{n,\ast}_{E,h},\dbar)$ is exact in degree $q=0$, and it is also known to be exact in degrees $q\geq1$ when restricted to the regular locus $\reg$ (see \cite[Theorem\,5.3]{Wat25a}). 
    Hence, it suffices to show that it is exact in degrees $q\geq1$ along the singular locus $\xs$. 
    
    For any fixed singular point $x_0\in\xs$, take a relatively compact Stein open neighborhood $U$ of $x_0$ which admits a holomorphic closed embedding $\iota_U:U\hookrightarrow\bb{C}^N$. 
    Here $(z_1,\ldots,z_N)$ denote the coordinate on $\bb{C}^N$, and taking $\psi:=\iota^*_U|z|^2=\iota^*_U\sum^N_{j=1}|z_j|^2$ as a smooth strictly plurisubharmonic function on $U$, the space $U$ carries a \kah metric $\omega:=\idd\psi=\iota_U^*\idd|z|^2$ and a smooth positive metric $e^{-\psi}$ on $\cal{O}_X|_U$. 
    By the assumption of $h$, there exists a constant $C>0$ such that $he^{-C\psi}$ is Nakano semi-positive on $U$. 

    Since $\psi$ is smooth on $U$, it follows that $L^{2,loc}_{p,q}(U,E,h)=L^{2,loc}_{p,q}(U,E,he^{-(C+1)\psi})$; in particular, $\scr{L}^{p,q}_{E,h}(U)=\scr{L}^{p,q}_{E,he^{-(C+1)\psi}}(U)$. 
    By $L^2$-existence Theorem \ref{Theorem: global L2-estimates of Nakano in section 5}, for any $\dbar$-closed $f\in L^2_{n,q}(U,E;\omega,he^{-(C+1)\psi})=L^2_{n,q}(U,E;\omega,h)$, there exists a solution $u\in L^2_{n,q-1}(U,E;\omega,$ $he^{-(C+1)\psi})=L^2_{n,q-1}(U,E;\omega,h)$ satisfying $\dbar u=f$ on $U_{reg}$ and the $L^2$-estimate 
    \begin{align*}
        ||u||^2_{he^{-(C+1)\psi}\!,\,\omega}\leq\int_U\lara{B^{-1}_{0,\psi,\omega}f}{f}_{he^{-C\psi}\!,\,\omega}\,e^{-\psi}dV_\omega=\frac{1}{q}||f||^2_{he^{-(C+1)\psi}\!,\,\omega}<+\infty,
    \end{align*}
    where $B_{0,\psi,\omega}=[\idd\psi\oid{E},\Lambda_\omega]=[\omega\oid{E},\Lambda_\omega]=q$ on $\bigwedge^{n,q}T^*_{\ureg}\otimes E$.
    Therefore, we obtain the $L^2$-cohomology vanishing 
    \begin{align*}
        H^{n,q}_{L^2}(\ureg,E;\omega,h)_{max}=0
    \end{align*}
    for any $q\geq1$. 
    This implies exactness in degrees $q\geq1$, and the proof is complete.
\end{proof}

\begin{remark}\label{Remark: Generalization of L2-Dol resolusion}
    On a complex analytic space, the $L^2$-Dolbeault resolution for holomorphic vector bundles is given in \cite[Theorem 2.11]{SZ23b} under the assumption that the vector bundle has a smooth Nakano semi-positive metric on $\reg$. 
    In \cite[Theorem 5.1]{SZ24}, an $L^2$-Dolbeault resolution is obtained for the multiplier $S$-sheaf $S(IC_{X}(\bb{V}),\varphi)\otimes E$ in the setting where an $\bb{R}$-polarized variation of Hodge structure $\mathbb{V}=(\mathcal{V},\nabla,\mathcal{F}^{\bullet},h_{\mathbb{V}})$ exists on a Zariski open subset $X^{o}\subset\reg$, 
    provided that the singular Hermitian metric $h$ on $E$ can only be decomposed as $h=h_{0}e^{-\varphi}$, where $h_{0}$ is a smooth metric and $\varphi$ is a quasi-plurisubharmonic function, i.e., its singularity is of line-bundle type. 
    More recently, an analogous result of line-bundle type has been obtained without assuming the existence of an $\bb{R}$-polarized variation of Hodge structure $($see \cite[Theorem 1.1]{Wat26b}$)$.
\end{remark}

Hence, Theorem \ref{Theorem: fine Dolbeault resolution} provides $L^2$-Dolbeault resolutions for more general singular Hermitian metrics without assuming the existence of an $\bb{R}$-polarized variation of Hodge structure.
Here, if $X$ is smooth, $X\setminus X^o$ is a (possibily empty) normal crossing, and $\bb{V}$ has unipotent local monodromies, then an isomorphism $S(IC_{X}(\bb{V}),\varphi)\cong S(IC_{X}(\bb{V}))\otimes\scr{I}(\varphi)$ is obtained (see \cite[Corollary 4.8]{SZ24}), where $S(IC_{X}(\bb{V}))$ denotes Saito's $S$-sheaf.

Furthermore, by applying Theorems \ref{Theorem: h>Grif then h det h>Nak} and \ref{Theorem: h>Grif then iO_det h>0 as currents}, we obtain the following from Theorems \ref{Theorem: fine Dolbeault resolution}, \ref{Theorem: L2-Dolbeault isomorphism} and \ref{Theorem: higher direct image vanishing}.

\begin{theorem}\label{Theorem: fine Dolbeault resolution, isomorphism and higher direct image vanishing of Griffiths}
    Let $X$ be a complex space of pure dimension $n$. 
    If $h$ is locally Griffiths bounded from below on $X$, then the following holds: 
    \begin{itemize}
        \item 
        the $L^2$-Dolbeault complex 
        \begin{align*}
            0\longrightarrow \ogr(E\otimes\det E,h\otimes\det h)\hookrightarrow\scr{L}^{n,0}_{E\otimes\det E,h\otimes\det h}\overset{\dbar}{\longrightarrow}\scr{L}^{n,1}_{E\otimes\det E,h\otimes\det h}\overset{\dbar}{\longrightarrow}\cdots 
        \end{align*}
        is exact; that is, the complex $(\scr{L}^{n,\ast}_{E\otimes\det E,h\otimes\det h},\dbar)$ is $L^2$-Dolbeault fine resolution of $\ogr(E\otimes\det E,h\otimes\det h)$.
        Thus, for any $q>0$, we have the following 
        \begin{align*}
            H^q(X,\ogr(E\otimes\det E,h\otimes\det h))\cong H^q(\Gamma(X,\scr{L}^{n,\ast}_{E\otimes\det E,h\otimes\det h})).
        \end{align*}
        \item 
        the $L^2$-Dolbeault complex $(\pi_*\scr{L}^{n,*}_{\pi^*E\otimes\det\pi^*E,\pi^*h\otimes\det\pi^*h},\pi_*\dbar)$ 
        is a fine resolution of 
        \begin{align*}
            \ogr(E\otimes\det E,h\otimes\det h)\cong\pi_*(\omega_{\tx}\otimes\scr{E}(\pi^*h\otimes\det\pi^*h)).
        \end{align*}
        Thus, for any $q\geq0$, we have the following $L^2$-Dolbeault isomorphism 
        \begin{align*}
            H^q(X,\ogr(E\otimes\det E,h\otimes\det h))\cong H^q(\tx,K_{\tx}\otimes\scr{E}(\pi^*h\otimes\det\pi^*h)). 
        \end{align*}
        \item 
        for any $q>0$, we have the higher direct image vanishing 
        \begin{align*}
            R^q\pi_*(K_{\tx}\otimes\scr{E}(\pi^*h\otimes\det\pi^*h))=0. 
        \end{align*}
    \end{itemize}
\end{theorem}

\section{Nakano-Nadel vanishing on weakly pseudoconvex complex spaces}

In this section, we provide Nakano-Nadel type vanishing theorems on weakly pseudoconvex complex spaces, taking care to account for the existence of a \kah metric.

\subsection{Compact case}

Note that in general $L^2_{p,q}(X,E;\omega,h)\subsetneq L^{2,loc}_{p,q}(X,E,h)$, and these spaces coincide when $X$ is compact.
Hence, the following immediately follows from Theorem \ref{Theorem: L2-Dolbeault isomorphism}.

\begin{theorem}\label{Theorem: isom of cohomology in compact spaces}
    Let $X$ be a compact complex space of pure dimension $n$, $\pi:\tx\longrightarrow X$ be any resolusion of singularities. 
    If $h$ is locally Griffiths-Nakano bounded from below on $X$, then the pullback of forms under $\pi$ induces a natural isomorphism 
    \begin{align*}
        \pi^*:H^{n,q}_{L^2}(\reg,E;\omega,h)_{max}\overset{\cong}{\longrightarrow} H^q(\tx,K_{\tx}\otimes\scr{E}(\pi^*h)).
    \end{align*}
    for any Hermitian metric $\omega$ on $X$ and any $q\geq0$, where we have 
    \begin{align*}
        H^{n,q}_{L^2}(\reg,E;\omega,h)_{max}=H^q(X,\ogr(E,h)). 
    \end{align*}
\end{theorem}

Here, for a Hermitian manifold $(M,\omega)$, it is known that the equality $\dbar_{max}=\dbar_{min}$ holds, i.e., the density of $\scr{D}^{\ast,\ast}_M$ in $\rom{Dom}\,\dbar_{max}$ with respect to the graph norm, if $\omega$ is complete (see \cite{AV65}, \cite[Chapter VIII]{Dem-book}), 
and that, if the complex space $X$ is compact, then its regular locus $\reg$ admits a complete Hermitian metric (see \cite{Dem82}). 

\begin{remark}
    The Hermitian metric $\omega$ in Theorem \ref{Theorem: isom of cohomology in compact spaces} is required to be defined on $X$. 
    Even though $\reg$ admits a complete Hermitian metric, it is not clear whether we can choose an appropriate $\omega$ on $X$ such that the left-hand side of the isomorphism induced by $\pi^*$ coincides with $H^{n,q}_{L^2}(\reg,E;\omega,h)_{min}$.
\end{remark}

Under the assumption of (relative) compactness, we prove that the vanishing of higher cohomology can be obtained without the existence of a \kah metric.

\begin{theorem}\label{Theorem: Nakano-Nadel vanishing on compact cpx sp without Kahler}
    Let $X$ be a compact complex space of pure dimension. 
    If $h$ is Nakano positive on $X$ in the sense of $L^2$-estimates, and is Griffiths positive and a.e.\! Griffiths semi-positive on $X$, then $X$ is Moishezon and we have the following vanishing 
    \begin{align*}
        H^q(X,\ogr(E,h))=0 
    \end{align*}
    for any $q>0$, without assuming the existence of a \kah metric.  
\end{theorem}

\begin{proof}
    Let $\pi:\tx\longrightarrow X$ be a resolusion of singularities. 
    By Theorem \ref{Theorem: h>Nak then pi*he-psi>Nak & E(pi*h)=E(pi*he-psi)}, there exist a quasi-plurisubharmonic function $\psi:\tx\longrightarrow[-\infty,+\infty)$, which is smooth on $\tx\setminus\exc$, and a small number $\varepsilon_X>0$ such that the singular Hermitian metric $\pi^*he^{-\varepsilon\psi}$ on $\pi^*E$ is Nakano and Griffiths positive on $\tx$, and satisfies $\scr{E}(\pi^*h)=\scr{E}(\pi^*he^{-\varepsilon\psi})$ on $\tx$, for any $0<\varepsilon<\varepsilon_X$. 
    Theorem \ref{Theorem: h>Grif then iO_det h>0 as currents} and Proposition \ref{Proposition: idd>0 iff a.e. psh} imply that the determinant singular metric $\det\pi^*h\cdot e^{-r\varepsilon\psi}$ on $\det\pi^*E$ is singular positive on $\tx$. Hence, $\tx$ is Moishezon, and consequently $X$ is also Moishezon (see \cite{Moi66}).

    Fix $\varepsilon\in(0,\varepsilon_X)$, and set $\cal{H}:=\det\pi^*h\cdot e^{-r\varepsilon\psi}$ for simplicity. 
    Consider the approximation $\{\cal{H}_\nu\}_{\nu\in\bb{N}}$ of this singular positive metric $\cal{H}$ obtained by the refined Demailly approximation (see \cite[Theorem 3.2]{Wat24b}). 
    Let $\mu:\widehat{X}\longrightarrow\tx$ be the blow-ups along the analytic subset given by the singular locus of one such metric $\cal{H}_\nu$. 
    Then $\widehat{X}$ admits a positive line bundle and is therefore projective; in particular, $\widehat{X}$ carries a \kah metric (see \cite[Theorem 3.5]{Wat24b}). 
    Proceeding as in the proof of Theorem \ref{Theorem: h>Nak then pi*he-psi>Nak & E(pi*h)=E(pi*he-psi)}, applying Lemma \ref{Lemma: key lemma} to the blow-ups $\mu:\widehat{X}\longrightarrow\tx$ yields a quasi-plurisubharmonic function $\varPsi:\widehat{X}\longrightarrow[-\infty,+\infty)$, which is smooth on $\widehat{X}\setminus\mu$-$\exc$, and a small number $\delta_{\tx}$ such that 
    the singular Hermitian metric $\mu^*\pi^*he^{-\varepsilon\mu^*\psi-\delta\varPsi}$ on $\mu^*\pi^*E$ is Nakano and Griffiths positive on $\widehat{X}$ and satisfies $\scr{E}(\mu^*\pi^*he^{-\varepsilon\mu^*\psi})=\scr{E}(\mu^*\pi^*he^{-\varepsilon\mu^*\psi-\delta\varPsi})$ on $\widehat{X}$, for any $0<\delta<\delta_{\tx}$.

    Hence, the Nakano-Nadel vanishing theorem on projective manifolds (see \cite[Theorem 1.5]{Ina22}, \cite[Corollary 6.2]{Wat25a}) yields $H^q(\widehat{X},K_{\widehat{X}}\otimes\scr{E}(\mu^*\pi^*he^{-\varepsilon\mu^*\psi-\delta\varPsi}))=0$ for any $q>0$.
    Combining this with Theorem \ref{Theorem: L2-Dolbeault isomorphism} and Remark \ref{Remark: Thms condition of pi add blow ups}, the desired vanishing 
    \begin{align*}
        0\!=\!H^q(\widehat{X},K_{\widehat{X}}\otimes\scr{E}(\mu^*\pi^*he^{-\varepsilon\mu^*\psi}))\cong &\, H^q(\tx,\mu_*(K_{\widehat{X}}\otimes\scr{E}(\mu^*\pi^*he^{-\varepsilon\mu^*\psi})))\\ 
        =&\, H^q(\tx,K_{\tx}\otimes\scr{E}(\pi^*he^{-\varepsilon\psi}))\!=\!H^q(\tx,K_{\tx}\otimes\scr{E}(\pi^*h))\\
        \cong&\, H^q(X,\ogr(E,h))
    \end{align*}
    follows from the $L^2$-subsheaf conditions established above.
\end{proof}

Similarly to this proof, the following is obtained by using \cite[Corollary 6.2]{Wat25a}.

\begin{theorem}\label{Theorem: Nakano-Nadel vanishing on relatively cpt w.p.c. cpx sp}
    Let $(X,\varPsi)$ be a weakly pseudoconvex complex space of pure dimension and $c>\inf_X\varPsi$ be arbitrary. 
    If $h$ is Nakano positive on $X_{c+\tau}$ in the sense of $L^2$-estimates, and is Griffiths positive and a.e.\! Griffiths semi-positive on $X_{c+\tau}$ for some $\tau>0$, then we have the following vanishing 
    \begin{align*}
        H^q(X_c,\ogr(E,h))=0 
    \end{align*}
    for any $q>0$, without assuming the existence of a \kah metric.  
\end{theorem}

Applying Theorem \ref{Theorem: h>Grif then h det h>Nak}, we obtain an analogous result for Griffiths positivity.

\begin{theorem}
    Let $X$ be a compact complex space of pure dimension. 
    If $h$ is Griffiths positive and a.e.\! Griffiths semi-positive on $X$, then $X$ is Moishezon and for any $q>0$, we have the following vanishing 
    \begin{align*}
        H^q(X,\ogr(E\otimes\det E,h\otimes\det h))=0. 
    \end{align*}
\end{theorem}

Of course, an analogue of Theorem \ref{Theorem: Nakano-Nadel vanishing on relatively cpt w.p.c. cpx sp} for Griffiths positivity can also be obtained for relatively compact sublevel sets $X_c$ of weakly pseudoconvex complex spaces.

\subsection{Non-compact case}

In this subsection, a global Nakano-Nadel type vanishing theorem is established on non-compact weakly pseudoconvex complex spaces.

\begin{theorem}[{= Theorem \ref{Theorem: Nakano-Nadel vanishing on w.p.c. cpx sp}}]\label{Theorem: Nakano-Nadel vanishing on w.p.c. cpx sp in subsection}
    Let $X$ be a weakly pseudoconvex \kah complex space of pure dimension. 
    Assume that $X$ admits a singular positive line bundle. 
    If $h$ is Nakano positive on $X$ in the sense of $L^2$-estimates, 
    then we have the following vanishing 
    \begin{align*}
        H^q(X,\ogr(E,h))=0 
    \end{align*}
    for any $q>0$. 
    Furthermore, if $X$ is not necessarily K\"{a}hler, then we have only the first cohomology vanishing $H^1(X,\ogr(E,h))=0$.
\end{theorem}

\begin{proof}
    From Theorem \ref{Theorem: fine Dolbeault resolution}, we obtain the $L^2$-Dolbeault isomorphism 
    \begin{align*}
        H^q(X,\ogr(E,h))\cong H^q(\Gamma(X,\scr{L}^{n,\ast}_{E,h}))
    \end{align*}
    and we prove that the right-hand side vanishes. Let $\omega$ be a \kah metric on $X$. 
    By the assumption, $h$ is $\theta$-Nakano positive for some positive continuous $(1,1)$-form $\theta$ on $X$.
    We arbitrarily take a convex increasing function $\chi\in\cal{C}^{\infty}(\bb{R},\bb{R})$. 
    For any $\dbar$-closed $(n,q)$-form $f\in \Gamma(X,\scr{L}^{n,q}_{L,h})$, the integral
    \begin{align*}
        \int_X\lara{B^{-1}_{\theta,\chi\circ\varPsi,\omega}f}{f}_{h,\omega}\,e^{-\chi\circ\varPsi}\dvo 
    \end{align*}
    become convergent if $\chi$ grows fast enough. By the global $L^2$-existence Theorem \ref{Theorem: global L2-estimates of Nakano}, 
    there exists $u\in L^2_{n,q-1}(X,E;\omega,he^{-\chi\circ\varPsi})$ satisfying $\dbar u=f$ on $\reg$ and 
    \begin{align*}
        \int_X|u|^2_{h,\omega}e^{-\chi\circ\varPsi}\dvo\leq\int_X\lara{B^{-1}_{\theta,\chi\circ\varPsi,\omega}f}{f}_{h,\omega}\,e^{-\chi\circ\varPsi}\dvo<+\infty. 
    \end{align*}
    By the smoothness of $\chi\circ\varPsi$, it follows that $u\in\Gamma(X,\scr{L}^{n,q-1}_{E,h})$, and hence the desired cohomology vanishing $H^q(X,\ogr(E,h))=0$ is obtained.
    
    The first cohomology vanishing can similarly be obtained using Theorem \ref{Theorem: global L2-estimates of Nakano}.
\end{proof}

By Remark \ref{Remark: condition of L2-estimates for Grif>0 and singular positive}, the following two conditions can be replaced by each other in Theorems \ref{Theorem: Nakano-Nadel vanishing on compact cpx sp without Kahler}, \ref{Theorem: Nakano-Nadel vanishing on relatively cpt w.p.c. cpx sp}, and \ref{Theorem: Nakano-Nadel vanishing on w.p.c. cpx sp in subsection}:
\begin{itemize}
\item $X$ admits a singular positive line bundle;
\item $h$ is Griffiths positive and a.e.\! Griffiths semi-positive on $X$.
\end{itemize}
Furthermore, Proposition \ref{Proposition: normal and Gri>0 then a.e. Grif>0} shows that Theorems \ref{Theorem: Nakano-Nadel vanishing on w.p.c. cpx sp in subsection} and \ref{Theorem: Griffiths vanishing on on w.p.c. cpx sp} remain valid 
if the assumption of a.e.\! Griffiths semi-positivity is replaced by the condition that either $X$ is normal, or $X$ is locally irreducible and $h^*$ is locally bounded above near $\xs$.

\begin{remark}
    Even if we attempt to reduce the problem on $X$ to the known Nakano-Nadel vanishing theorem on a weakly pseudoconvex manifold $\tx$ via Theorem \ref{Theorem: L2-Dolbeault isomorphism}, several difficulties arise. 
    More precisely, it is unclear whether $\pi^*E$ admits a singular Hermitian metric on $\tx$ that preserves positivity and the $L^2$-subsheaf on the whole of $\tx$, and it is also unclear whether the required global $L^2$-estimates can be obtained on $\tx$ with respect to a single fixed singular Hermitian metric on $\pi^*E$.
\end{remark}

Indeed, when $h$ is $\theta$-Nakano positive for some positive continuous $(1,1)$-form $\theta$ on $X$, $\pi^*\theta$ degenerates along $\exc$. 
Thus, for a given $f\in L^{2,loc}_{n,q}(\tx,\pi^*E,\pi^*h)$, even after choosing an appropriate convex increasing function $\chi\in\cal{C}^{\infty}(\bb{R},\bb{R})$, 
it is not clear whether the integral $\displaystyle\int_{\tx}\lara{B^{-1}_{\pi^*\theta,\chi\circ\pi^*\varPsi,\omega_{\tx}}f}{f}_{h,\omega_{\tx}}\,e^{-\chi\circ\pi^*\varPsi}dV_{\omega_{\tx}} $ is finite, where $\omega_{\tx}$ is a \kah metric on $\tx$. 
Consequently, the existence of global solutions to the $\dbar$-equation is unclear. Of course, if $X$ is (relatively) compact, the positivity can be recovered by Theorem \ref{Theorem: h>Nak then pi*he-psi>Nak & E(pi*h)=E(pi*he-psi)}.

Applying Proposition \ref{Proposition: smooth Grif/Nak>0 on X_reg then also >0 on X} to Theorem \ref{Theorem: Nakano-Nadel vanishing on w.p.c. cpx sp in subsection}, we obtain the following.

\begin{corollary}
    Let $X$ be a weakly pseudoconvex \kah complex space of pure dimension. 
    Assume that $h$ is smooth Hermitian metric on $E$, then the following holds:
    \begin{itemize}
        \item if $h$ is Nakano positive on $X$, then for any $q>0$, we have the vanishing 
        \begin{align*}
            H^q(X,\ogr\otimes\cal{O}_X(E))=0.
        \end{align*}
        \item if $h$ is Griffiths positive on $X$, then for any $q>0$, we have the vanishing 
        \begin{align*}
            H^q(X,\ogr\otimes\cal{O}_X(E\otimes\det E))=0.
        \end{align*}
    \end{itemize}
    Furthermore, if $X$ is not necessarily K\"{a}hler, then we have only the first cohomology vanishing $H^1(X,\ogr\otimes\cal{O}_X(E))=0$ and $H^1(X,\ogr\otimes\cal{O}_X(E\otimes\det E))=0$, respectively.
\end{corollary}

Here again, uniform Griffiths/Nakano positivity near $\xs$ is required, since, as in the previous case, degeneration of positivity along $\xs$ makes it unclear whether the integral of $\lara{B^{-1}_{\theta,\chi\circ\varPsi,\omega}f}{f}_{h,\omega}\,e^{-\chi\circ\varPsi}$ over $X$ remains finite.
The smooth $(1,1)$-from $\theta$ on $X$ is positive on $\reg$, but degenerates along $\xs$.


\vspace*{5mm}
\noindent
{\bf Acknowledgement.} 
The author is supported by Grant-in-Aid for Early-Career Scientists $\sharp$26K16989 from the Japan Society for the Promotion of Science (JSPS).



\end{document}